\documentclass[amssymb,11pt]{article}
\usepackage[colorlinks,linkcolor=blue,citecolor=cyan,urlcolor=black,hyperindex,CJKbookmarks,dvipdfm]{hyperref}
\usepackage[dvipdfm]{graphicx}
\usepackage{amsthm}
\usepackage{amsmath}
\usepackage{amssymb}
\usepackage{bm}
\usepackage{amsfonts}
\usepackage[subnum]{cases}
\usepackage{indentfirst}
\usepackage{booktabs}
\date{}
\newtheorem{definition}{Definition}[section]

\newtheorem{algorithm}{Algorithm}[section]
\newtheorem{theorem}{Theorem}[section]

\newtheorem{remark}{Remark}[section]
\newtheorem{example}{Example}[section]

\renewcommand{\thefootnote}{\fnsymbol{footnote}}

\begin{document}
\title{\Large Image segmentation based on the hybrid total variation model and the K-means clustering strategy}
\begin{upshape}
\author{Baoli Shi$^\dag$, Zhi-Feng Pang$^\dag$ and Jing Xu$^\S$}
\end{upshape} \maketitle
\begin{center}
\indent\indent \textsl{\scriptsize \hspace{5pt}$^\dag$College of
Mathematics  and Statistic, Henan University, kaifeng, 475004,
China.\\
$^\S$School of Statistic and Mathematics Zhejiang Gongshang University, Hangzhou, 310012, China
 \indent\indent}
\end{center}
\renewcommand{\thefootnote}{\fnsymbol{footnote}}
\footnotetext{

\ $^{\dagger}$This work is supported by NSF of China (Nos.U1304610,11001239,11101365, 11401170,11426137), Foundation of Henan Educational Committee of China (Nos.14A110018, 14B110019) and the National Basic Research Program of China
(973 Program)(No.2015CB856003..


\scriptsize \textsl{ E-mail address:baolishi1983@163.com; zhifengpang@163.com; jingxu@amss.ac.cn}\indent\indent\small }
\textbf{Abstract }The performance of image segmentation highly relies on the original inputting image. When the image is contaminated by some noises or blurs, we can not obtain the efficient segmentation result by using direct segmentation methods. In order to efficiently segment the contaminated image, this paper proposes a two step method based on the hybrid total variation model with a box constraint and the K-means clustering method. In the first step, the hybrid model is based on the weighted convex combination between the total variation functional and the high-order total variation  as the regularization term to obtain the original clustering data. In order to deal with non-smooth regularization term, we solve  this model by employing  the alternating split Bregman method.  Then, in the second step, the segmentation can be obtained by thresholding this clustering data into different phases, where the thresholds can be given by using the K-means clustering method. Numerical comparisons show that our proposed model can provide more efficient segmentation results dealing with the noise image and blurring image.
\begin{flushleft}
\end{flushleft}

\setcounter{equation}{0}
\section{Introduction}
Image segmentation has become increasingly important in the last decade, due to a fast expanding field of applications in image analysis and computer vision. It aims to separate objects of interest from each others or from the backward or   find boundaries of such objects.  Until now, a wide variety of techniques
including variational partial differential equation (PDE) image segmentation methods \cite{1,2,3,4,5}, histogram analysis, region growing and edge detection \cite{6,7}  have been proposed to solve the image segmentation problem.

The variational segmentation methods are characterized by deriving an energy functional from some a priori
mathematical model and minimizing this energy functional over all possible partitions. Among them, the region-based models \cite{4,5} have been proposed and applied to the image segmentation field by incorporating  region information so that image within each segmented region has a uniform characteristic such as  intensity
and texture.  The most well-known region-based model is the Mumford and Shah (MS) model \cite{3}, in which an image is decomposed into a set of regions within the bounded open set and these regions are separated by smooth edges. The difficulty in studying the MS model is that it involves two unknowns: the intensity function and the set of edges. So the MS model is hard to implement in practice since the discretization of the unknown set of edges is a very complex task. In order to effectively solve the MS model,  Chan and Vese (CV) \cite{5} proposed an easily handle model by assuming that the segmentation result was a piecewise constant image with two different constant values. The CV model has achieved good performance in image segmentation task because of its ability to obtain a larger convergence range and handle topological changes naturally \cite{8}. However, more work needs to be done on effective representation of regions and their boundaries for multi-phase segmentation. Vese and Chan extended the  work in \cite{9} to utilize multi-phase level set functionals to represent multiple regions. Similar to the two-phase case, the model is non-convex and thus the global minimization can not be guaranteed.  In order to overcome these drawbacks, convex relaxation methods \cite{4,10,11} and graph cut \cite{335,14,15,336,16,42} method were proposed in this field.

During some phases of  obtaining a  real image, we can only get a contaminated image due to the interference of some random noises and blurs. This interference leads not to obtain an expected segmentation results  by using classical methods such as the MS model and  clustering methods. So it is very important  to suppress these contaminated information before segmenting the image  \cite{16,18,19,20}. Recently a two step method based on the MS model was proposed in \cite{26}. The first step is through restoring the contaminated image to obtain a smooth solution by using the modified MS model under the space $W^{1,2}(\Omega)$ for the image domain $\Omega\subset R^2$. Once the restored image is obtained,  the second stage is to threshold  it into different phases by using the K-means clustering method.  However, the image space is not continuous in the region of edges, so  the space $W^{1,2}(\Omega)$ used in \cite{26} is not suitable. Moreover, the image pixel value represents physical quantity such as photon count or energy, which is often non-negative and has a upper bound or saturated value due to the finite number of bits being used for image representation. An intuitive approach is to solve the unconstrained problem and then project the restored values in  the original chosen range. However, this approach results in the presence of spurious ripples in the restored image. Chopping off these projection values may also introduce patches of black color which could be visually unpleasant. Therefore, imposing a priori constraint is a natural requirement for the image restoration problem. In this paper, by generalizing above two-step strategy to a more suitable space,  we propose a two-step method for the image segmentation problem. The motivation of our work is to present a hybrid total-variation model in the first step by combining edge based method with region based method to improve the efficiency of traditional models in \cite{12,27,35}. Furthermore, we constrain  pixel  values in this step to lie in a certain dynamic range in order to improve the preconditioning results. In addition, the proposed model is solved by an algorithm based on the alternating split Bregman method, which can guarantee  the implementation efficiency with respect to the resultant image quality based on the convergence of this method. Once the solution is obtained, we set it as the inputting data in the second step and use the K-means clustering method to threshold it into expected phases. Then we can obtain the segmentation by setting related regions to be different constants in these phases. The numerical experiments show that our proposed model and method are very effective for segmenting the degraded image.

The rest of the paper is organized as follows. In Second 2, we review some works in the image segmentation problem under the energy minimization framework. In section 3, we propose the hybrid total variation model with a constraint and then employ the  split Bregman method to solve it as the first step. In the second step, we give the detail of how to  threshold the clustering data generated by  the K-means clustering method into different phases. In  Section 4, we present some experimental results of our proposed strategy and compare their performance with the reviewed  models in Section 2. Some conclusions are drawn in Section 5.

\section{Related works}
Before we go into the details of our proposed hybrid total-variation segmentation models, we first review some  works that are closely related to this paper.
\subsection{Mumford-Shah (MS) model}
The MS model is an established image segmentation model with a wide range of applications. An energy functional  to be minimized is to find the set of discontinuities $\Gamma$ or the edge set, and a piecewise smooth approximation $u$ of a given image intensity function $f:\Omega\subset  R^2\rightarrow R$. Formally, the MS model can be written as
\begin{eqnarray}\label{21}
\min_{u,\Gamma}\left\{\frac{1}{2}\|f-u\|^2_{L^2(\Omega)}+\lambda\int_{\Omega\setminus\Gamma}|\nabla u|^2\mathrm{dx}+\gamma Length(\Gamma)\right\},
\end{eqnarray}
where $\lambda$ and $\gamma$ are positive constants.  However, the problem is hard in this form, as the two variables $u$ and $\Gamma$ of the optimization are of very difficult natures; the edge set $\Gamma$ may be quite general and may exhibit singularities. As alternative solutions to this problem, many
works have been done to simplify or modify the general MS  model (\ref{21}).
\subsection{The piecewise constant (PC) model}
The piecewise constant Mumford-Shah model is a simpler model obtained by assuming that image approximation
is constant inside each connected component (i.e. $u=c_i$, where $c_i$ is a constant). At this point, the MS model can be simplistically rewritten as
\begin{eqnarray}\label{22}
\min_{c_i,\Gamma}\sum_{i=1}^{\mathcal{M}}\int_{\Omega_i}|f-c_i|^2\mathrm{d}x+\gamma Length(\Gamma)
\end{eqnarray}
such that $\displaystyle\bigcup_{i=1}^{\mathcal{M}}\Omega_i=\Omega\setminus\Gamma$ and $\Omega_i\cap\Omega_j=\phi$ if $i\neq j$,
where $\mathcal{M}$ is the number of connected components.

To minimize the  model (\ref{22}), the zero level set  $\Gamma=:\{x\in\Omega|\phi(x):=0\}$  is  used to replace the unknown curve $\Gamma_i$,  where $\displaystyle\bigcup_{i=1}^{\mathcal{M}}\Gamma_i=\Gamma$.
For the case $\mathcal{M}=2$, the model (\ref{22}) is  the classical CV model \cite{5}, where Chan and Vese proposed to alternately obtain the curve $\Gamma$ and the constant values $c_1$ and $c_2$  by solving the following variational formulation
\begin{eqnarray}\label{23}
\left\{
\begin{array}{ll}
c_1(\phi)=\frac{\int_{\Omega}fH(\phi)\mathrm{dx}}{\int_{\Omega}  H(\phi)\mathrm{d}x},\\\vspace{3pt}
c_2(\phi)=\frac{\int_{\Omega}f(1-H(\phi))\mathrm{d}x}{\int_{\Omega}(1-H(\phi))\mathrm{d}x},\\\vspace{3pt}
\frac{\partial\phi}{\partial t}=\delta(\phi)\left[\gamma\mathrm{div}\frac{\nabla \phi}{|\nabla\phi|}-(f-c_1)^2+(f-c_2)^2\right],
\end{array}
\right.
\end{eqnarray}
where $H(\cdot)$ is the Heaviside function. Different from the CV model, the problem for ${\mathcal{M}}>2$ is more challenging. The main difficulty is to find effective representations of regions and their boundaries.  Several recent works are related to this problem. Vese and Chan \cite{9} proposed to use the four color theorem to provide segmentation for any number of objects in $\Omega$. One level set function can only separate two classes, hence $m_1=\mathrm{log}_2{\mathcal{M}}$ is the number of level sets necessary to segment ${\mathcal{M}}$ distinct classes.  A variant of the level set method in \cite{22}, the so-called piecewise constant level method (PCLSM), expresses the energy in terms of a labeling function. They proposed to use binary level set functions which took value $1$ and $-1$ instead of using the classical continuous level set function as did in \cite{9,5}. Unfortunately, the main disadvantage of both level set methods is their probable of getting stuck in a possibly inferior local minima because of their nonconvex formulations.

\subsection{The Potts model}

In the work of \cite{4},  Chan et al obtained  global optimal solution of the CV model when the two piecewise constant values are first known.  The key idea is to relax the characteristic function  in such a way that the characteristic function minimizers can be obtained from the minimizers of the relaxed problem by a simple thresholding procedure. This allows the non-convex Chan-Vese problem to be globally solved  by using standard convex minimization methods. Following this work and the assumption that we know the values of every piecewise constant, the multiphase segmentation problem (\ref{21}) can be written as
\begin{eqnarray}\label{24}
\min_{\Gamma}\sum_{i=1}^{\mathcal{M}}\int_{\Omega_i}|f-c_i|^2\mathrm{d}x+\gamma Length(\Gamma).
\end{eqnarray}
Actually, in this case, each pixel is assigned one of region labels (i.e. the known piecewise constants). Assume that a labeling $l:\Omega\rightarrow\{c_1,c_2,\cdots,c_{\mathcal{M}}\}$ indicates the $m$ regions each pixel belongs to $\Omega_i=\{l(x)=c_i\}$,  we introduce the indicator function $u:=(u_1,u_2,\cdots,u_{\mathcal{M}})\in BV(\Omega,\{0,1\})^{\mathcal{M}}$
\begin{eqnarray*}
u_i(x)=
\left\{
\begin{array}{ll}
1\hspace{5pt}\mathrm{if}\hspace{5pt}l(x)=c_i
\\0\hspace{5pt}\mathrm{otherwise}
\end{array}
\right.
\end{eqnarray*}
for $i=1,2,\cdots,\mathcal{M}$, where $BV$ refers to the bounded variation space,  then the model (\ref{24}) can be reduced to the formulation of the Potts model
\begin{eqnarray}\label{25}
\displaystyle\min_{u\in\mathcal{C}_1}\sum_{i=1}^{\mathcal{M}}\left\{\int_{\Omega}|f-c_i|^2u_i\mathrm{d}x+
\gamma\int_{\Omega}|\nabla u_i|\mathrm{d}x\right\},
\end{eqnarray}
where $\mathcal{C}_1:=\left\{u\in\{0,1\}^{\mathcal{M}}\mid\displaystyle\sum_{i=1}^{\mathcal{M}}u_i=1\right\}$.

The above Potts model is a nonconvex  problem due to the binary constraints. A major class of methods is based on the convex relaxation of the admissible set by allowing for the labeling functions
to take intermediate values from the unit simplex. That is to say, the problem (\ref{25}) is transformed into
\begin{eqnarray}\label{26}
\displaystyle\min_{u\in\mathcal{C}}\sum_{i=1}^{\mathcal{M}}\left\{\int_{\Omega}|f-c_i|^2u_i\mathrm{d}x+
\gamma\int_{\Omega}|\nabla u_i|\mathrm{d}x\right\},
\end{eqnarray}
where $\mathcal{C}:=\left\{u\in[0,1]^{\mathcal{M}}\mid\displaystyle\sum_{i=1}^{\mathcal{M}}u_i=1\right\}$.
Then we can solve this relaxed version by using some operator splitting methods \cite{23,24,25}. If the minimizer of the problem (\ref{26}) happens to be binary everywhere, it is also a global minimizer of the original problem (\ref{25}). However, unlike the two label problem \cite{5}, if the computed minimizer of the problem (\ref{26}) is not binary, there is no thresholding scheme which can keeps a binary global minimizer of (\ref{25}). Even if such a binary minimizer exists, the problem (\ref{26}) may result in nonbinary solutions due to non-uniqueness. Then a approximated strategy is to use the indicator function of the largest component $u_i$ as an approximate
binary solution \cite{23,24,25}.  Graph-based optimization techniques have been used as components in optimization
methods for functionals formulated in the continuous space. Boykov and Cremers \cite{335} proposed to use a max-flow/min-cut step to assist in level set updates. Grady \cite{336} employed a max-flow/min-cut operation
as a component of their piecewise constant Mumford-Shah complete computations.  More recently,  Bae etc. proposed a smooth dual model of the Potts model in \cite{13}. Pock etc. \cite{19} developed a tight convex relaxation framework for Potts model. Yuan etc \cite{130,131} have designed a max-flow approach to the Potts model.

\subsection{The convex variant of the MS model}
Based on the observation about binary images: a binary image can be recovered quite well from its smoothed
version by thresholding with a proper threshold, Cai et al \cite{26} recently proposed to use two stages for the image segmentation problem. In their scheme, the first stage is to find a smooth image $g$ that can facilitate the segmentation, and then the second stage is to threshold $g$ to reveal different segmentation features by using one of chosen thresholds based on the K-means clustering method.  Specially, this course can be formulated as
\begin{itemize}
\item[$\bullet$]The first step is to compute $g$ by
\begin{eqnarray}\label{27}
\min_g\|f-\mathcal{A}g\|^2_{L^2(\Omega)}+\lambda\int_{\Omega}|\nabla g|^2\mathrm{d}x+\gamma\int_{\Omega}|\nabla g|\mathrm{d}x,
\end{eqnarray}
where $f$ denotes the degraded image modeled by $f=\mathcal{A}g+\eta$ and $\mathcal{A}$ is a suitable linear operator. Here $\eta$ denotes some noises.
\item[$\bullet$]The second step is to obtain the segmentation by segmenting $g$ using properly chosen thresholds, which is generated by the K-means clustering method.
\end{itemize}

Actually, they used the fact that $\int_{\Gamma}|\nabla g|^2\mathrm{d}x=0$ when $g\in W^{1,2}(\Omega)$,  then the problem (\ref{27}) is equivalent to the original MS model (\ref{21}). Note that here $W^{a,b}(\Omega)$ denotes the Sobolev space defined by $W^{a,b}(\Omega):=\Big\{u:u\in L^b(\Omega)$, $ D^{\alpha}u\in L^b (\Omega)$ $0\leq|\alpha|\leq a\Big\}$,  where $D^{\alpha}:=D_1^{\alpha_1}D_2^{\alpha_2}\cdots D_r^{\alpha_r}$  with $\displaystyle\sum_{i=1}^r\alpha_i=|a|$ and $D_j=\frac{\partial}{\partial x_j}$.
The problem (\ref{27}) is a non-smooth optimization problem, so they employed the splitting Bregman method to solve it.

\section{The segmentation model based on  constrained hybrid  total variation model}
In this section, we will introduce our proposed model following from the problem (\ref{27}) and bring out its salient merits and numerical method. In the problem (\ref{27}), there include a fitting term and two regularization terms. Its nature is the image restoration problem, so the segmentation result extremely desponds on the restored image. However, the global parameters $\lambda$ and $\gamma$ can not depict the local properties while the image owns complex structures. Furthermore, image values which represent physical quantities such as photon counts
or energies are often non-negative and also have the upper bound or saturated value, so it is a natural requirement to impose a priori constraint on the proposed model.
\subsection{The proposed model of the first step}
The total variation type models have been extensively used in image restoration such as the ROF model \cite{27} and the higher-order model \cite{1001,28,29,1002}. However, these two classes of models respectively keep image edges or only keep image smoothing regions. On the other hand, image value represents physical quantities such as photon counts or energies, which is often non-negative and also has the upper bound or saturated value due to finite number of bits being used for image representation.  This is particularly true in applications such as astronomical image \cite{200, 199}. Therefore, imposing a priori constraint is a natural requirement for the deblurring problem.  To effectively overcome these drawbacks, some hybrid models based on the total variation functional have been noticed. Following the problem (\ref{27}), we propose the following variational problem
\begin{eqnarray}\label{31}
\min_{g\in[0,\iota]}\|f-\mathcal{A}g\|^2_{L^2(\Omega)}+\lambda\int_{\Omega}(1-\omega(x))|\nabla^2g|\mathrm{d}x
+\gamma\int_{\Omega}\omega(x)|\nabla g|\mathrm{d}x,
\end{eqnarray}
where $\omega(x)\in(0,1)$ is a weighted function related to the region structures of the inputting  image. Obviously, we replace the second term in the problem (\ref{27}) by $\int_{\Omega}|\nabla^2g|\mathrm{d}x$, which actually requires $u\in  W^{2,1}(\Omega)$,  then we can obtain more efficient restored image. However, this space have not a reflexive solution of the minimization problem by the direct calculus of the problem (\ref{31}). Actually, we can extend it to a larger space as being done in \cite{35,36}. Furthermore, we also add the weighted function $\omega(x)$ with formally convex combination of $\int_{\Omega}|\nabla^2g|\mathrm{d}x$ and $\int_{\Omega}|\nabla g|\mathrm{d}x$. This paper chooses the weighed function  to be  the edge indicator function. Thus  edges can be found and retained during restoration process in the first stage of our proposed method.


\begin{definition}
$BV^\kappa{(\Omega)}$ is a
subspace of functions $u\in
 {L^{1}}(\Omega)$ such that the following quantity is satisfied:
\begin{eqnarray*}
\int_\Omega{\left|\verb"D"^{\kappa} u\right|} =\sup\left\{\int_\Omega{ u \mathrm{div}^\kappa p}\mathrm{d}x\mid
p\in C_c^\kappa\left(\Omega, R^{2\times\kappa}\right),\| p\|_{L^\infty}\leq 1 \right\}<\infty,
\end{eqnarray*}
where
\begin{displaymath}
\|p(x)\|_{\infty}=\left\|\sqrt{\displaystyle\sum_{j=1}^2\sum_{i=1}^\kappa
\left(p^{i,j}\right)^2}\right\|_{L^\infty}
\end{displaymath}
for $\kappa=1,2$. Here $C_c^\kappa(\Omega)$ denotes the space of continuously
differentiable vector-valued functions on $\Omega$ with compact support.
\end{definition}
It is obviously  that $BV(\Omega)$  and $BV^2(\Omega)$ are a Banach space respectively equipped with  norms $\|u\|_{BV(\Omega)}=\|u\|_{L^1(\Omega)}+|\verb"D"u|(\Omega)$ and $\|u\|_{BV^2(\Omega)}=\|u\|_{BV(\Omega)}+|\verb"D"^2u|(\Omega)$ \cite{35,36}.
\begin{theorem}
Let $\Omega$ be a bounded connected open subset of $R^2$ with the Lipschitz boundary. Assume that $f\in BV(\Omega)\cap BV^2(\Omega)\cap L^{\infty}(\Omega)$. Then the problem
\begin{eqnarray}\label{301}
\min_{g\in[0,\iota]}\|f-\mathcal{A}g\|^2_{L^2(\Omega)}+\lambda\int_{\Omega}(1-\omega(x))|\verb"D"^2g|
+\gamma\int_{\Omega}\omega(x)|\verb"D" g|,
\end{eqnarray}
has a minimizer $g^*\in BV(\Omega)\cap BV^2(\Omega)\cap L^{\infty}(\Omega)$.
\end{theorem}
\begin{proof}
By introducing   an indicator function
\begin{eqnarray}\label{32}
\delta_{\mathcal{D}}(g)=
\left\{
\begin{array}{ll}
0,\hspace{8pt}\mathrm{if}\hspace{3pt}g\in\mathcal{D},\\
\infty,\hspace{3pt}\mathrm{if}\hspace{3pt}g\not\in\mathcal{D}
\end{array}
\right.
\end{eqnarray}
with $\mathcal{D}:=[0,\iota]$, then we can write the problem (\ref{301}) as
\begin{eqnarray}\label{3110}
\min_g\|f-\mathcal{A}g\|^2_{L^2(\Omega)}+\lambda\int_{\Omega}(1-\omega(x))|\verb"D"^2g|
+\gamma\int_{\Omega}\omega(x)|\verb"D" g|+\delta_{\mathcal{D}}(g).
\end{eqnarray}
Then the problem (\ref{3110}) is the strictly convexity. So based on the work in \cite{35}, we can use the similar method to obtain above results.
\end{proof}

\subsection{The split Bregman method to solve the problem (\ref{31})}
In this subsection we work with the discretized version of the problem (\ref{36}) and  assume a periodic boundary condition  for $g$. By choosing the periodic boundary condition, the action of each of discrete differential operators can be regarded as a circular convolution of the image $g$ and allows the use of fast Fourier transform (FFT).
For a vector field $\bm{\nu}:R^{s\times t}\rightarrow R^\tau$, we respectively define the discretization norm $|\cdot|_1$ and $\|\cdot\|_2$ corresponding to the  norms $\ell^1$ and $\ell^2$ with $\bm{\nu}=(\nu^1,\nu^2,\cdots,\nu^\tau)$ as follows
\begin{eqnarray*}
|\bm{\nu}|_1&=&\sum_{i=1}^s\left(\sum_{j=1}^t(\nu_{i,j})^2\right)^{\frac{1}{2}}\hspace{3pt}\mbox{and}\hspace{3pt}
\|\bm{\nu}\|_2=\left(\sum_{i=1}^s\sum_{j=1}^t\left(\nu_{i,j}\right)^2\right)^{\frac{1}{2}}.
\end{eqnarray*}
In the following,  we define the discrete gradient $\nabla u_{i,j}=(D^+_xu_{i,j},D^+_yu_{i,j})$ as a forward difference operator
\begin{eqnarray*}
D^+_xu_{i,j}=\left\{
\begin{array}{ll}
u_{i+1,j}-u_{i,j}\hspace{3pt}&\mbox{if}\hspace{3pt}1\leq i< m, 1\leq j\leq n\\
u_{1,j}-u_{i,j}\hspace{3pt}&\mbox{if}\hspace{3pt}i=m, 1\leq j\leq n,
\end{array}
\right.\\
D^+_yu_{i,j}=\left\{
\begin{array}{ll}
u_{i,j+1}-u_{i,j}\hspace{3pt}&\mbox{if}\hspace{3pt}1\leq i\leq  m, 1\leq j<n\\
u_{i,1}-u_{i,j}\hspace{3pt}&\mbox{if}\hspace{3pt}1\leq i\leq  m, j=n.
\end{array}
\right.
\end{eqnarray*}
Similarly, we can also define the backward difference operator as
\begin{eqnarray*}
D^-_xu_{i,j}=\left\{
\begin{array}{ll}
u_{i,j}-u_{i-1,j}\hspace{3pt}&\mbox{if}\hspace{3pt}1<i\leq m, 1\leq j\leq n\\
u_{i,j}-u_{m,j}\hspace{3pt}&\mbox{if}\hspace{3pt}i=1, 1\leq j\leq n,
\end{array}
\right.\\
D^-_yu_{i,j}=\left\{
\begin{array}{ll}
u_{i,j}-u_{i,j-1}\hspace{3pt}&\mbox{if}\hspace{3pt}1\leq i\leq  m, 1<j\leq n\\
u_{i,j}-u_{i,n}\hspace{3pt}&\mbox{if}\hspace{3pt}1\leq i\leq  m, j=1.
\end{array}
\right.
\end{eqnarray*}
Then we can obtain
\begin{eqnarray*} \nabla^2u_{i,j}=\left(D^-_x(D^+_xu_{i,j}),D^-_x(D^+_yu_{i,j}),D^+_y(D^-_xu_{i,j}),D^+_y(D^-_yu_{i,j})\right).
\end{eqnarray*}
Using the divergence theorem
\begin{eqnarray*}
-\mathrm{div}\bm{\rho}\cdot u=\bm{\rho}\cdot\nabla u\hspace{5pt}\mbox{and}\hspace{5pt}\mathrm{div}^2\textbf{p}\cdot u=\textbf{p}\cdot\nabla^2u
\end{eqnarray*}
for $\forall u\in R^{m\times n}$, $\bm{\rho}\in  R^{m\times n}\times R^{m\times n}$, $\textbf{p}\in R^{m\times n} \times R^{m\times n} \times R^{m\times n} \times R^{m\times n}$, where
$\mathrm{div}$ and $\mathrm{div}^2$ respectively denote the adjoint operator of $\nabla$ and $\nabla^2$, then we have
\begin{eqnarray*}
\mathrm{div}\bm{\rho}_{i,j}&=&D^-_x\bm{\rho}_{i,j}+D^-_y\bm{\rho}_{i,j},\\
\mathrm{div}^2\textbf{p}_{i,j}&=&D^-_x(D^+_x\textbf{p}_{i,j})+D^-_x(D^+_y\textbf{p}_{i,j})
+D^+_x(D^-_x\textbf{p}_{i,j})+D^-_y(D^+_y\textbf{p}_{i,j}).
\end{eqnarray*}

In order to discretize  (\ref{31}),   by using the indicator function (\ref{32}) and the corresponding discrete operators and norms,  we can obtain the discretization form as
\begin{eqnarray}\label{33}
\min_{g}\|f-\mathcal{A}g\|^2_2+\lambda(1-\omega(x))|\nabla^2g|_1
+\gamma\omega(x)|\nabla g|_1+\delta_{\mathcal{D}}(g).
\end{eqnarray}
The problem (\ref{33}) is a nonsmoothing optimization problem. Variable splitting methods such as the  alternating direction method of multipliers  (ADMM) \cite{30,31,32} and the operator splitting methods \cite{33,34} have been recently used in solving this class of problems. The key of this class of methods is to transform the original problem into some subproblems so that we can easily solve these subproblems by  some traditionally numerical methods. By introducing some auxiliary variables $\textbf{q}=\nabla^2g$, $\textbf{v}=\nabla g$  and $z=g$, we convert the problem (\ref{33}) into the following form
\begin{eqnarray}\label{34}
\left\{
\begin{array}{ll}
\displaystyle\min_{g,\textbf{q},\textbf{v},z}\|f-\mathcal{A}g\|^2_2+\lambda(1-\omega(x))|\textbf{q}|_1+
\gamma\omega(x)|\textbf{v}|_1+\delta_{\mathcal{D}}(z),
\\\mathrm{s.t.}\left\{
\begin{array}{ll}
\textbf{q}=\nabla^2g,\\
\textbf{v}=\nabla g,\\
z=g.
\end{array}
\right.
\end{array}
\right.
\end{eqnarray}
Then the problem (\ref{34}) can be solved under the framework of the alternating split Bregman  method as
\begin{numcases}{\label{35}}
(g^{k+1},\textbf{q}^{k+1},\textbf{v}^{k+1},z^{k+1}):=\displaystyle\mathop{\mathrm{argmin}}_{g,\textbf{q},\textbf{v},z}
\left\|f-\mathcal{A}g\right\|^2_2+\lambda(1-\omega(x))|\textbf{q}|_1+\gamma\omega(x)|\textbf{v}|_1\nonumber\\
\hspace{10pt}+\delta_{\mathcal{D}}(z)+\frac{\mu_1}{2}\left\|\textbf{b}^k+\nabla^2g-
\textbf{q}\right\|^2_2+\frac{\mu_2}{2}\left\|\textbf{c}^k+\nabla g-\textbf{v}\right\|^2_2+\frac{\mu_3}{2}\left\|d^k+g-z\right\|^2_2,\label{35a}\\
\hspace{18pt}\textbf{b}^{k+1}:=\textbf{b}^k+\nabla^2g^{k+1}-\textbf{q}^{k+1},\\
\hspace{18pt}\textbf{c}^{k+1}:=\textbf{c}^k+\nabla g^{k+1}-\textbf{v}^{k+1},\\
\hspace{18pt}d^{k+1}:=d^k+g^{k+1}-z^{k+1}.
\end{numcases}
The minimization problem (\ref{35a}) yielding $(g^{k+1},\textbf{q}^{k+1},$$\textbf{v}^{k+1},z^{k+1})$ is not trivial since it includes two nonseparable terms and two nonsmooth terms. A natural approach is to alternate belong minimizing with respect to $g, \textbf{q}, \textbf{v}, z$ while keeping others fixed. Then we have the following strategy to solve the problem (\ref{33})
\begin{numcases}{\label{36}}
g^{k+1}:=\displaystyle\mathop{\mathrm{argmin}}_{g}\left\|f-\mathcal{A}g\right\|^2_2
+\frac{\mu_1}{2}\left\|\textbf{b}^k+\nabla^2g-\textbf{q}^k\right\|^2_2\nonumber\\
\hspace{60pt}+\frac{\mu_2}{2}\left\|\textbf{c}^k+\nabla g-\textbf{v}^k\right\|^2_2+\frac{\mu_3}{2}\left\|d^k+ g-z^k\right\|^2_2,\label{36a}\\
\textbf{q}^{k+1}:=\displaystyle\mathop{\mathrm{argmin}}_{\textbf{q}}\lambda(1
-\omega(x))|\textbf{q}|_1+\frac{\mu_1}{2}\left\|\textbf{b}^k+
\nabla^2g^{k+1}-\textbf{q}\right\|^2_2,\label{36b}\\
\textbf{v}^{k+1}:=\displaystyle\mathop{\mathrm{argmin}}_{\textbf{v}}
\gamma\omega(x)|\textbf{v}|_1+\frac{\mu_2}{2}\left\|\textbf{c}^k+\nabla g^{k+1}-\textbf{v}\right\|^2_2,\label{36c}\\
z^{k+1}:=\displaystyle\mathop{\mathrm{argmin}}_{z}\delta_{\mathcal{D}}(z)+\frac{\mu_3}{2}
\left\|d^k+g^{k+1}-z\right\|^2_2,\label{36d}\\
\textbf{b}^{k+1}:=\textbf{b}^k+\nabla^2g^{k+1}-\textbf{q}^{k+1},\label{36e}\\
\textbf{c}^{k+1}:=\textbf{c}^k+\nabla g^{k+1}-\textbf{v}^{k+1},\label{36f}\\
d^{k+1}:=d^k+g^{k+1}-z^{k+1}.\label{36g}
\end{numcases}
For the iteration scheme (\ref{36}), we can derive the following convergence result. Though the proof is similar to the course used in \cite{43} and there includes three constrained conditions in solved problem (\ref{35}), the data fitting term in problem (\ref{34}) can make the proof more simple.
\begin{theorem}\label{th33}
Assume that   $\mathrm{Ker}(\mathcal{A})=\{0\}$
and the point $\left(g^*,\textbf{q}^*,\textbf{v}^*,z^*\right)$ is the solution of the problem (\ref{35}). Then the sequence $(g^k,\textbf{q}^k,\textbf{v}^k,z^k)$ generated by the scheme (\ref{36}) converges to the solution of the problem (\ref{34}) for every choosing original point $\left(g^0,\textbf{q}^0,\textbf{v}^0,z^0\right)$.
\end{theorem}
\begin{proof}
The sequence $(g^{k+1},\textbf{q}^{k+1},\textbf{v}^{k+1}, z^{k+1},\textbf{b}^{k+1},\textbf{b}^{k+1},d^{k+1})$ generated by the subproblems (\ref{36a})-(\ref{36g}) satisfies the following the first order optimization condition as
\begin{numcases}{\label{14}}
0=\mathcal{A}^\star(\mathcal{A}g^{k+1}-f)+\mu_1\mathrm{div}^2(\nabla^2g^{k+1}+\textbf{b}^k-\textbf{q}^k)\nonumber\\
\hspace{120pt}-\mu_2\mathrm{div}(\nabla g^{k+1}+\textbf{c}^k-\textbf{v}^k)+\mu_3(g^{k+1}+d^k-z^k),\\
0=\lambda(1-\omega(x))\textbf{r}^{k+1}+{\mu_1}(\textbf{q}^{k+1}-\textbf{b}^k-\nabla^2g^{k+1}), \\
0=\gamma\omega(x)\textbf{s}^{k+1}+\mu_2(\textbf{v}^{k+1}-\textbf{c}^{k+1}-\nabla g^{k+1}),\\
0={\varpi}^{k+1}+\mu_3(z^{k+1}-d^k-g^{k+1}),\label{14d}\\
\textbf{b}^{k+1}=\textbf{b}^k+\nabla^2g^{k+1}-\textbf{q}^{k+1},\label{14e}\\
\textbf{c}^{k+1}=\textbf{c}^k+\nabla g^{k+1}-\textbf{v}^{k+1},\label{14f}\\
d^{k+1}=d^k+g^{k+1}-z^{k+1},\label{14g}
\end{numcases}
where $\mathcal{A}^\star$, $\mathrm{div}^2$ and $\mathrm{div}$ denote the adjoint operator of operators $\mathcal{A}$, $\nabla^2$ and $\nabla$,  $\textbf{r}^{k+1}\in\partial (|\textbf{q}^{k+1}|)$, $\textbf{s}^{k+1}\in\partial(|\textbf{v}^{k+1}|)$ and $\varpi^{k+1}\in\partial(\delta_{\mathcal{D}}(z^{k+1}))$. Here $\partial\mathcal{F}(x)$ denotes the subdifferential set at $x$.  By the assumption that the point $(g^*,q^*,v^*,z^*)$ is the solution of the problem (\ref{35}), we can obtain
\begin{numcases}{\label{15}}
0=\mathcal{A}^\star(\mathcal{A}g^*-f)+\mu_1\mathrm{div}^2(\nabla^2g^*+\textbf{b}^*-\textbf{q}^*)\nonumber\\
\hspace{120pt}-\mu_2\mathrm{div}(\nabla g^*+\textbf{c}^*-\textbf{v}^*)+\mu_3(g^*+d^*-z^*),\\
0=\lambda(1-\omega(x))\textbf{r}^*+{\mu_1}(\textbf{q}^*-\textbf{b}^*-\nabla^2g^*), \\
0=\gamma\omega(x)\textbf{s}^*+\mu_2(\textbf{v}^*-\textbf{c}^*-\nabla g^*),\\
0={\varpi}^*+\mu_3(z^*-d^k-g^*),\\
\textbf{b}^*=\textbf{b}^*+\nabla^2g^*-\textbf{q}^*,\\
\textbf{c}^*=\textbf{c}^*+\nabla g^*-\textbf{v}^*,\\
d^*=d^*+g^*-z^*,
\end{numcases}
where $\textbf{r}^*\in\partial (|\textbf{q}^*|)$, $\textbf{s}^*\in\partial(|\textbf{v}^*|)$ and $\varpi^*\in\partial(\delta_{\mathcal{D}}(z^*))$. Set $g_e^k=g^k-g^*$, $\textbf{q}^k_e=\textbf{q}^k-\textbf{q}^*$, $\textbf{v}_e^k=\textbf{v}^k-\textbf{v}^*$, $\textbf{r}_e^k=\textbf{r}^k-\textbf{r}^*$,
$\textbf{s}_e^k=\textbf{s}^k-\textbf{s}^*$, $\varpi_e^k=\varpi^k-\varpi^*$,
$\textbf{b}_e^k=\textbf{b}^k-\textbf{b}^*$, $\textbf{c}_e^k=\textbf{c}^k-\textbf{c}^*$, $d_e^k=d^k-d^*$, subtracting (\ref{15}) from (\ref{14}) and then respectively take  the inter product $\langle\cdot,\cdot\rangle$ of the left- and right-hand sides by $g_e^{k+1}$, $\textbf{q}_e^{k+1}$  $\textbf{v}_e^{k+1}$, $\textbf{z}_e^{k+1}$, $\textbf{b}_e^k$, $\textbf{c}_e^k$, $d_e^k$, we can obtain that
\begin{numcases}{\label{16}}
0=\|\mathcal{A}g_e^{k+1}\|_2^2+\mu_1\|\nabla^2g_e^{k+1}\|_2^2+\mu_1\langle \textbf{b}_e^k-\textbf{q}_e^k,\nabla^2g_e^{k+1}\rangle+\mu_2\|\nabla g_e^{k+1}\|_2^2\nonumber\\\hspace{80pt}+\mu_2\langle \textbf{c}_e^k-\textbf{v}_e^k,\nabla g_e^{k+1}\rangle+\mu_3\|g_e^{k+1}\|_2^2+\mu_3\langle d_e^k-z_e^k,g_e^{k+1}\rangle, \label{16a}\\
0=\lambda\langle(1-\omega(x))\textbf{r}^{k+1}_e,\textbf{q}_e^{k+1}\rangle
+\mu_1\|\textbf{q}_e^{k+1}\|_2^2-\mu_1\langle\textbf{b}^k_e+\nabla^2g^{k+1}_e,\textbf{q}_e^{k+1}),\label{16b}\\
0=\gamma\langle\omega(x)\textbf{s}^{k+1}_e,\textbf{v}_e^{k+1}\rangle+\mu_2\|\textbf{v}_e^{k+1}\|-\mu_2\langle\textbf{c}^k_e+
\nabla g_e^{k+1},\textbf{v}_e^{k+1}),\label{16c}\\
0=\langle\varpi^{k+1}_e,z_e^{k+1}\rangle+\mu_3\|z^{k+1}_e\|_2^2-\mu_3\langle d_e^k+g^{k+1}_e,z^{k+1}_e\rangle,\label{16d}\\
\langle\textbf{b}_e^{k+1},\textbf{b}_e^k\rangle=\|\textbf{b}^k_e\|_2^2+\langle\nabla^2g^{k+1}_e-\textbf{q}^{k+1}_e,
\textbf{b}_e^k\rangle,\label{16e}\\
\langle\textbf{c}_e^{k+1},\textbf{c}_e^k\rangle=\|\textbf{c}^k_e\|^2_2+\langle\nabla g^{k+1}_e-\textbf{v}^{k+1}_e,\textbf{c}_e^k\rangle,\label{16f}\\
\langle d_e^{k+1},d_e^k\rangle=\|d^k_e\|_2^2+\langle g^{k+1}_e-z^{k+1}_e,d_e^k\rangle.\label{16g}
\end{numcases}
Since (\ref{16e})-(\ref{16g}) can be written as
\begin{numcases}{\label{17}}
\frac{\mu_1}{2}(\|\textbf{b}_e^k\|_2^2-\|\textbf{b}_e^{k+1}\|_2^2)=\mu_1\langle\textbf{b}_e^k,
\textbf{q}_e^{k+1}-\nabla^2g_e^{k+1}\rangle-\frac{\mu_1}{2}\|\nabla^2g_e^{k+1}-\textbf{q}_e^{k+1}\|_2^2\label{17a}\\
\frac{\mu_2}{2}(\|\textbf{c}_e^k\|_2^2-\|\textbf{c}_e^{k+1}\|_2^2)=\mu_2\langle\textbf{c}_e^k,\textbf{v}_e^{k+1}
-\nabla g_e^{k+1}\rangle-\frac{\mu_2}{2}\|\nabla g_e^{k+1}-\textbf{v}_e^{k+1}\|_2^2\label{17b}\\
\frac{\mu_3}{2}(\|d_e^k\|_2^2-\|d_e^{k+1}\|_2^2)=\mu_3\langle d_e^k, z_e^{k+1}-g_e^{k+1}
\rangle-\frac{\mu_3}{2}\|g_e^{k+1}-z_e^{k+1}\|_2^2\label{17c}
\end{numcases}
Summing (\ref{16a})-(\ref{16d}) with (\ref{17a})-(\ref{17c}), we can obtain
\begin{eqnarray*}
&&\frac{\mu_1}{2}(\|\textbf{b}_e^k\|_2^2-\|\textbf{b}_e^{k+1}\|_2^2)+
\frac{\mu_2}{2}(\|\textbf{c}_e^k\|_2^2-\|\textbf{c}_e^{k+1}\|_2^2)+
\frac{\mu_3}{2}(\|d_e^k\|_2^2-\|d_e^{k+1}\|_2^2)\nonumber\\
&&=\|\mathcal{A}g_e^{k+1}\|_2^2+\mu_1\|\nabla^2 g_e^{k+1}\|_2^2+\mu_2\|\nabla g_e^{k+1}\|_2^2+\mu_3\|g_e^{k+1}\|_2^2+\mu_1\|\textbf{q}_e^{k+1}\|_2^2\nonumber\\
&&+\mu_2\|\textbf{v}_e^{k+1}\|_2^2
+\mu_3\|z_e^{k+1}\|_2^2+\lambda\langle(1-\omega(x))\textbf{r}^{k+1}_e,\textbf{q}_e^{k+1}\rangle+
\gamma\langle\omega(x)\textbf{s}^{k+1}_e,\textbf{v}_e^{k+1}\rangle\nonumber\\
&&+\langle\varpi_e^{k+1},z_e^{k+1}\rangle-\mu_1\langle\nabla^2g_e^{k+1},\textbf{q}_e^k+\textbf{q}_e^{k+1}\rangle
-\mu_2\langle\nabla g_e^{k+1},\textbf{v}_e^k+\textbf{v}_e^{k+1}\rangle\nonumber
\\&&-\mu_3\langle g_e^{k+1},z_e^k+z_e^{k+1}\rangle-\frac{\mu_1}{2}\|\nabla^2g_e^{k+1}-\textbf{q}_e^{k+1}\|_2^2-\frac{\mu_2}{2}\|\nabla g_e^{k+1}-\textbf{v}_e^{k+1}\|_2^2\nonumber\\
&&-\frac{\mu_3}{2}\|g_e^{k+1}-z_e^{k+1}\|_2^2.
\end{eqnarray*}
By summing the above equation from $k=0$ to $k=\mathcal{K}$, we get
\begin{eqnarray*}
&&\frac{\mu_1}{2}(\|\textbf{b}_e^0\|_2^2-\|\textbf{b}_e^{\mathcal{K}+1}\|_2^2)
+\frac{\mu_2}{2}(\|\textbf{c}_e^0\|_2^2-\|\textbf{c}_e^{\mathcal{K}+1}\|_2^2)
+\frac{\mu_3}{2}(\|d_e^0\|_2^2-\|d_e^{\mathcal{K}+1}\|_2^2)\nonumber\\
&&=\sum_{k=0}^{\mathcal{K}}\Big[\|\mathcal{A}g_e^{k+1}\|_2^2+\lambda\langle(1-\omega(x))\textbf{r}^{k+1}_e,
\textbf{q}_e^{k+1}\rangle+
\gamma\langle\omega(x)\textbf{s}^{k+1}_e,\textbf{v}_e^{k+1}\rangle\nonumber
\\&&+\langle\varpi_e^{k+1},z_e^{k+1}\rangle+\frac{\mu_1}{2}\|\nabla^2g_e^{k+1}-\textbf{q}_e^{k+1}\|_2^2
+\frac{\mu_2}{2}\|\nabla g_e^{k+1}-\textbf{v}_e^{k+1}\|_2^2\nonumber\\
&&+\frac{\mu_3}{2}\|g_e^{k+1}-z_e^{k+1}\|_2^2\Big]+\frac{\mu_1}{2}\|\textbf{q}_e^{\mathcal{K}+1}\|_2^2
-\frac{\mu_1}{2}\|\textbf{q}_e^0\|_2^2+\frac{\mu_2}{2}\|\textbf{v}_e^{\mathcal{K}+1}\|_2^2
\nonumber\\&&-\frac{\mu_2}{2}\|\textbf{v}_e^0\|_2^2+\frac{\mu_3}{2}\|z_e^{\mathcal{K}+1}\|_2^2
-\frac{\mu_3}{2}\|z_e^0\|_2^2.
\end{eqnarray*}
Using the fact that the monotonicity of the subdifferential
\begin{eqnarray*}
\left\{
\begin{array}{ll}
\zeta_x\in\partial f(x)\Longleftrightarrow f(y)-f(x)\geq\langle\zeta_x,y-x\rangle, \hspace{5pt}x\in R^n\hspace{3pt}\mbox{for}\hspace{3pt}\forall\hspace{3pt}y\in R^n,\\
\zeta_y\in\partial f(y)\Longleftrightarrow f(x)-f(y)\geq\langle\zeta_y,x-y\rangle, \hspace{5pt}x\in R^n\hspace{3pt}\mbox{for}\hspace{3pt}\forall\hspace{3pt}y\in R^n,
\end{array}
\right.
\end{eqnarray*}
implies
\begin{eqnarray*}
\langle\zeta_x-\zeta_y,x-y\rangle\geq0,
\end{eqnarray*}
we get
\begin{eqnarray*}
&&\frac{\mu_1}{2}\|\textbf{b}_e^0\|_2^2+\frac{\mu_2}{2}\|\textbf{c}_e^0\|_2^2+\frac{\mu_3}{2}\|d_e^0\|_2^2+
\frac{\mu_1}{2}\|\textbf{q}_e^0\|_2^2+\frac{\mu_2}{2}\|\textbf{v}_e^0\|_2^2+\frac{\mu_3}{2}\|z_e^0\|_2^2\nonumber\\
&&\geq\sum_{k=0}^{\mathcal{K}}\Big[\|\mathcal{A}g_e^{k+1}\|_2^2+\frac{\mu_1}{2}\|\nabla^2g_e^{k+1}-\textbf{q}_e^{k+1}
\|_2^2+\frac{\mu_2}{2}\|\nabla g_e^{k+1}-\textbf{v}_e^{k+1}\|_2^2\nonumber\\&&+\frac{\mu_3}{2}\|g_e^{k+1}-z_e^{k+1}\|_2^2\Big].
\end{eqnarray*}
This leads to the following conclusions:
\begin{numcases}{\label{18}}
\lim_{k\rightarrow\infty}\|\mathcal{A}g_e^k\|_2^2=0, \label{18a}\\
\lim_{k\rightarrow\infty}\frac{\mu_1}{2}\|\nabla^2g_e^k-\textbf{q}_e^k\|_2^2=0,\label{18b}\\
\lim_{k\rightarrow\infty}\frac{\mu_1}{2}\|\nabla g_e^k-\textbf{v}_e^k\|_2^2=0,\label{18c}\\
\lim_{k\rightarrow\infty}\frac{\mu_1}{2}\| g_e^k-z_e^k\|_2^2=0.\label{18d}
\end{numcases}
By the assumption of $Ker(\mathcal{A})=\{0\}$,  the formula (\ref{18a}) implies
$\displaystyle\lim_{k\rightarrow\infty}g^k=g^*$. Then we also orderly obtain  $\displaystyle\lim_{k\rightarrow\infty}\textbf{q}^k=\textbf{q}^*$, $\displaystyle\lim_{k\rightarrow\infty}
\textbf{v}^k=\textbf{v}^*$, $\displaystyle\lim_{k\rightarrow\infty}z^k=z^*$.
\end{proof}

\begin{remark}
Since we used the periodic boundary condition for the differential operators  $\nabla$ and $\nabla^2$, it is easy to deduce that they are injective operators. So the assumption $Ker(\mathcal{A})=\{0\}$ in Theorem \ref{th33}	implies that $\mathrm{Ker}(\mathcal{A})=\{0\}\cap \mathrm{Ker}(\nabla)\cap \mathrm{Ker}\left(\nabla^2\right)={0}$. Then the objective function in the problem (\ref{33}) is  strong convexity, therefore the solution is unique. Furthermore, it also means that non-zeros constant images are not included in the null space of the operator $\mathcal{A}$, which is true for the image deconvolution problem in our numerical implementations.
\end{remark}
\begin{remark}
For the dual variables $(\textbf{b}^k,\textbf{c}^k,d^k)$ in the iteration schemes (\ref{36a})-(\ref{36g}), following from the assertions  of  $(\textbf{q}^k, \textbf{v}^k, z^k)\rightarrow (\textbf{q}^*, \textbf{v}^*, z^*)$ as $k\rightarrow \infty$ in Theorem \ref{th33} and  equations (\ref{18b})-(\ref{18d}), we obtain that $\nabla^2 g*=\textbf{q}^*$, $\nabla g*=\textbf{v}^*$ and $g^*=z^*$.  Then we can obtain  $(\textbf{b}^k,\textbf{c}^k,d^k)\rightarrow(\textbf{b}^*,\textbf{c}^*,d^*)$ by taking $k\rightarrow \infty $ in the side of equations (\ref{36e})-(\ref{36g}) or (\ref{14e})-(\ref{14g}).
\end{remark}

We now consider to solve the subproblem in the scheme (\ref{36}).
\begin{itemize}
\item[$\bullet$]\textbf{To} \textbf{solve} \textbf{the} \textbf{subproblem} \textbf{(\ref{36a})}:

The subproblem (\ref{36a}) is a smooth and convex optimization problem, so its Euler-Lagrange equation  satisfies
\begin{eqnarray*}
\left[\mathcal{A}^\star\mathcal{A}+\mu_1\mathrm{div}^2\nabla^2-\mu_2\mathrm{div}\nabla+\mu_3I\right]g^{k+1}=
\mathcal{A}^\star f
-\mu_1\mathrm{div}^2(\textbf{b}^k-\textbf{q}^k)\\\hspace{80pt}
+\mu_2\mathrm{div}(\textbf{c}^k-\textbf{v}^k)-\mu_3(d^k-z^k),
\end{eqnarray*}
where $\mathcal{A}^\star$ is the conjugate operator of $\mathcal{A}$.  For this equation, we can obtain the explicit solution by using the fast Fourier transform (FFT) while $\mathcal{A}$  has a special form such as the identity operator $I$ or the blurring operator:
\begin{eqnarray}\label{37}
g^{k+1}=\mathcal{F}^{-1}\left(\frac{M}{D}\right),
\end{eqnarray}
where $M:=\mathcal{F}(\mathcal{A}^\star)\mathcal{F}(f)
-\mu_1\mathcal{F}(\mathrm{div}^2)\mathcal{F}(\textbf{b}^k-\textbf{q}^k)+\mu_2
\mathcal{F}(\mathrm{div})\mathcal{F}(\textbf{c}^k-\textbf{v}^k)-\mu_3\mathcal{F}(d^k-z^k)$ and $D:=\mathcal{F}(\mathcal{A}^\star\mathcal{A})+\mu_1\mathcal{F}(\mathrm{div}^2\cdot\nabla^2)
-\mu_2\mathcal{F}(\mathrm{div}\cdot\nabla)+\mu_3\mathcal{F}(I)$. Here $\mathcal{F}$ denotes the fast Fourier transform and $\mathcal{F}^{-1}$ denotes the inverse of the fast Fourier transform.
\item[$\bullet$]\textbf{To} \textbf{solve} \textbf{the} \textbf{subproblem} \textbf{(\ref{36b})} \textbf{and } \textbf{(\ref{36c})}:
The subproblems (\ref{36b}) and (\ref{36c}) have the closed form solution. Since the closed form solution of minimization problem
\begin{eqnarray*}
\min_\varrho\frac{\xi}{2}\|\varrho-\varpi\|_2^2+|\varrho|_1
\end{eqnarray*}
is
\begin{eqnarray*}
\varrho=\max\left\{|\varpi|_1-\frac{1}{\xi},0\right\}\frac{\varpi}{|\varpi|_1}
\end{eqnarray*}
for $\varrho,\varpi\in R^s$ and $\xi>0$,
then the solution  of the problems (\ref{36b}) and (\ref{36c}) can be denoted as
\begin{numcases}{}
\textbf{q}^{k+1}=\max\left\{|\hbar|_1-\frac{\lambda(1-\omega(x))}{\mu_1},0\right\}
\frac{\hbar}{|\hbar|_1}\label{38a}\\
\textbf{v}^{k+1}=\max\left\{|\Im|_1-\frac{\gamma\omega(x)}{\mu_2},0\right\}
\frac{\Im}{|\Im|_1},\label{38b}
\end{numcases}
where $\hbar=\textbf{b}^k+\nabla^2g^{k+1}$ and $\Im=\textbf{c}^k+\nabla g^{k+1}$.
\item[$\bullet$]\textbf{To} \textbf{solve} \textbf{the} \textbf{subproblem} \textbf{(\ref{36d})}:
The subproblem is the projection problem on the convex set $\mathcal{D}$. So  its solution can be obtained by
\begin{eqnarray}
z^{k+1}:=\max\left\{\min\{d^k+g^{k+1},\iota\},0\right\}.\label{39}
\end{eqnarray}
\end{itemize}

\subsection{The  K-means clustering threshold}
With the development and improvement of data mining technology, data clustering algorithm is gradually applied to some fields. Among clustering algorithms, the $K$-means clustering method can be applied in many fields which include image and audio data compression, preprocess of system modeling with radial basis function networks, and task decomposition of heterogeneous neural network structure.

The clustering method of $K$-means originally proposed in \cite{41}  is a data mining algorithm which performs clustering by using an iterative approach. It takes the number or desired clusters and the initial means as input and produces final means as output. If the algorithm is required to produce $K$ clusters, then there will be $K$ initial means and $K$ final means. After termination of the $K$-means clustering, each object in dataset becomes a member of one cluster. It can provide relatively good result for convex cluster with relatively flexible and high efficient. Though this method is sensitive to the choice of starting points and can only be applied to a small dataset, in our experiment we still obtain suitable segmentation results. So we use it to obtain some suitable thresholds for the image segmentation problem.  In order to segment the restored image $g$ generated by the first step, we first linearly stretched it $\bar{g}$ in the second step by the strategy as
\begin{eqnarray*}
\bar{g}=\frac{g-g_{\min}}{g_{\max}-g_{\min}},
\end{eqnarray*}
where $g_{\min}$ and $g_{\max}$ represent minimum and maximum of $g$ respectively. Denote $\rho_1\leq \rho_2\cdots\leq\rho_K$ to be the centers of the $K$ clusters use $K$ as follows $K-1$ of the pixel intensities of the restored image. Here we define the $K-1$ thresholds to be $\bar{\rho}_i=(\rho_i+\rho_{i+1})/2$.  Denote $\rho_0=0$ and $\rho_K=1$, then the $i-$th phase of $\bar{g}$ is given by $\{x:\bar{\rho}_{i-1}\leq\bar{g}\leq\bar{\rho}_{i}\}$.

\section{Numerical implementation and experimental results}
In this section, we arrange some experimental comparisons from our proposed method with other recent methods \cite{26,18,22,130,131}. Following from the  primal-dual approach, the authors in \cite{26} extended the method developed in \cite{13} to solve  the continuous Potts model with applications to the multiphase piecewise constant Mumford-Shah model, where the numerical comparisons illustrated that their method outperforms some methods reviewed in \cite{13}.  On the other hand, the method in \cite{13}  has exhibited the strength over those popular methods such as  the $\alpha$ expansion and $\alpha_-\beta$ swap \cite{37}, the method of Pock et al. \cite{38}, and the algorithm in Lellmann et al. \cite{23}. For the strategy in \cite{18}, its numerical results showed that they outperform the methods used in \cite{13,38,40}. In experiments, the model (\ref{26}) is different to other models since we first need to know mean values in every segmentation region. To summarize, the two-stage image segmentation scheme is written as follows:
\begin{algorithm}\label{al31} Solving the two-stage image segmentation model (\ref{31}) with the clustering method:
\begin{itemize}
\item \textbf{Step 1}.  Choose original values $g^0=f$,  $\textbf{b}^0=\textbf{q}^0=\textbf{0}$ $\textbf{c}^0=\textbf{v}^0=\textbf{0}$, $d^0=0$. Set $k=0$;
\item \textbf{Step 2}. Compute  values of $\left(g^k,\textbf{q}^k,\textbf{v}^k,\textbf{b}^k,\textbf{c}^k, d^k\right)$ by the strategies (\ref{37})-(\ref{39}). If the stopping criterion is satisfied, output $g$ and do the next step;
\item \textbf{Step 3}. Apply K-means method to obtain the cluster  and then get the mean value of each cluster $\rho_1\leq\rho_2\leq\cdots\leq\rho_k$. Set the thresholds as  $\bar{\rho}_i=(\rho_i+\rho_{i+1})/2$ and then get the segmentation image.
\end{itemize}
\end{algorithm}

In the numerical implementation, all of parameters are chosen optimally with respect to the chosen dataset by trials and errors by getting the best segmentation results. The stopping criterion is of reaching the maximum iteration number $\mathcal{K}$ or of satisfying  $\min\Big\{\|\textbf{b}^{k+1}-\textbf{b}^k\|_2,\|\textbf{c}^{k+1}-\textbf{c}^k\|_2$,
$\|d^{k+1}-d^k\|_2\Big\}\leq\epsilon$, where $\epsilon$ denotes a small  positive number defined by the user. It is worth noting that it is suitable for running the inner loop of the proposed algorithm, we however set it to be one and find that we also get the pleasing numerical results.  For the original image, we normalize it into the range $[0,1]$, so we need to set $\iota=1$ for the constraint in our proposed model. For the weighed function $w(x)$, it can weaken the influence of noise as the form $\omega(x)=\frac{1}{1+\varsigma\|\nabla f_{\sigma}\|_2^2}$, where $f_{\sigma}=G_{\sigma}\ast f$. Here  $G_{\sigma}$ denotes the Gaussian kernel function with the standard deviation $\sigma$, $'\ast'$ is the convolution operation and $\varsigma$ is a scalar parameter. Furthermore, in order to simplify notations of those previous methods, we set the method used in \cite{5} to solve the CV model (\ref{22})  to be the CVM, the method used  the max flow method in \cite{130,131} to solve the Potts-type model  in (\ref{26}) as   the MFPM(max flow method for the POtts-type model), the method used  the new primal dual method in \cite{18} to solve the Potts-type model  in (\ref{26}) as the PDPM(primal dual method for the Potts-type model), the method used the two step in \cite{26} for solving the Mumford-Shah model (\ref{27}) to be the TSMSM(two step for the Mumford-Shah model).  Moreover, we set our proposed method by using the Algorithm \ref{al31} as the HTVWM. Simultaneously, we use the method similar to Algorithm \ref{al31} to solve the unstrained model (removing the term $g\in[0,\iota]$  in (\ref{31})) as the HTVUM in order to show the effectiveness of the constrained problem (\ref{31}). As we know that the K-means method is a local method so  we run each method ten times and then choose their average value.

For degraded images, we used  Matlab functions ``\textbf{imfilter}" and ``\textbf{imnoise}" to generate noise and blurring effects. Here we only use blurring kernels of Gaussian and motion. For the convenience of description, we
denote the Gaussian blur  with a blurring size $s$ and a standard deviation $\sigma_g$ as $(G,s,\sigma_g)$. Similarly, the motion blur  with a motion length $\imath$ and an angle $\theta$ is denoted as $(M,\imath,\theta)$. To more fairly tune parameters, we save the corrupted image as the ``\textbf{mat}" format and then load it in numerical experiments.  All the experiments are run with the Matlab code on the work station of CPU 2.4GHz with RAM 4.00G.

\subsection{Piecewise constant image segmentation}
In this subsection, we use two artificial synthetic piecewise constant images shown in Figure \ref{fig40}  as the testing images.  The left image includes a simpler geometrical structures with two  different intensities, but the right image
\begin{figure}[h!]
\centering
\begin{minipage}[htbp]{0.245\linewidth}
\centering
\includegraphics[width=1.15in]{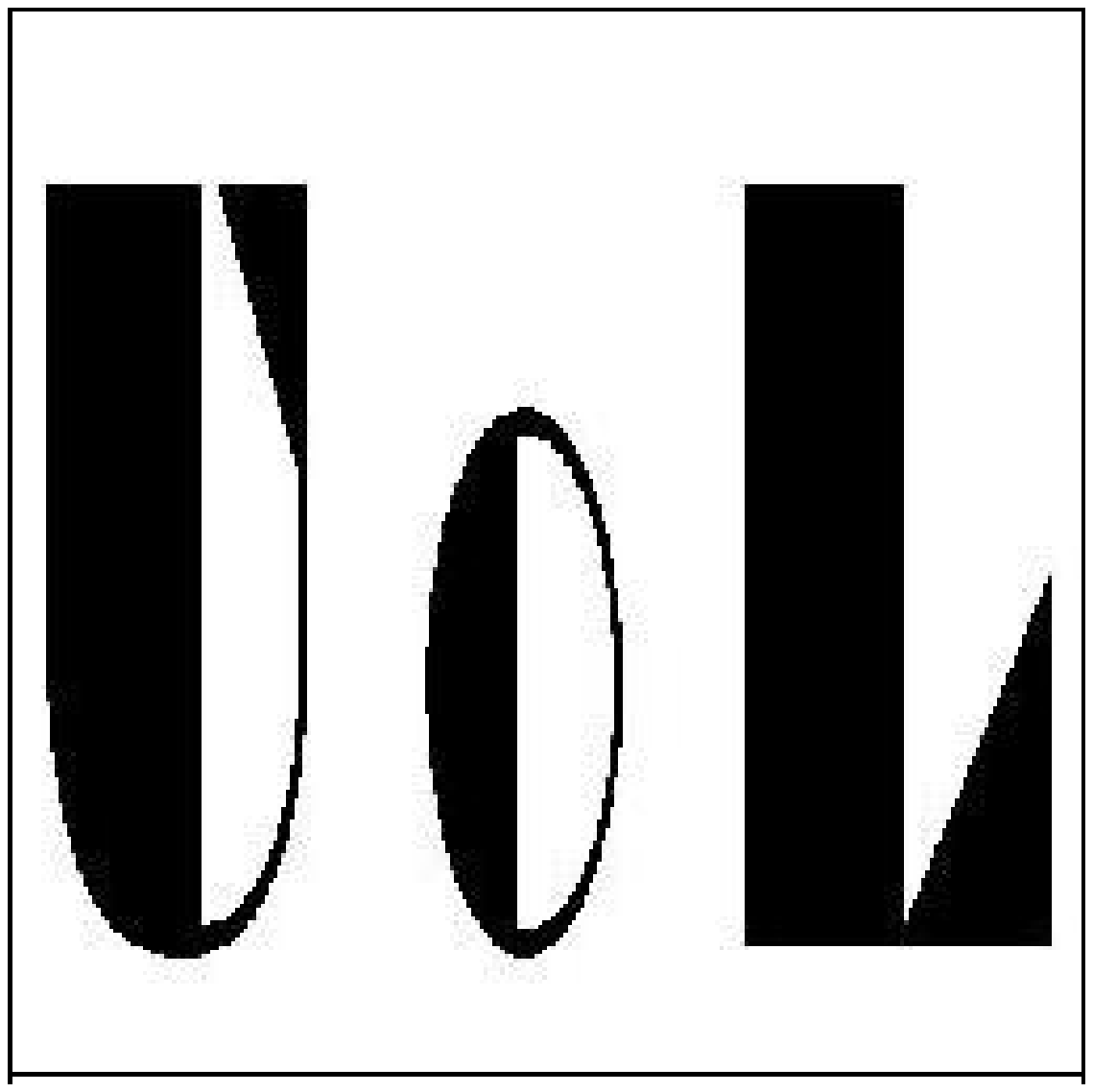}
{(a) Synthesis Image}
\end{minipage}
\begin{minipage}[htbp]{0.245\linewidth}
\centering
\includegraphics[width=1.15in]{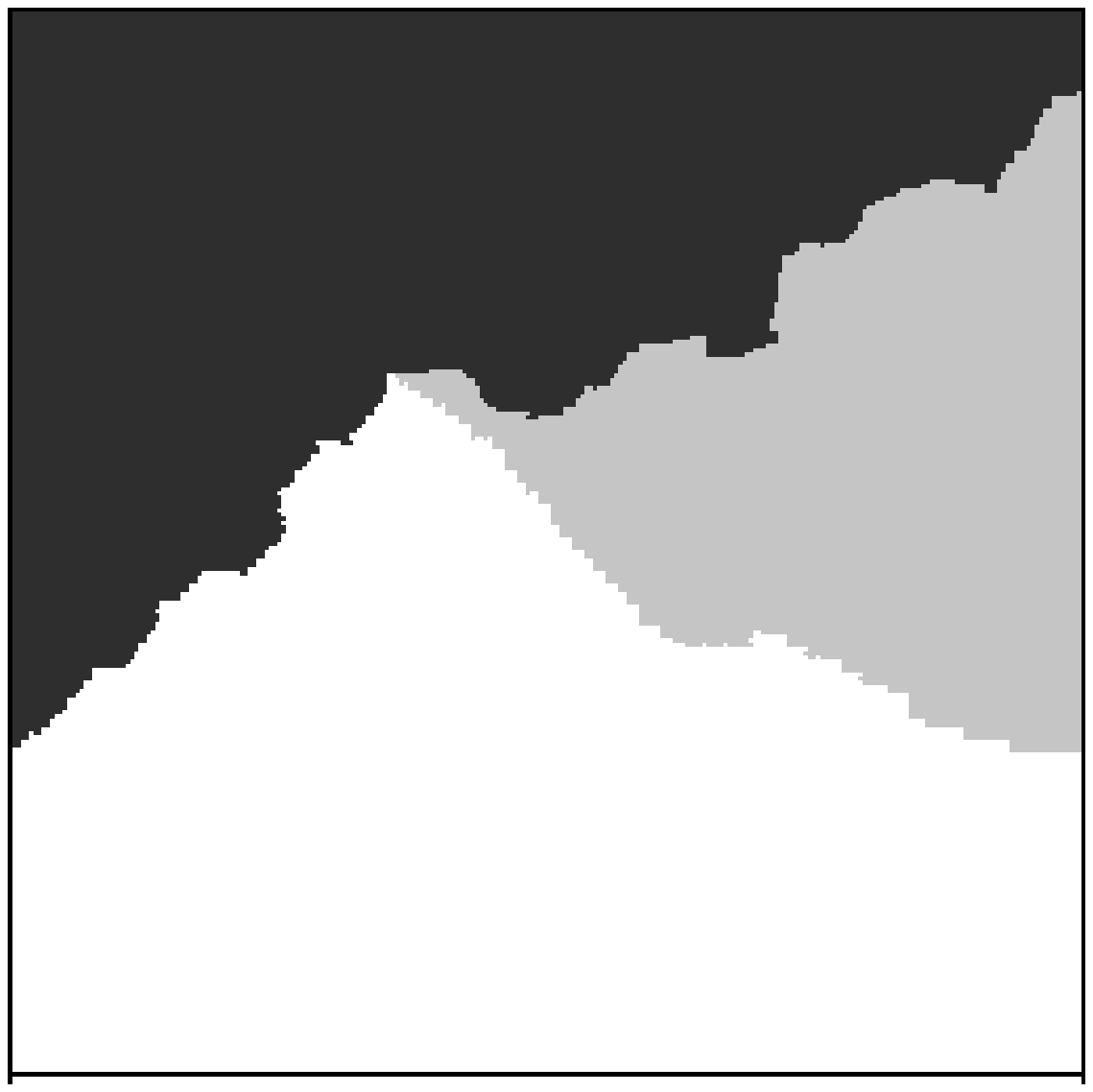}
{(b) Synthesis Image}
\end{minipage}
\caption{\label{fig40}The original image used in Example \ref{ex41} and \ref{ex42}. }
\end{figure}
includes more complicated boundaries with three different intensities. Since  the original segmentation results can  be accurately obtained for these images,  we define the segmentation accuracy(SA) as
\begin{eqnarray*}
SA=\frac{\mathrm{Number\hspace{2pt}of\hspace{2pt}uncorrectly\hspace{2pt}classified\hspace{2pt}pixels}}{\mathrm{Total
\hspace{2pt}number\hspace{2pt}of\hspace{2pt}pixels}}\times 100\%
\end{eqnarray*}
in order to evaluate the accuracy of the used models.

\begin{example}\label{ex41}
We choose Figure \ref{fig40}(a) for testing and comparison.  The degraded images are  shown in Figure \ref{fig411}. In Figure \ref{fig41}-\ref{fig43}, we show some
\begin{table}[htbp]
\centering
\begin{tabular}{|c||c||c||c|}
\hline Methods&Fig. 1&Fig. 2&Fig.3\\
\hline CVM&1.42\% &1.14\%&5.17\%\\
\hline MFPM&0.98\%&1.05\%&2.29\%\\
\hline PDPM&0.85\%&0.84\%&2.06\%\\
\hline TSMSM&0.72\%&0.82\%&1.39\%\\
\hline HTVUM&\textbf{0.70}\%&0.65\%&1.18\%\\
\hline HTVWM&0.72\%&\textbf{0.60}\%&\textbf{1.16}\%\\
\hline
\end{tabular}
\caption{\label{tab1}The SA in Example \ref{ex41}.}
\end{table}
segmentation results generated by using six kinds of aforementioned different models and methods, where the test image "UOL" is segmented into two phases  as foreground and background.  The related results are arranged in Table \ref{tab1}.

\begin{figure}[h!]
\centering
\begin{minipage}[htbp]{0.2\linewidth}
\centering
\includegraphics[width=0.85in]{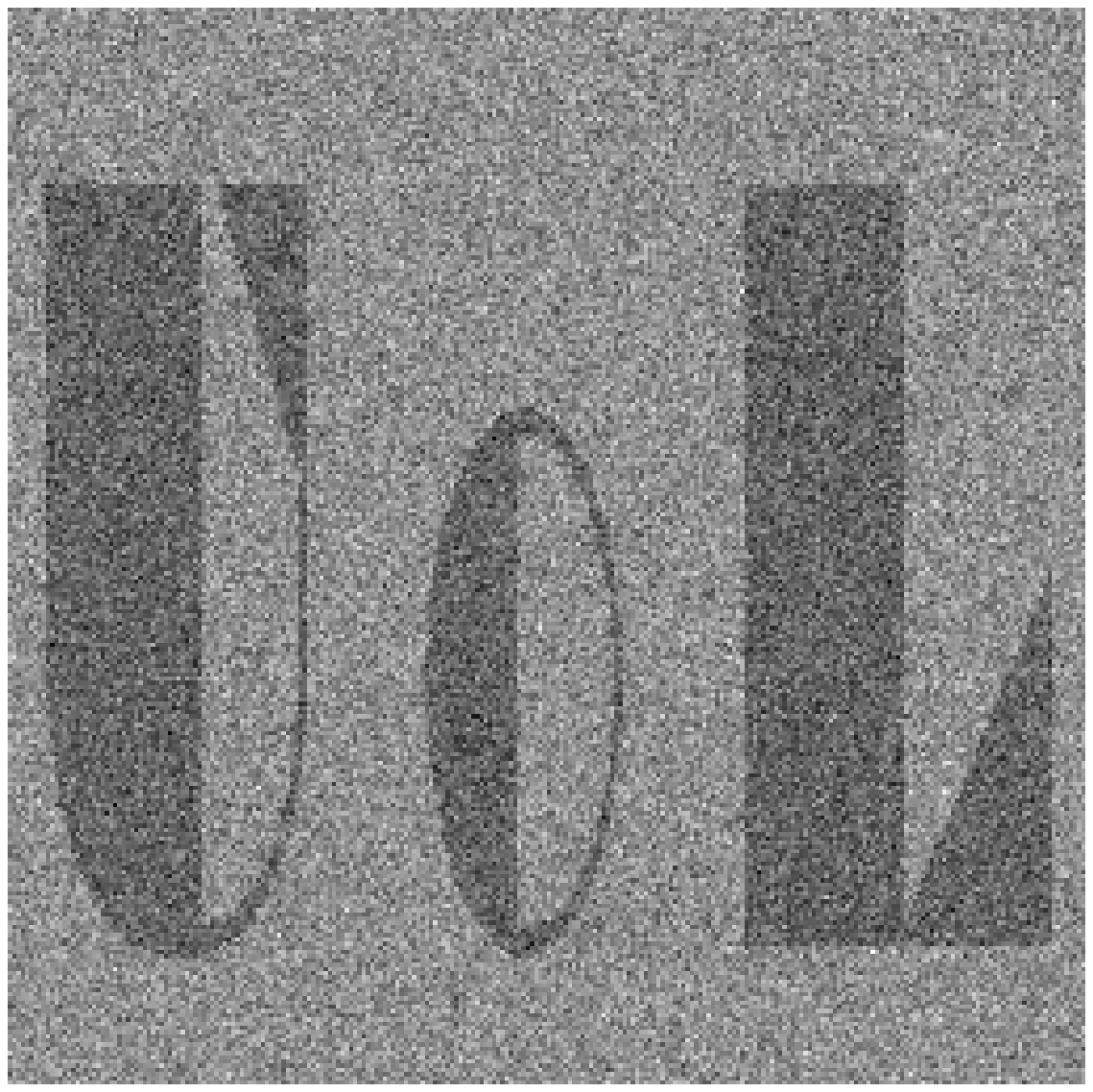}
{(a)  }
\end{minipage}
\begin{minipage}[htbp]{0.2\linewidth}
\centering
\includegraphics[width=0.85in]{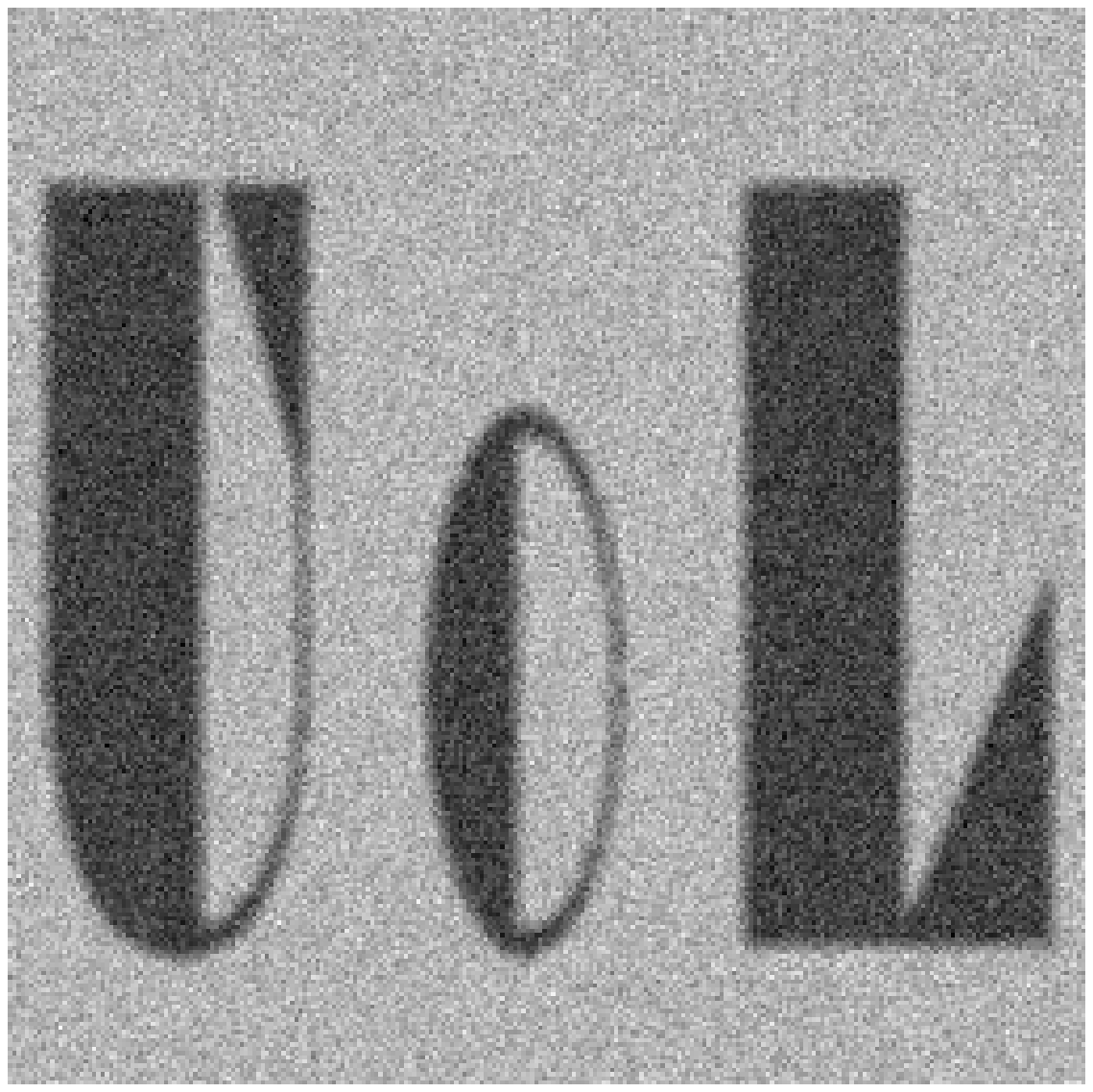}
{(b)}
\end{minipage}
\begin{minipage}[htbp]{0.2\linewidth}
\centering
\includegraphics[width=0.85in]{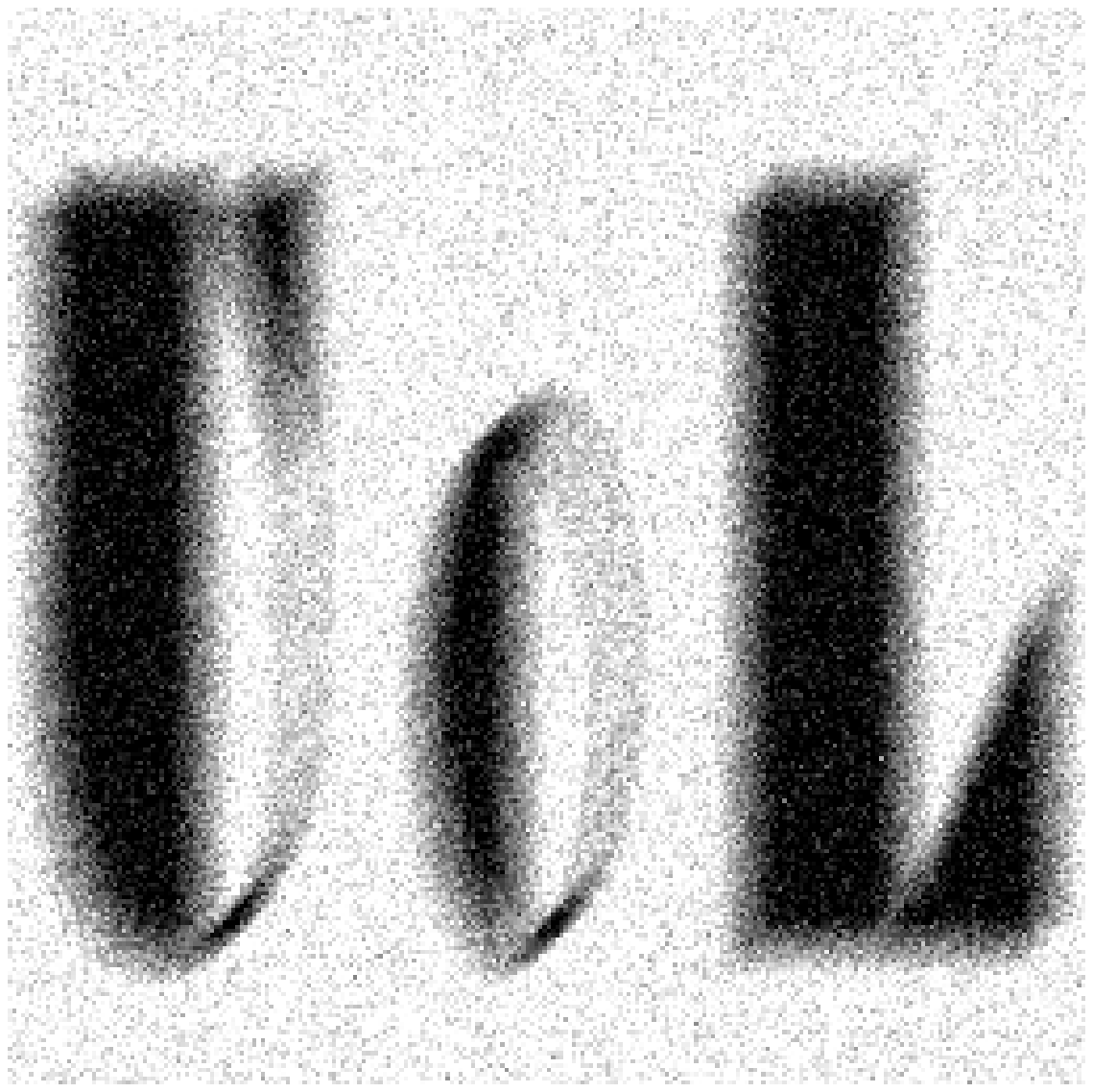}
{(c) }
\end{minipage}
\begin{minipage}[htbp]{0.1\linewidth}
\centering
\includegraphics[width=0.35in]{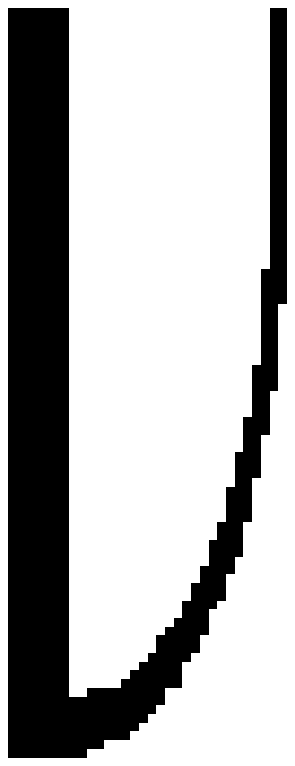}
{(d)}
\end{minipage}
\caption{\label{fig411}Some related images used in Example \ref{ex41}. (a). Contaminated by the white Gaussian noise with the variance 0.1; (b). Contaminated by the Gaussian blur $(G,5,5)$ and the white Gaussian noise with the variance 0.1; (c). Contaminated by the motion blur $(M,21,45)$ and the white Gaussian noise with the variance 0.01; (d). A part of "U" in "UOL".}
\end{figure}

As we can see from Figure \ref{fig41}, nearly all models  exhibit a similar segmentation performance to the noisy
\begin{figure}[h!]
\centering
\begin{minipage}[htbp]{0.15\linewidth}
\centering
\includegraphics[width=0.85in]{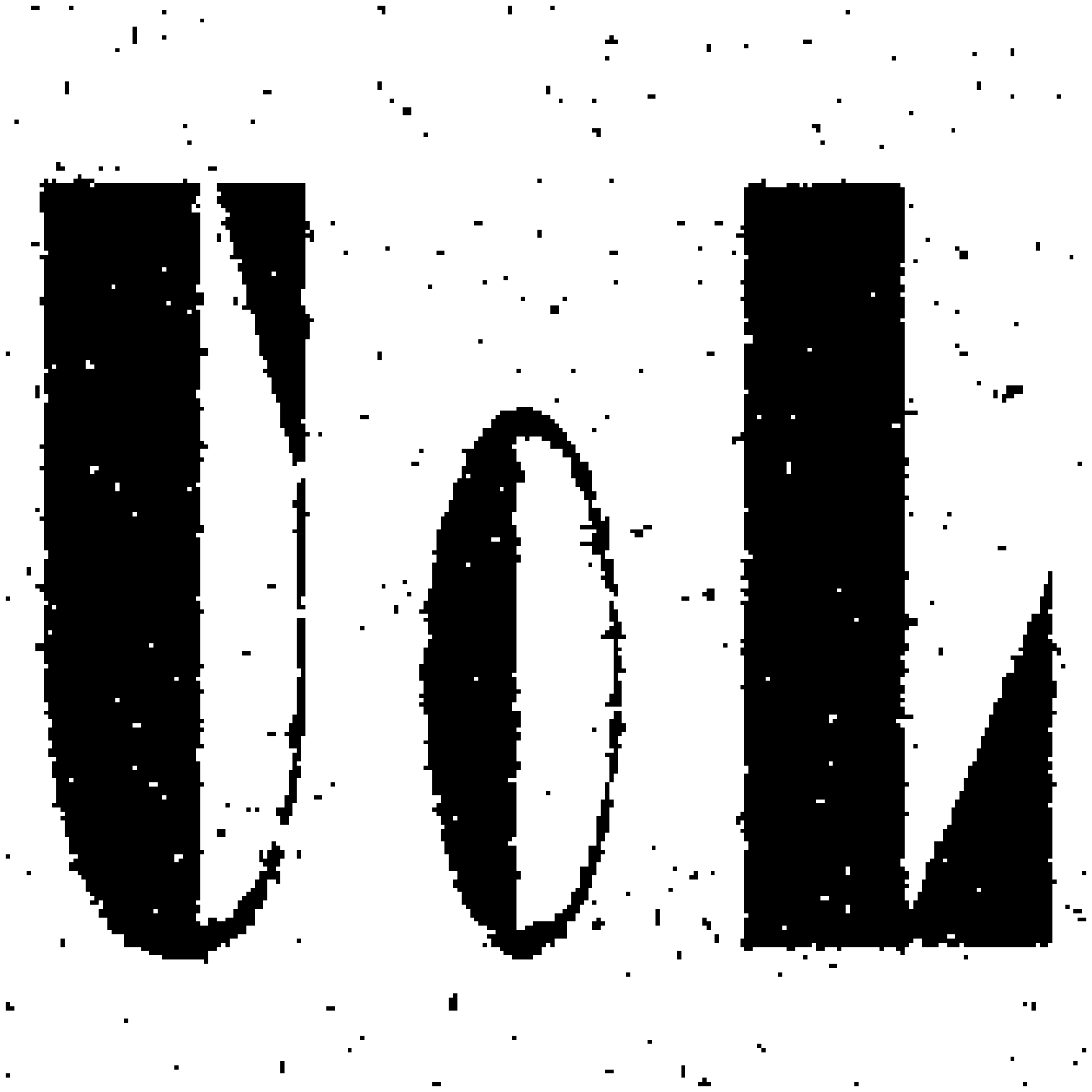}
{(a) \scriptsize{CVM}}
\end{minipage}
\begin{minipage}[htbp]{0.15\linewidth}
\centering
\includegraphics[width=0.85in]{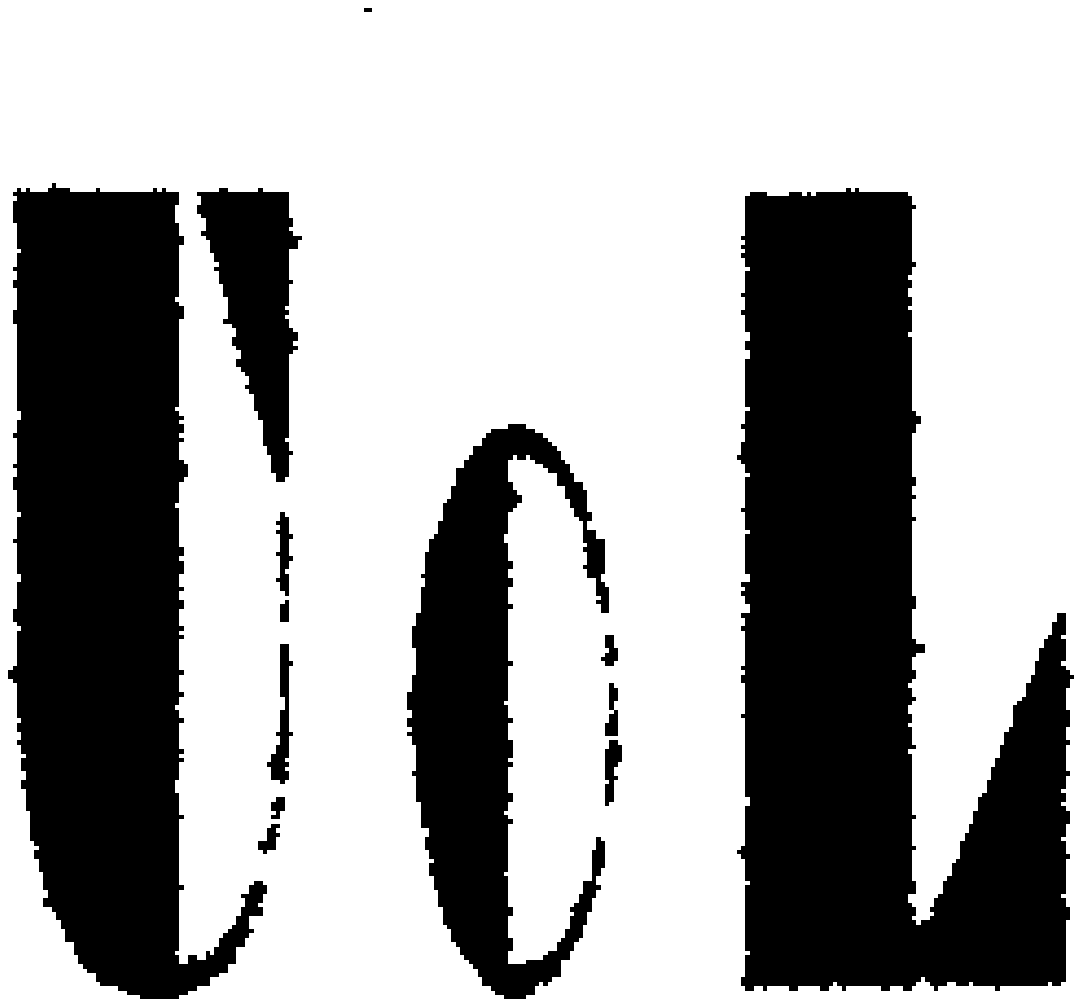}
{(b) \scriptsize{MFPM}}
\end{minipage}
\begin{minipage}[htbp]{0.15\linewidth}
\centering
\includegraphics[width=0.85in]{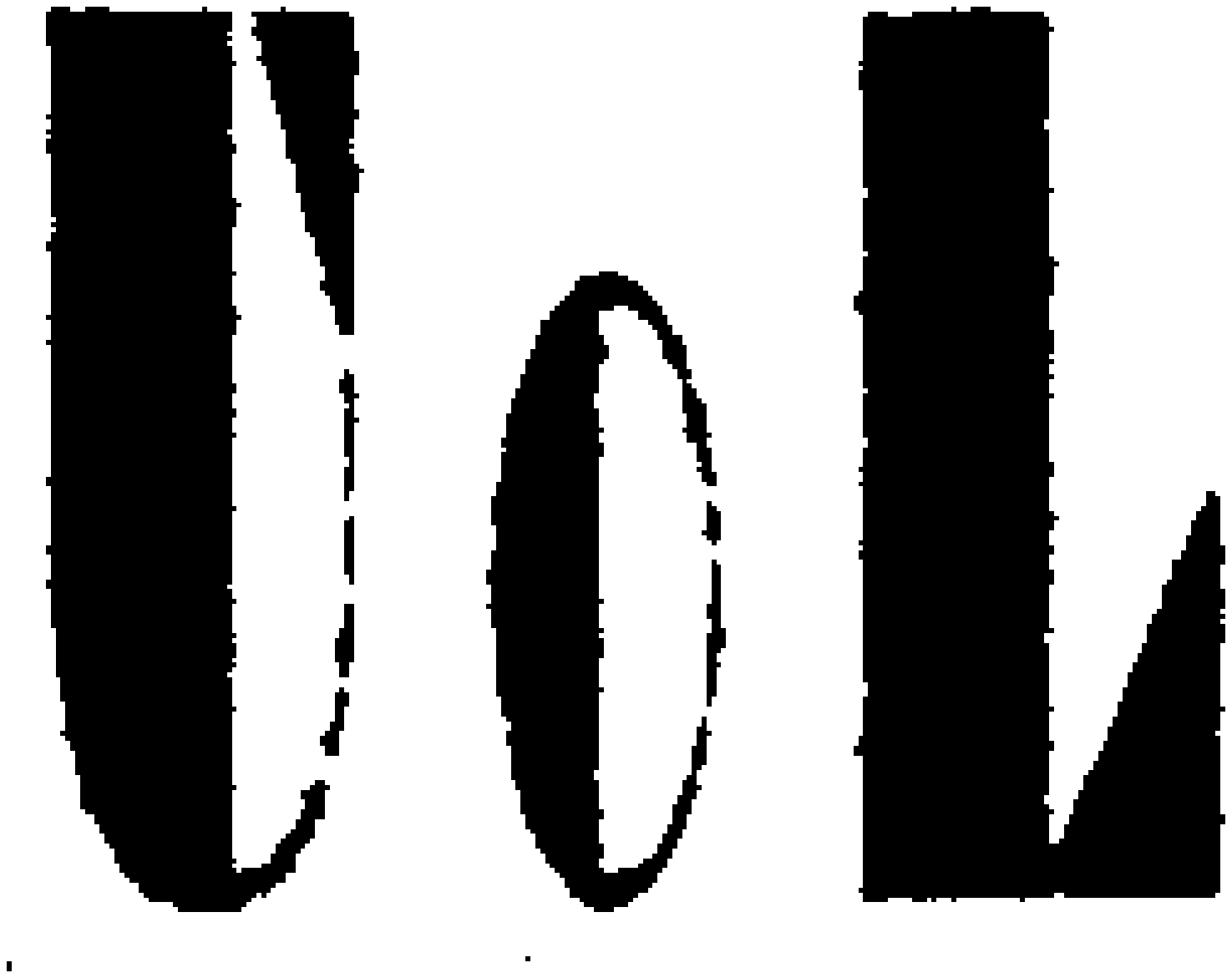}
{(c)\scriptsize{PDPM}}
\end{minipage}
\begin{minipage}[htbp]{0.15\linewidth}
\centering
\includegraphics[width=0.85in]{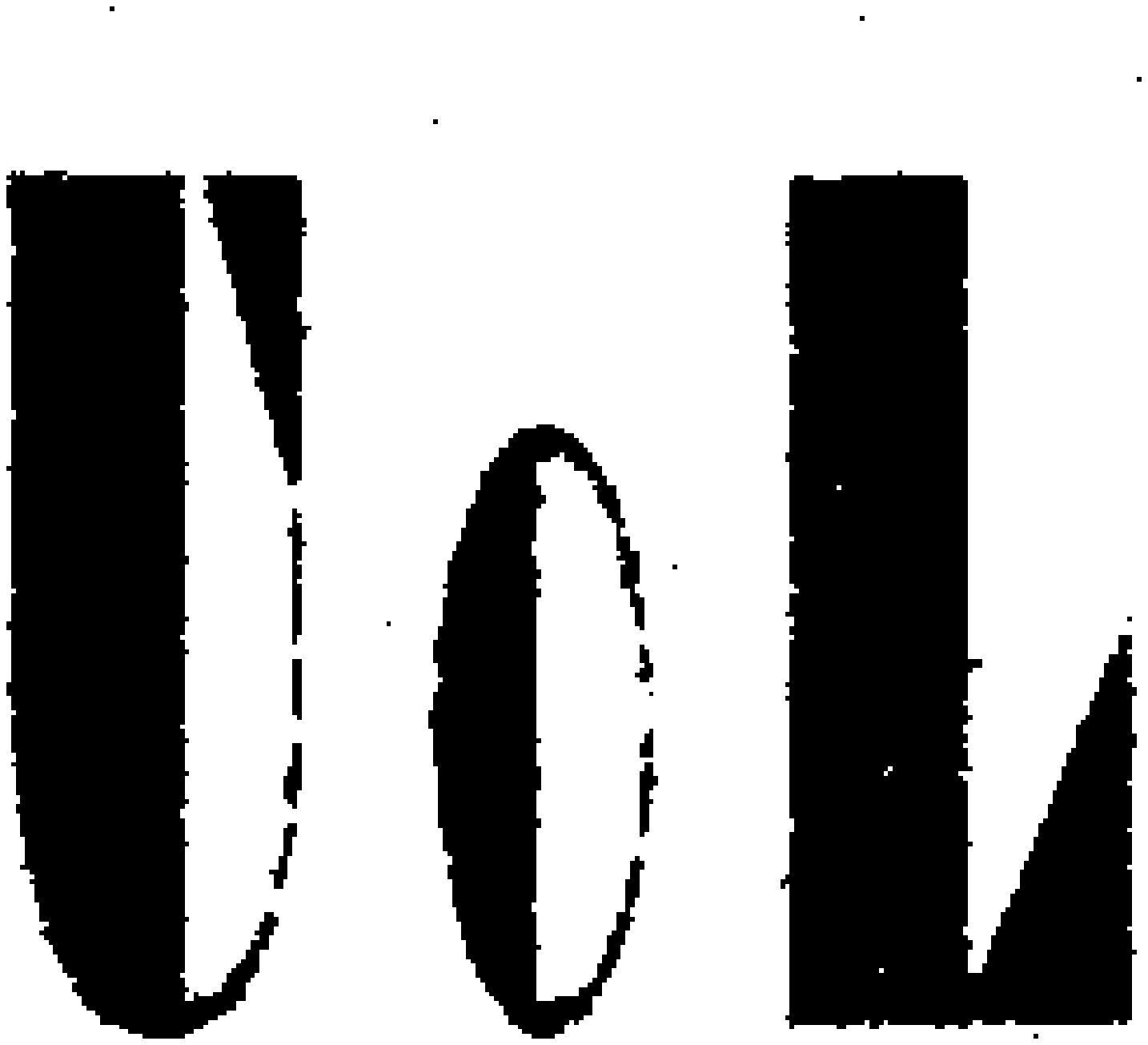}
{(d) \scriptsize{TSMSM}}
\end{minipage}
\begin{minipage}[htbp]{0.15\linewidth}
\centering
\includegraphics[width=0.85in]{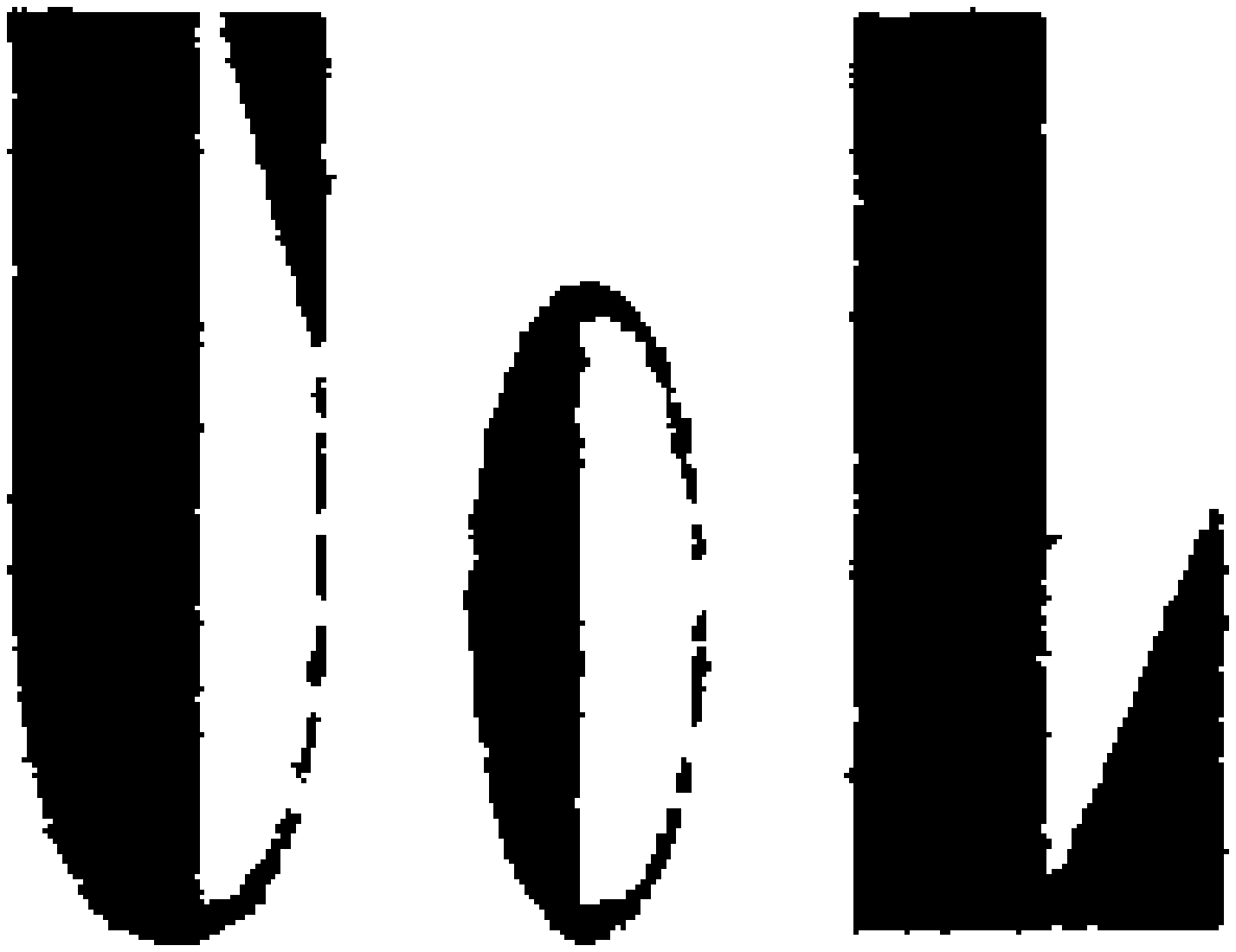}
{(e) \scriptsize{HTVUM}}
\end{minipage}
\begin{minipage}[htbp]{0.15\linewidth}
\centering
\includegraphics[width=0.85in]{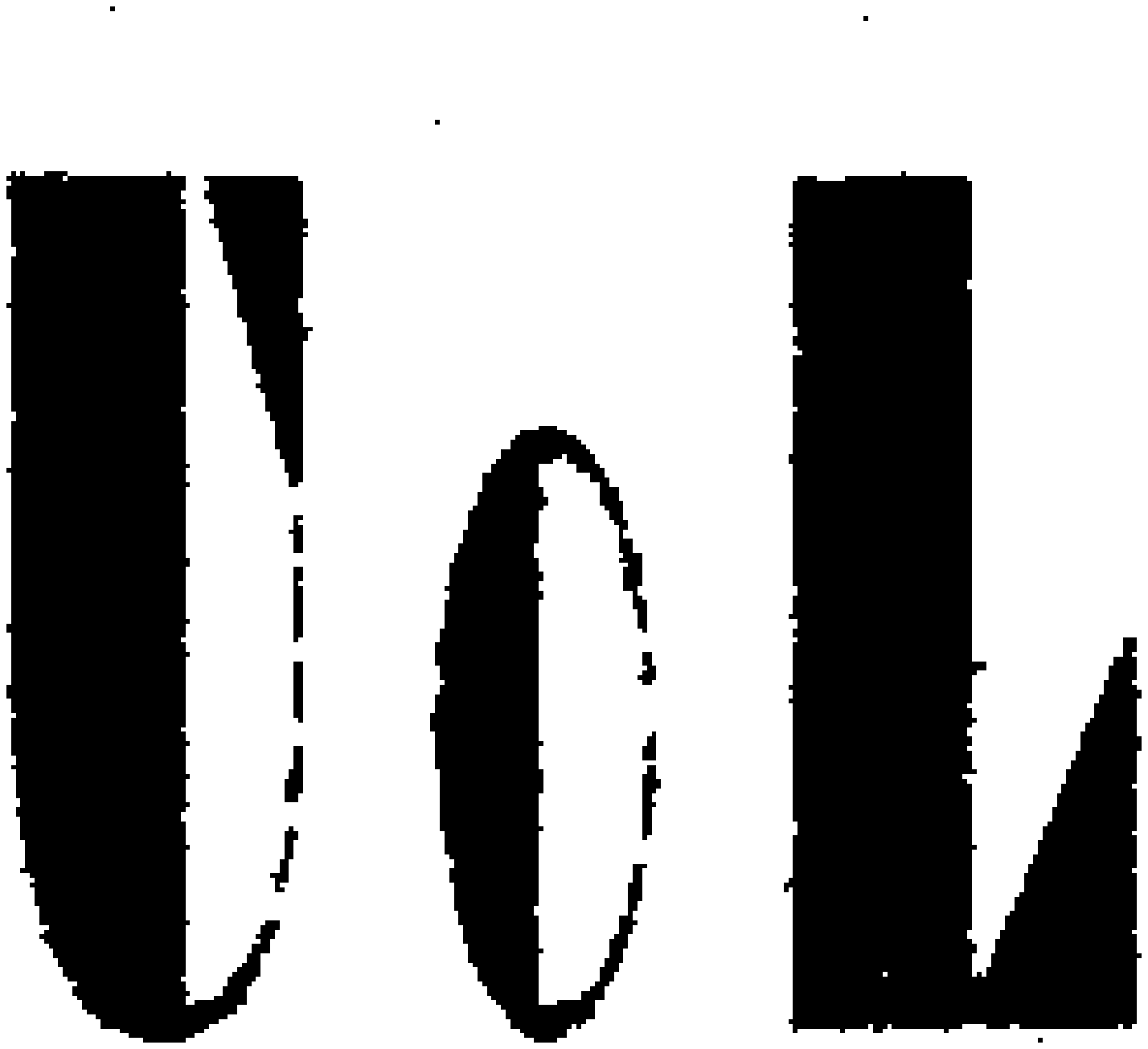}
{(f)\scriptsize{HTVWM}}
\end{minipage}\\
\begin{minipage}[htbp]{0.15\linewidth}
\centering
\includegraphics[width=0.45in]{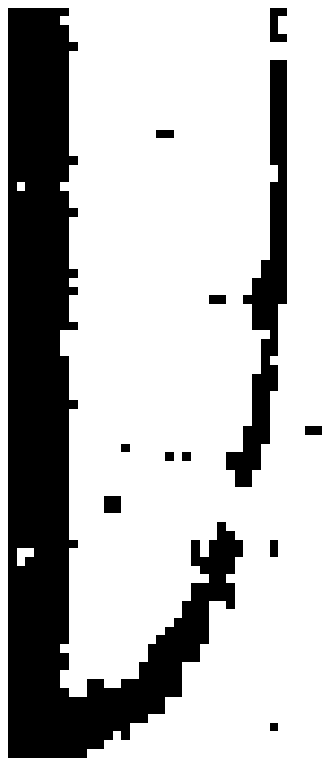}
{\hspace{10pt}(a1) \scriptsize{CVM}}
\end{minipage}
\begin{minipage}[htbp]{0.15\linewidth}
\centering
\includegraphics[width=0.45in]{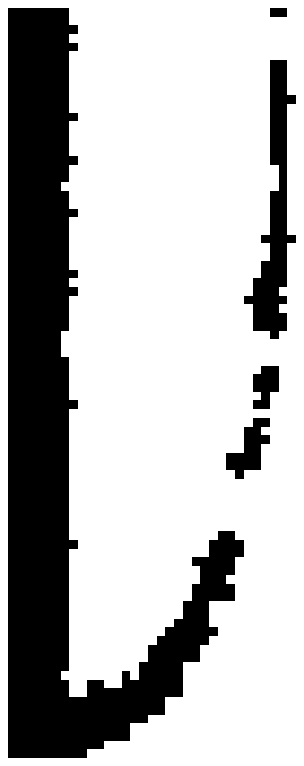}
{\hspace{10pt}(b1) \scriptsize{MFPM}}
\end{minipage}
\begin{minipage}[htbp]{0.15\linewidth}
\centering
\includegraphics[width=0.45in]{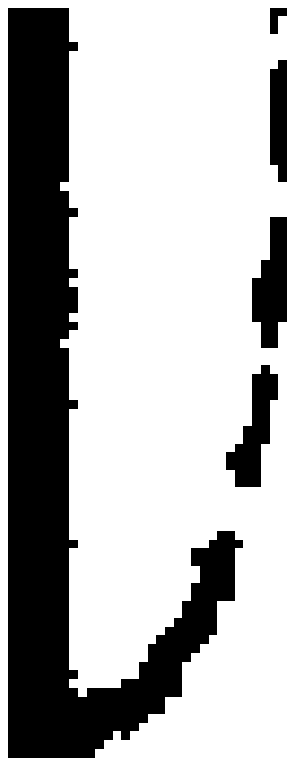}
{\hspace{10pt}(c1)\scriptsize{PDPM}}
\end{minipage}
\begin{minipage}[htbp]{0.15\linewidth}
\centering
\includegraphics[width=0.45in]{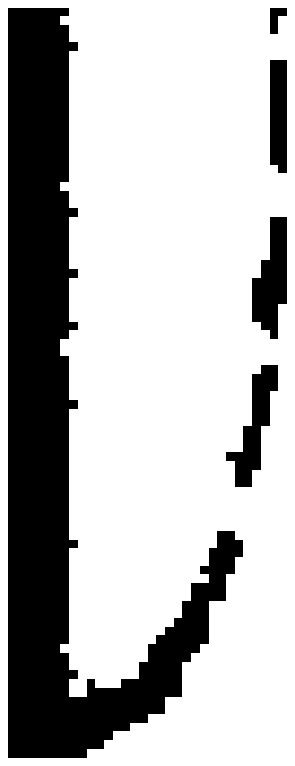}
{\hspace{10pt}(d1) \scriptsize{TSMSM}}
\end{minipage}
\begin{minipage}[htbp]{0.15\linewidth}
\centering
\includegraphics[width=0.45in]{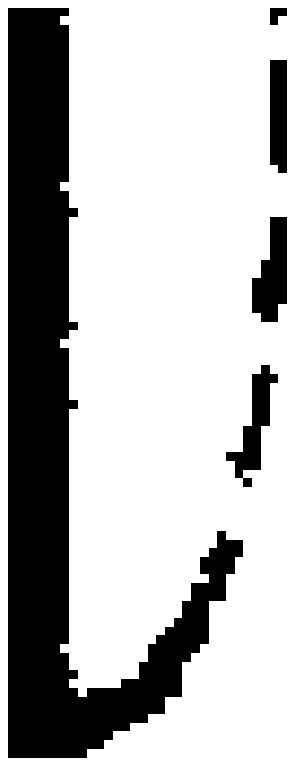}
{\hspace{10pt}(e1) \scriptsize{HTVUM}}
\end{minipage}
\begin{minipage}[htbp]{0.15\linewidth}
\centering
\includegraphics[width=0.45in]{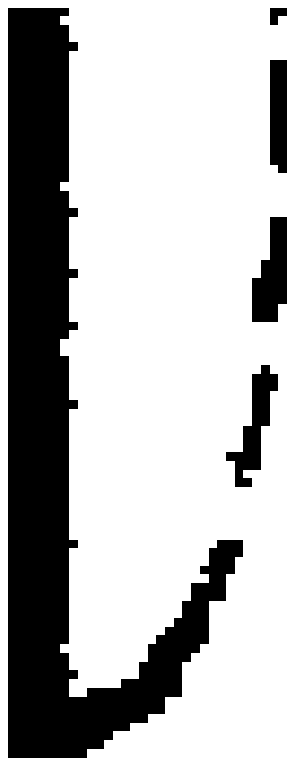}
{\hspace{10pt}(f1) \scriptsize{HTVWM}}
\end{minipage}
\caption{\label{fig41}Comparison of segmentation results  by using differential models for the noisy two-phase image in Example \ref{ex41}.  (a): The noisy image; (a1): The magnified portion of the original image; Row 2 (b1)-(f1): the magnified portion of the images corresponding to Row 1 (b-)-(f). Parameters: (b):$\gamma=200$; (c): $\gamma=0.65$; (d): $\gamma=2.1$ and $\lambda=0.02$; (e): $\gamma=1.95$ and $\lambda=0.2$; (f): $\gamma=1.95$ and $\lambda=0.1$.}
\end{figure}
image shown in Figure \ref{fig411}(a) except for the CVM \cite{5}. This is because of needing to choose the global constant for the CV model. For the blurring images in Figure \ref{fig411}(b) and \ref{fig411}(c), the CVM \cite{5} and the MFPM and PDPM  based on the  Potts-type model can not obtained efficient segmentation due to the blurring edge structures. For other three models, the degraded information has been efficiently suppressed before using the K-means method to segmentation, so we can get better segmentation results. Furthermore,  our proposed model (\ref{31}) yields the best segmentation
results because of adding the constraint and using the hybrid total variation space. We also see from Table \ref{tab1} that our proposed algorithm has a smaller SA values.
Related results can be also seen from the magnified portion of the letter U in Figure \ref{fig41}-\ref{fig43}.

\begin{figure}[h!]
\centering
\begin{minipage}[htbp]{0.15\linewidth}
\centering
\includegraphics[width=0.85in]{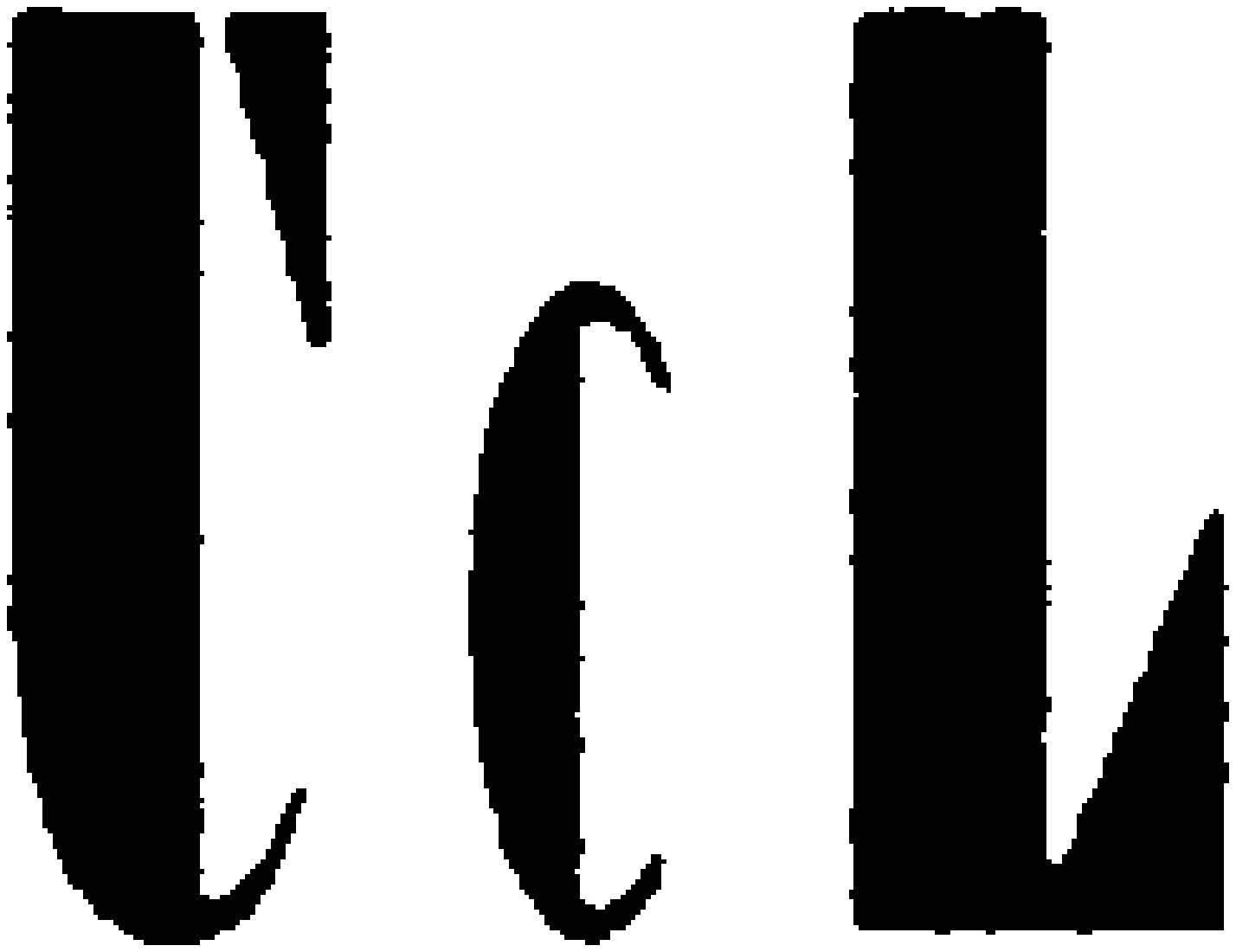}
{(a) \scriptsize{CVM}}
\end{minipage}
\begin{minipage}[htbp]{0.15\linewidth}
\centering
\includegraphics[width=0.85in]{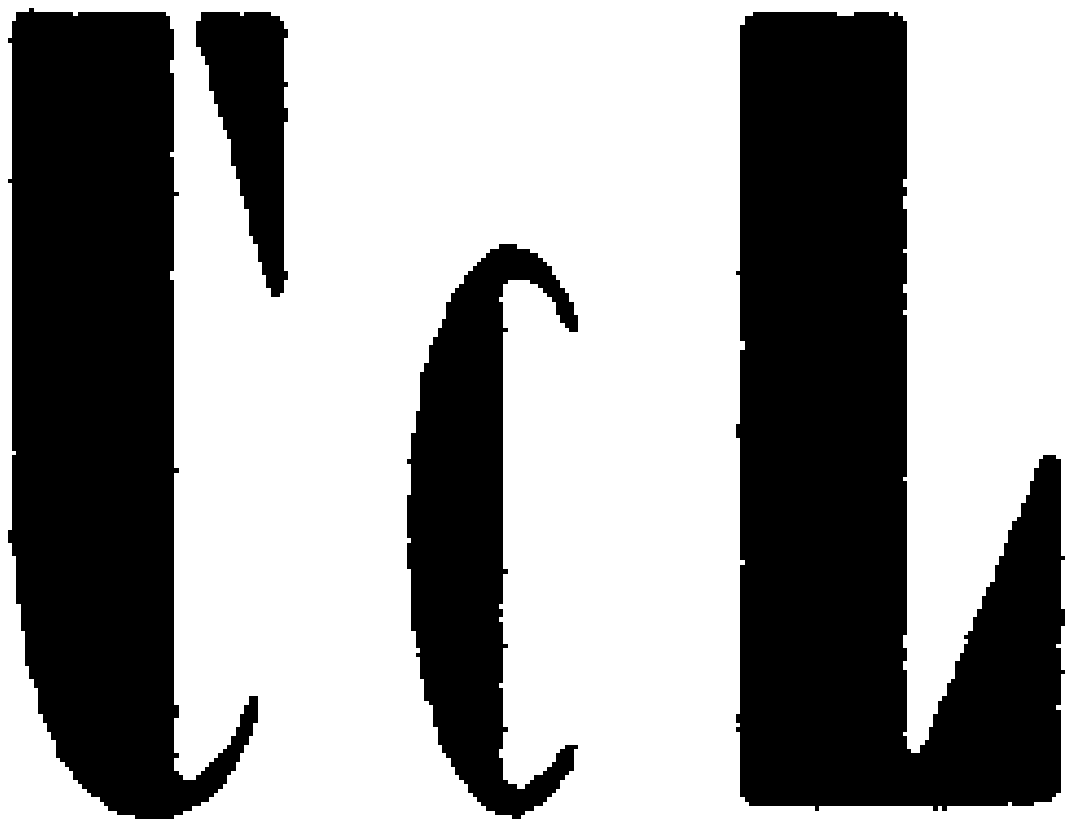}
{(b) \scriptsize{MFPM}}
\end{minipage}
\begin{minipage}[htbp]{0.15\linewidth}
\centering
\includegraphics[width=0.85in]{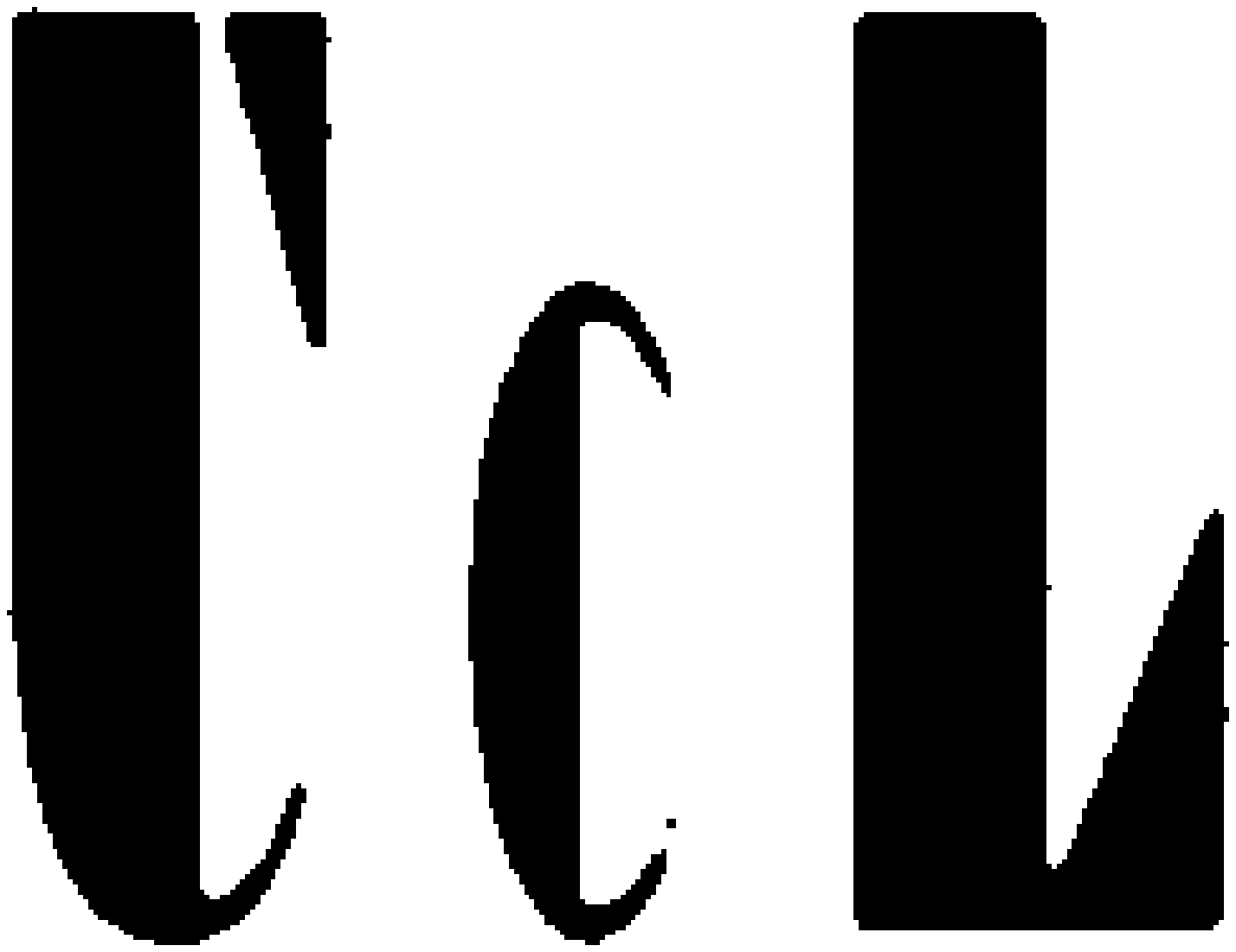}
{(c) \scriptsize{PDPM}  }
\end{minipage}
\begin{minipage}[htbp]{0.15\linewidth}
\centering
\includegraphics[width=0.85in]{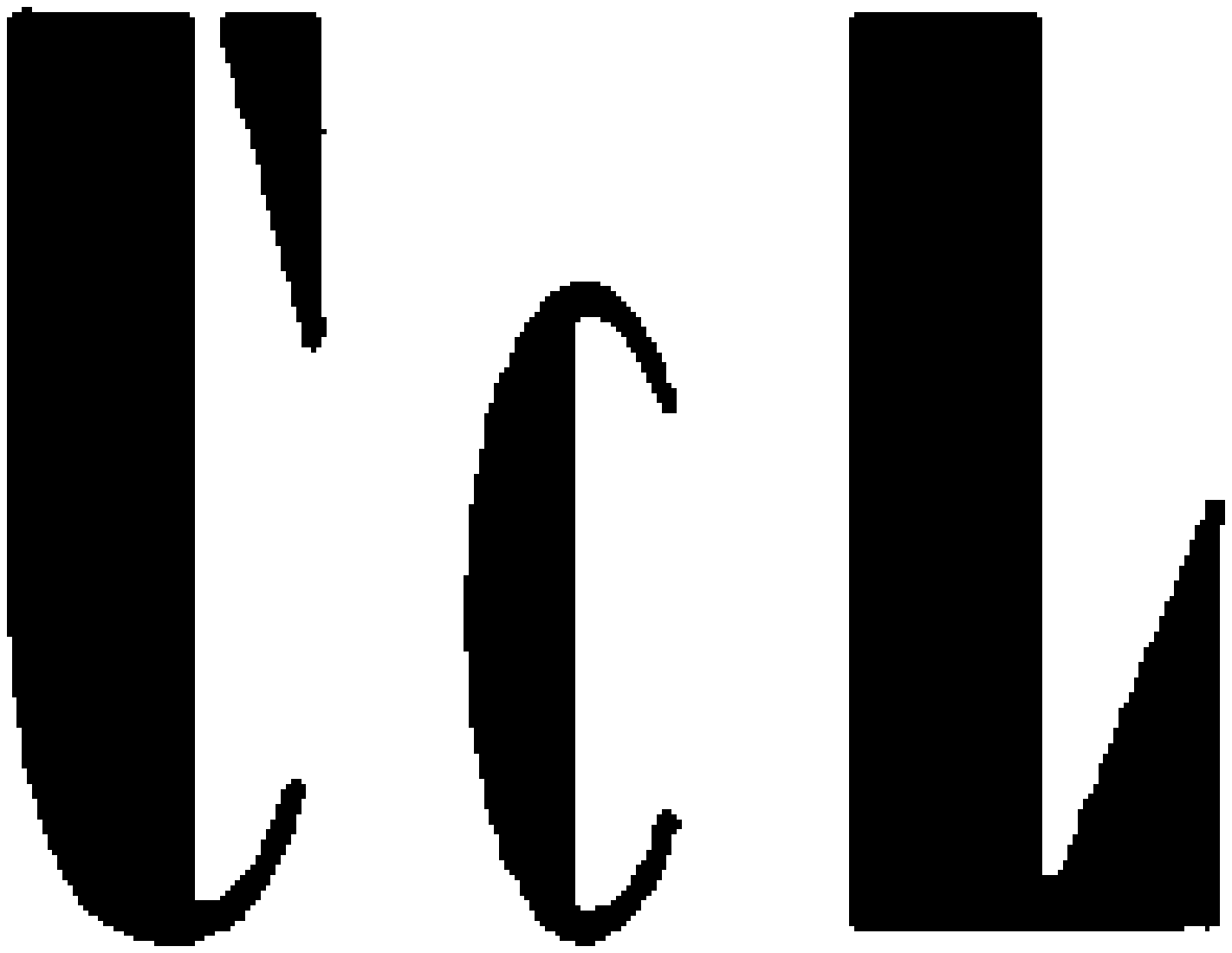}
{(d) \scriptsize{TSMSM}}
\end{minipage}
\begin{minipage}[htbp]{0.15\linewidth}
\centering
\includegraphics[width=0.85in]{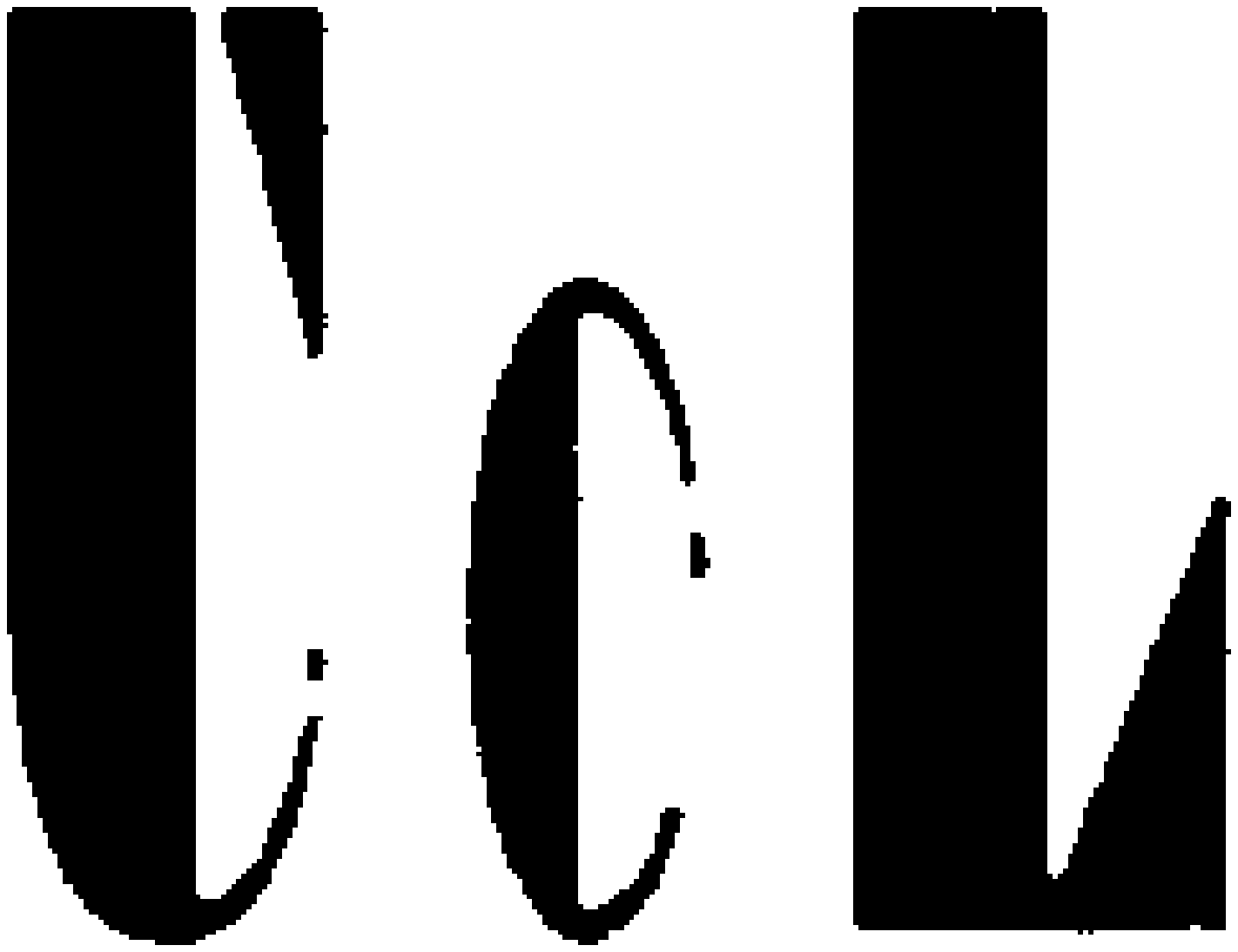}
{(e) \scriptsize{HTVUM}}
\end{minipage}
\begin{minipage}[htbp]{0.15\linewidth}
\centering
\includegraphics[width=0.85in]{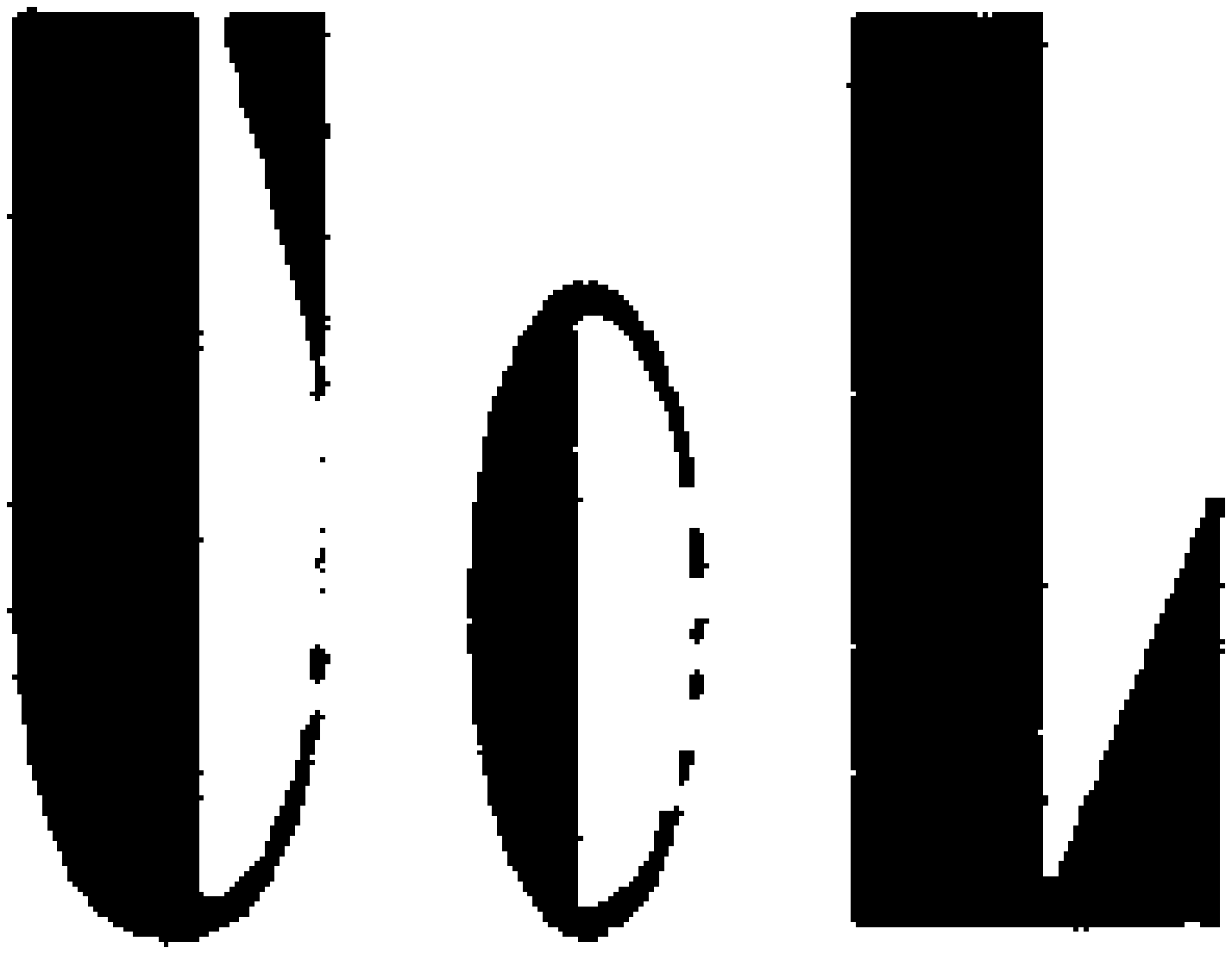}
{(f) \scriptsize{HTVWM}}
\end{minipage}\\
\begin{minipage}[htbp]{0.15\linewidth}
\centering
\includegraphics[width=0.45in]{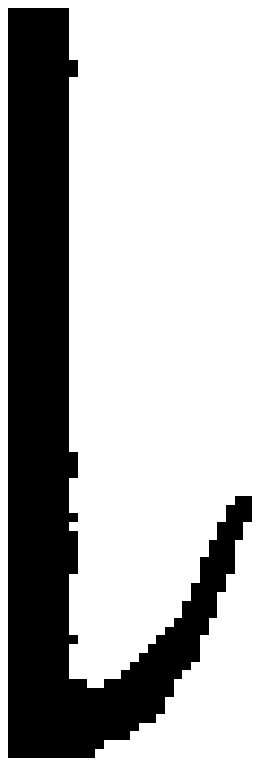}
{\hspace{10pt}(a1) \scriptsize{CVM}}
\end{minipage}
\begin{minipage}[htbp]{0.15\linewidth}
\centering
\includegraphics[width=0.45in]{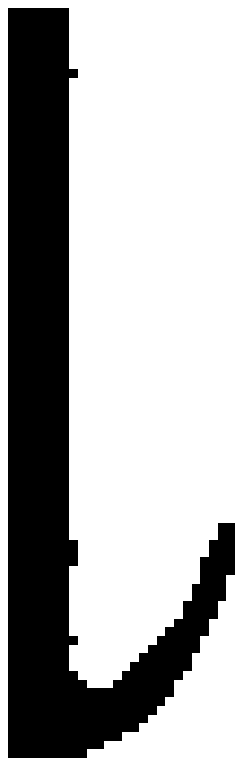}
{\hspace{10pt}(b1) \scriptsize{MFPM}}
\end{minipage}
\begin{minipage}[htbp]{0.15\linewidth}
\centering
\includegraphics[width=0.45in]{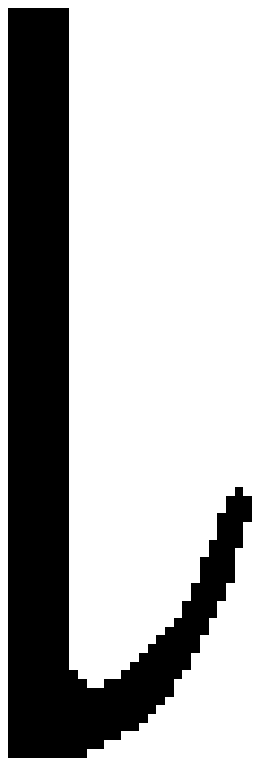}
{\hspace{10pt}(c1)  \scriptsize{FMPM}}
\end{minipage}
\begin{minipage}[htbp]{0.15\linewidth}
\centering
\includegraphics[width=0.45in]{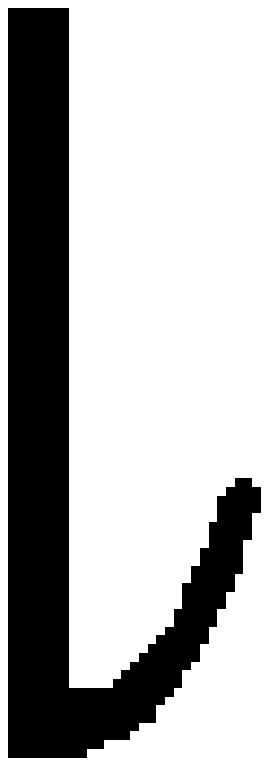}
{\hspace{10pt}(d1) \scriptsize{TSMSM}}
\end{minipage}
\begin{minipage}[htbp]{0.15\linewidth}
\centering
\includegraphics[width=0.45in]{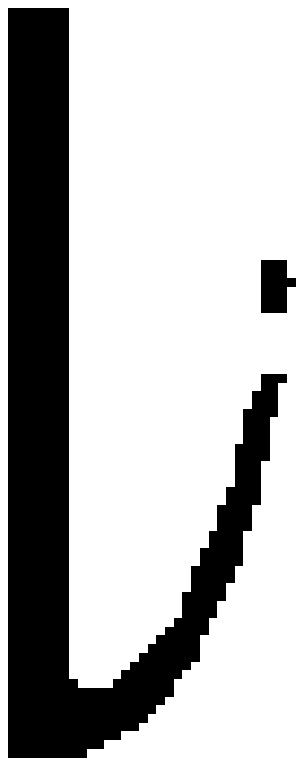}
{\hspace{10pt}(e1) \scriptsize{HTVUM}}
\end{minipage}
\begin{minipage}[htbp]{0.15\linewidth}
\centering
\includegraphics[width=0.45in]{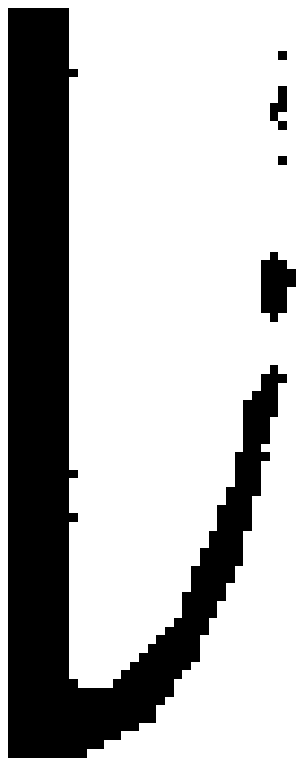}
{\hspace{10pt}(f1) \scriptsize{HTVWM}}
\end{minipage}
\caption{\label{fig42}Comparison of segmentation results  by using differential models for the noisy two-phase image in Example \ref{ex41}.  (a): The noisy image; (a1): The magnified portion of the original image; Row 2 (b1)-(f1): the magnified portion of the images corresponding to Row 1 (b-)-(f). Parameters: (b):  $\gamma=330$; (c): $\gamma=0.55$; (d): $\gamma=9.0$ and $\lambda=0.01$; (e): $\gamma=11.25$ and $\lambda=0.03$; (f): $\gamma=11.75$ and $\lambda=0.0075$.}
\end{figure}
\begin{figure}[h!]
\centering
\begin{minipage}[htbp]{0.15\linewidth}
\centering
\includegraphics[width=0.85in]{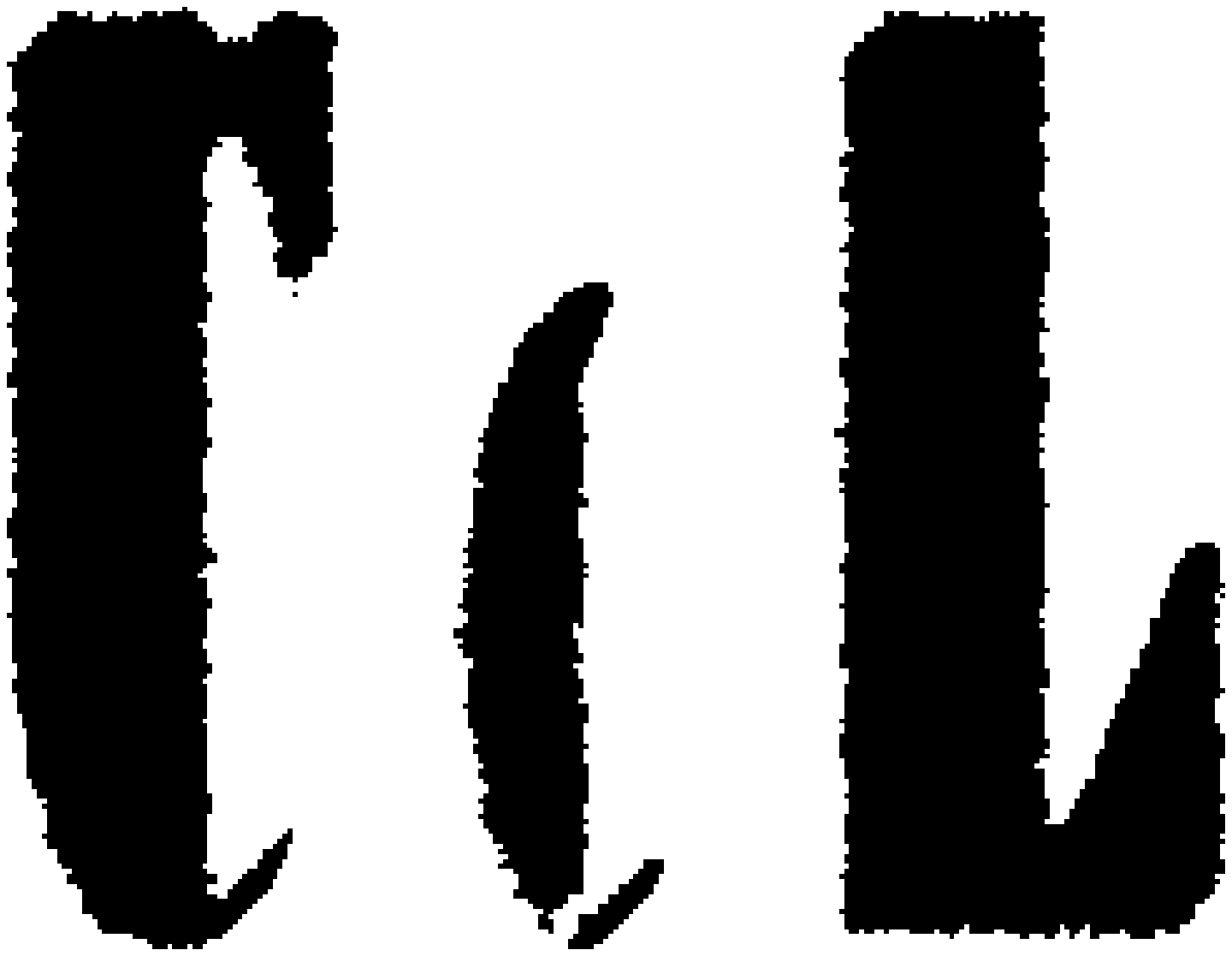}
{(a) \scriptsize{CVM}}
\end{minipage}
\begin{minipage}[htbp]{0.15\linewidth}
\centering
\includegraphics[width=0.85in]{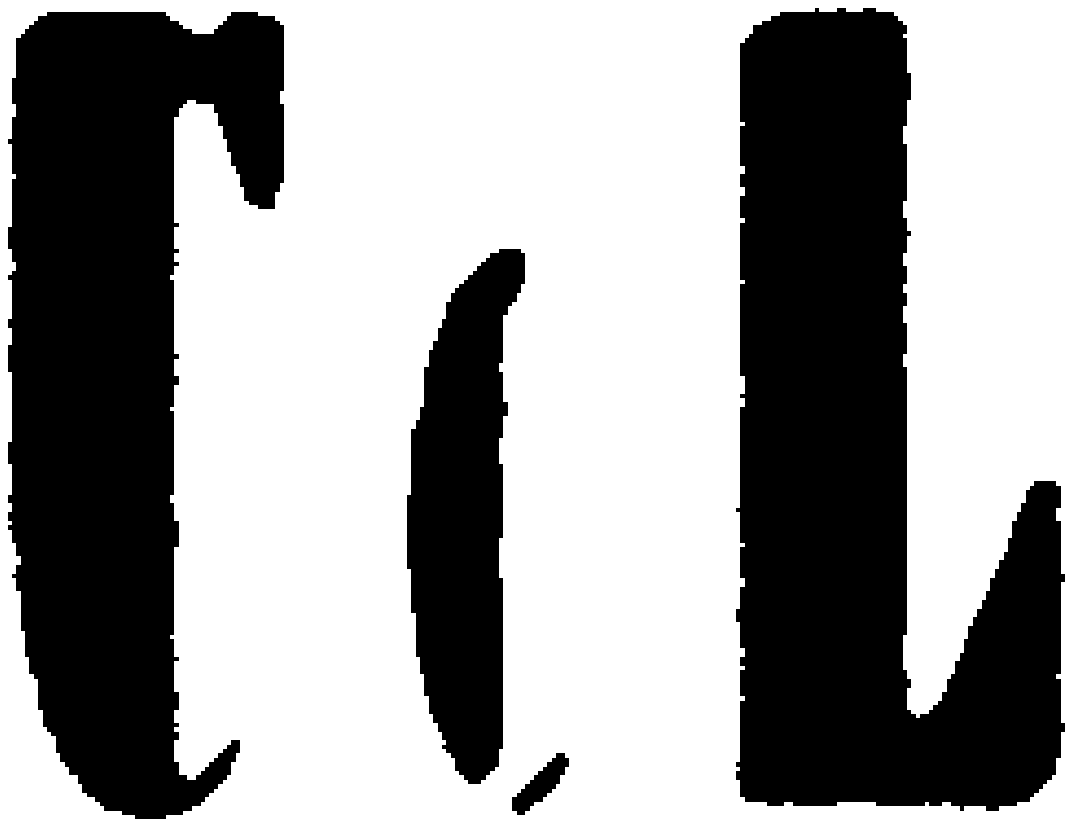}
{(b) \scriptsize{FMPM}}
\end{minipage}
\begin{minipage}[htbp]{0.15\linewidth}
\centering
\includegraphics[width=0.85in]{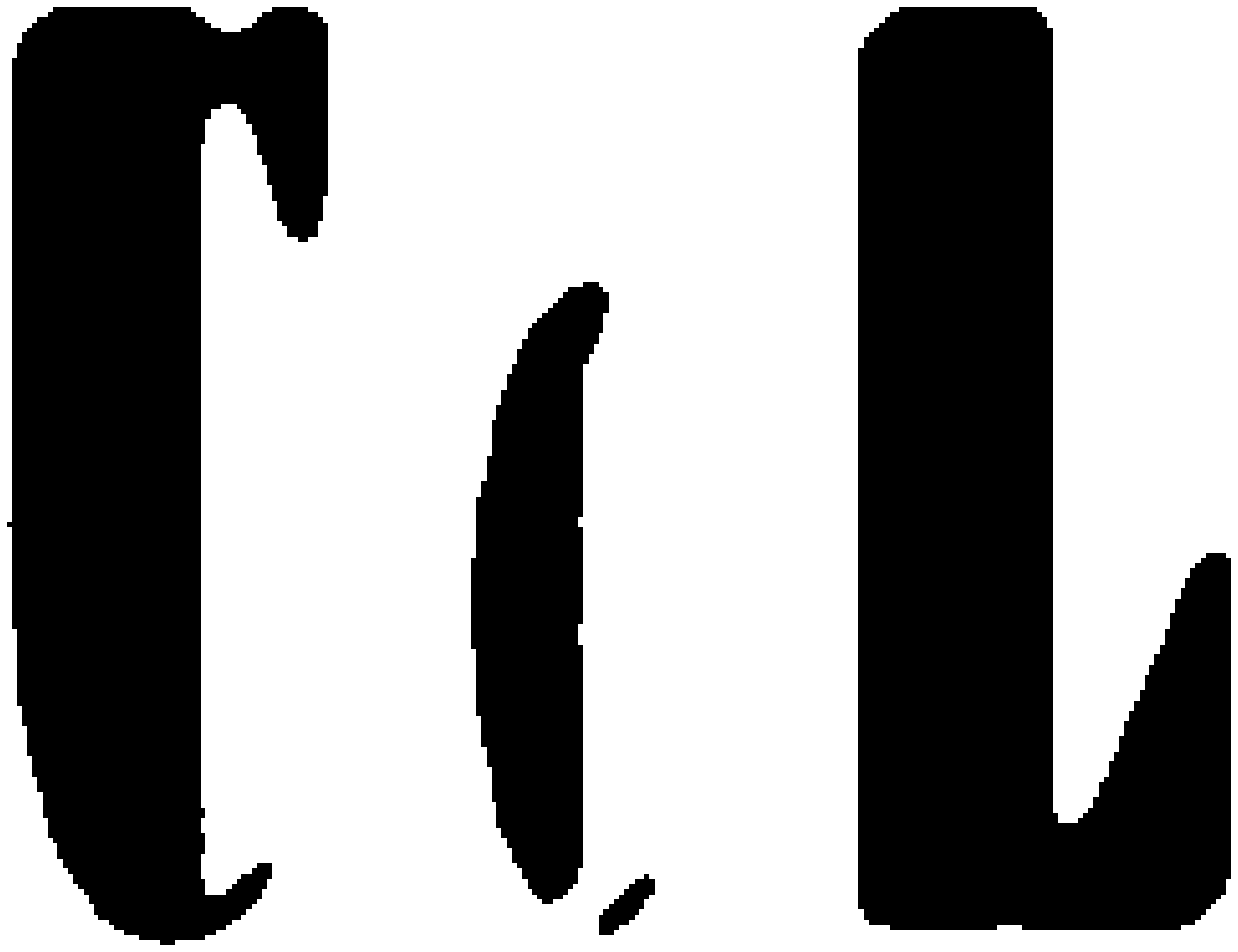}
{(c)  \scriptsize{PDPM}}
\end{minipage}
\begin{minipage}[htbp]{0.15\linewidth}
\centering
\includegraphics[width=0.85in]{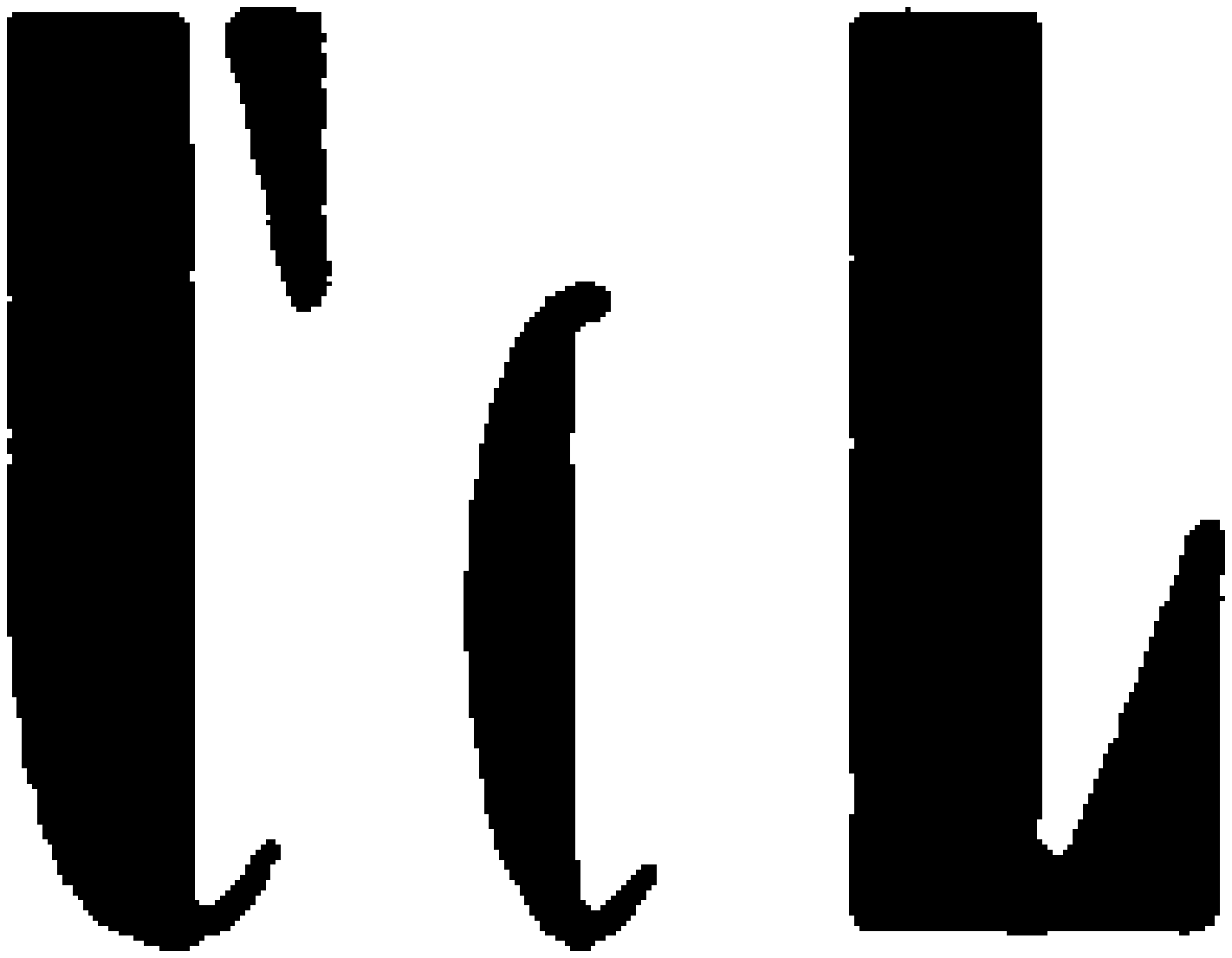}
{(d) \scriptsize{CCZM}}
\end{minipage}
\begin{minipage}[htbp]{0.15\linewidth}
\centering
\includegraphics[width=0.85in]{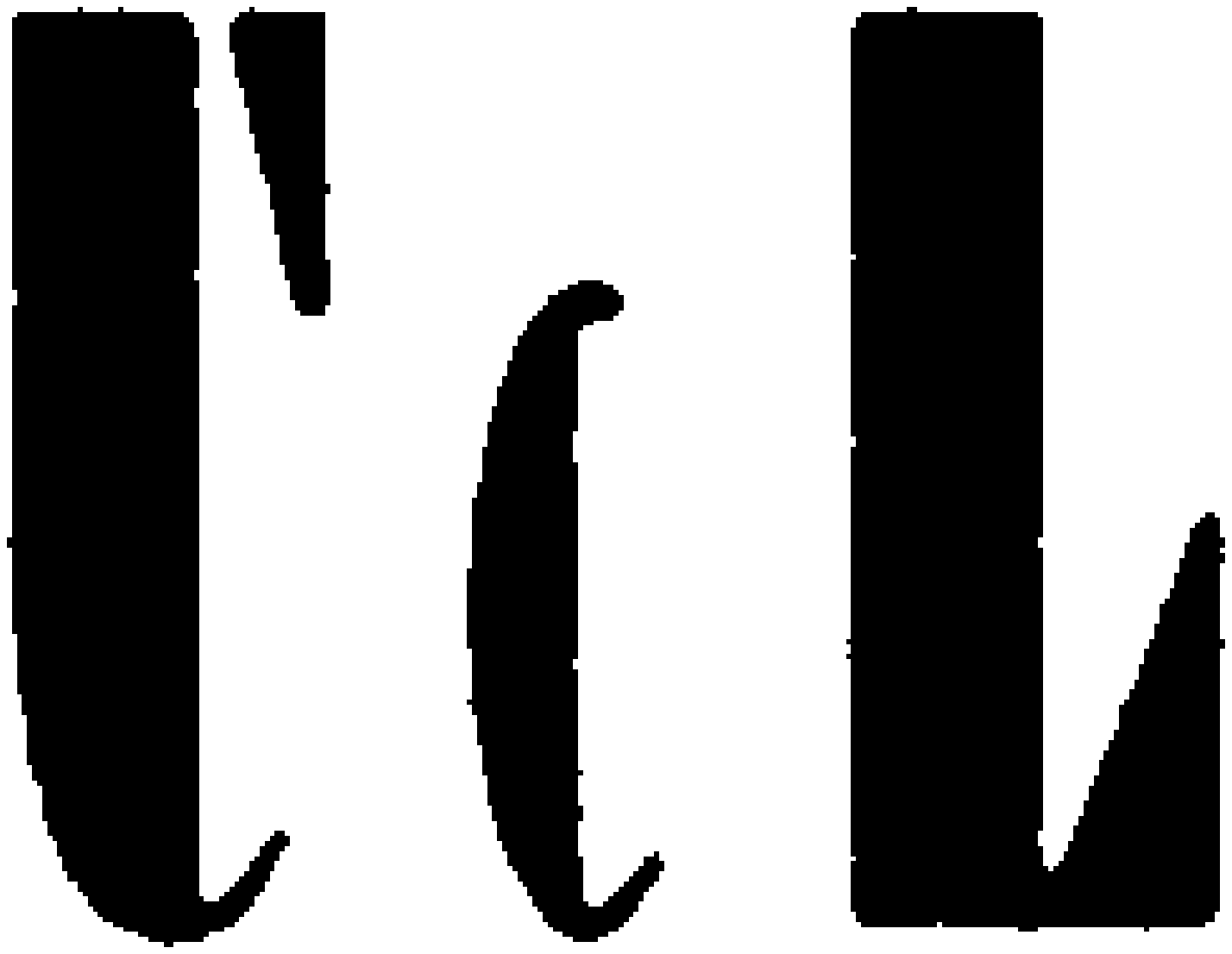}
{(e) \scriptsize{TSMSM}}
\end{minipage}
\begin{minipage}[htbp]{0.15\linewidth}
\centering
\includegraphics[width=0.85in]{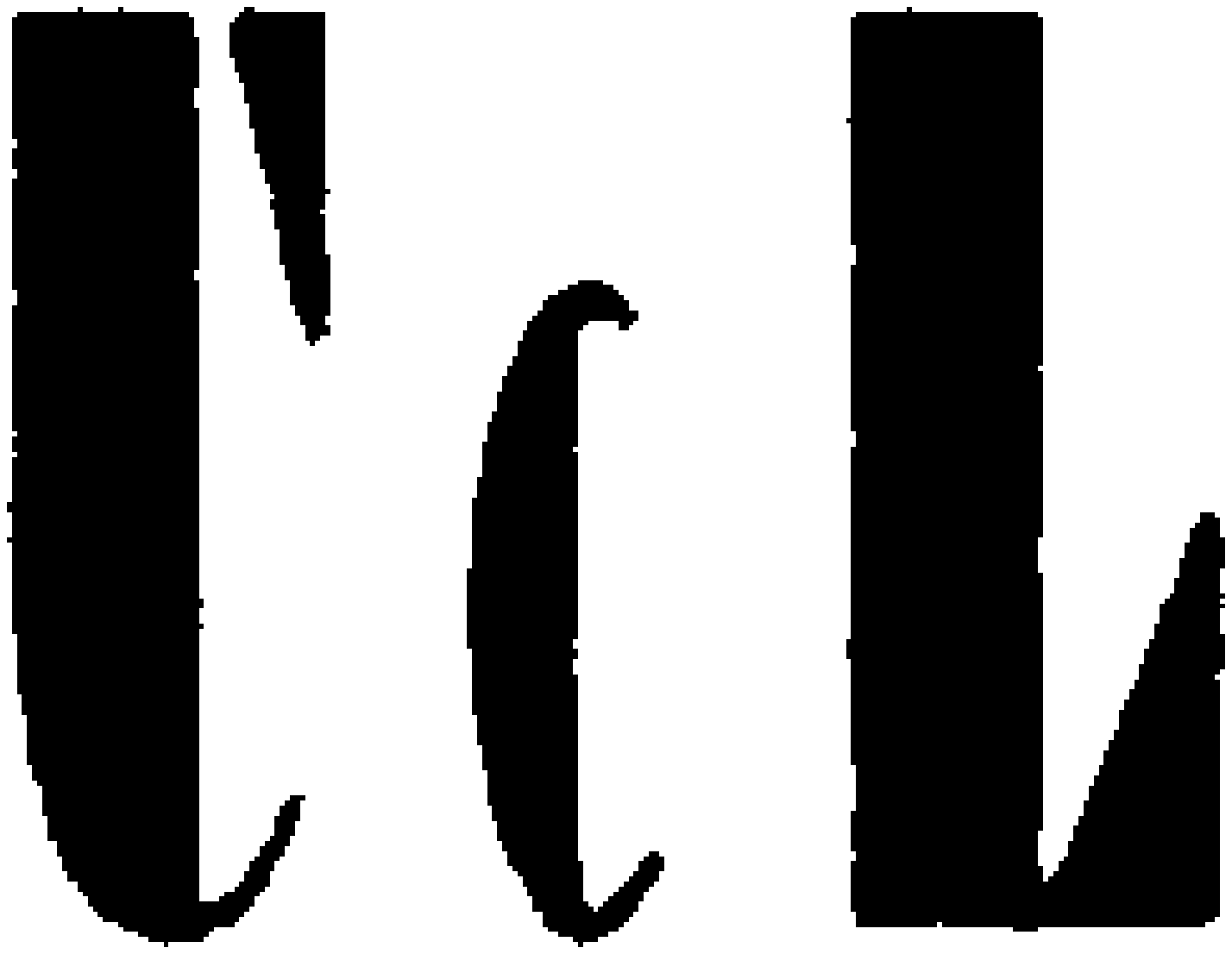}
{(f) \scriptsize{HTVWM}}
\end{minipage}\\
\begin{minipage}[htbp]{0.15\linewidth}
\centering
\includegraphics[width=0.45in]{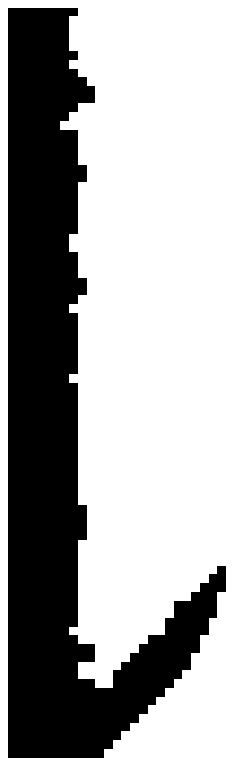}
{\hspace{10pt}(a1) \scriptsize{CVM}}
\end{minipage}
\begin{minipage}[htbp]{0.15\linewidth}
\centering
\includegraphics[width=0.45in]{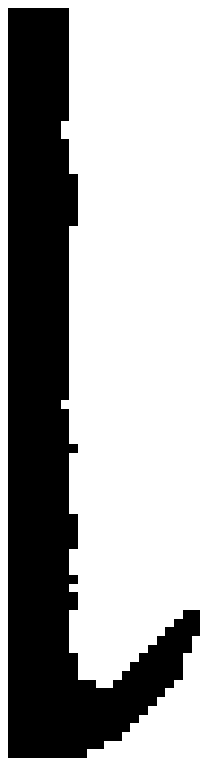}
{\hspace{10pt}(b1) \scriptsize{MFPM}}
\end{minipage}
\begin{minipage}[htbp]{0.15\linewidth}
\centering
\includegraphics[width=0.45in]{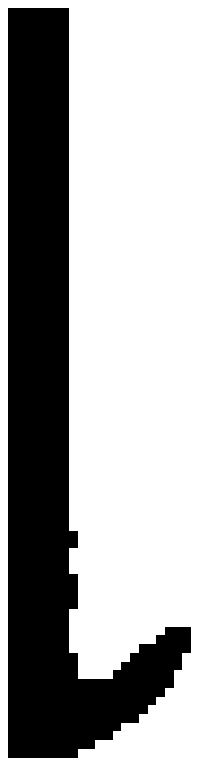}
{\hspace{10pt}(c1)  \scriptsize{PDPM}}
\end{minipage}
\begin{minipage}[htbp]{0.15\linewidth}
\centering
\includegraphics[width=0.45in]{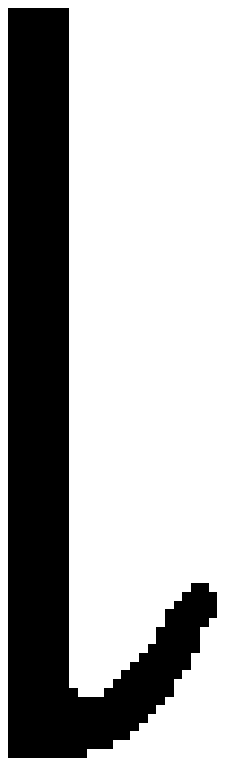}
{\hspace{10pt}(d1) \scriptsize{TSMSM}}
\end{minipage}
\begin{minipage}[htbp]{0.15\linewidth}
\centering
\includegraphics[width=0.45in]{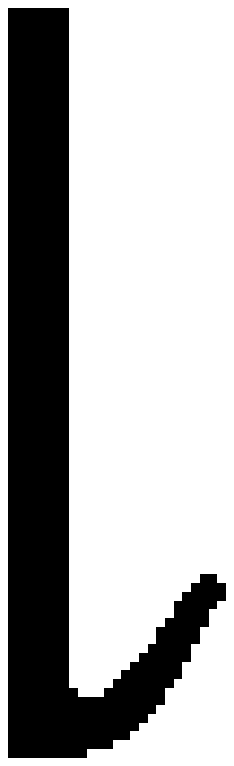}
{\hspace{10pt}(e1) \scriptsize{HTVUM}}
\end{minipage}
\begin{minipage}[htbp]{0.15\linewidth}
\centering
\includegraphics[width=0.45in]{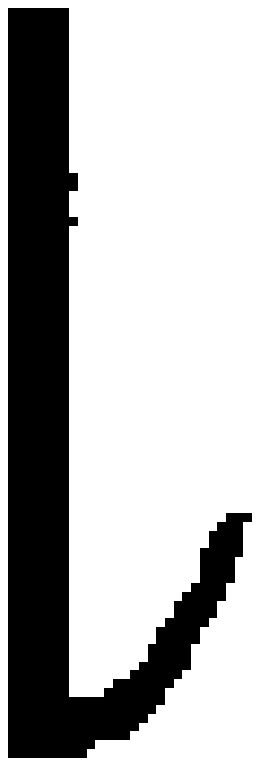}
{\hspace{10pt}(f1) \scriptsize{HTVWM}}
\end{minipage}
\caption{\label{fig43}Comparison of segmentation results  by using differential models for the noisy two-phase image in Example \ref{ex41}.  (a): The noisy image; (a1): The magnified portion of the original image; Row 2 (b1)-(f1): the magnified portion of the images corresponding to Row 1 (b-)-(f). Parameters: (b): $\gamma=200$; (c): $\gamma=0.15$; (d): $\gamma=11$ and $\lambda=0.001$; (e): $\gamma=21.5$ and $\lambda=0.02$; (f): $\gamma=26$ and $\lambda=0.8$.}
\end{figure}
\end{example}

\begin{example}\label{ex42}
The objective of this example is to illustrate the influence of the contaminated level on the segmentation results. We do not investigate the explicit relationship between all levels and the segmentation accuracy in this work. Instead, we empirically selected three class of contamination to show the influence of  segmentation.
As we can see from Example \ref{ex41} that the CVM \cite{5}, the MFPM \cite{130,131} and the PDPM \cite{18}  can not efficiently segment the degraded image, so we only consider other three segmentation schemes based on the two-step strategy.  Table \ref{tab42} demonstrates the chosen parameter $\gamma$ by fixing the other parameter $\lambda$ in models and the segmentation accuracy (SA) by using Figure \ref{fig40}(b) as the testing image. Our proposed constrained model can obviously improve segmentation results upon the HTVUM and the TSMSM especially at high noise levels and blurring effects. Actually, it is due to that the information of edges and constraint in our proposed model can efficiently suppress  noise and preserve edges. Simultaneously, we find out that the parameter $\gamma$ increases with the degraded level. Actually, this is because of that we need to increase the weighted values for penalising the regularization terms in models when the image contamination increases.
\begin{table}[htbp]\label{tab42}
\centering
\begin{tabular}{|c||c|c||c|c||c|c|}
\hline {} &\multicolumn{2}{c||}{$\sigma=0.005$} &\multicolumn{2}{c||}{$\sigma=0.01$}&\multicolumn{2}{c|} {$\sigma=0.02$} \\
\cline{2-7} \raisebox{0.9ex}{Variables}&$\gamma/\lambda$&$SA(\%)$&$\gamma/\lambda$&$SA(\%)$&$\gamma/\lambda$& $SA(\%)$ \\
\hline TSMSM&18/0.005&0.79&11/0.005&0.13&6.0/0.005&\textbf{0.29}\\
\hline HTVUM&17/0.5&0.87&9.0/0.5&0.14&6.2/0.5&0.33\\
\hline HTVWM&7.0/0.35&\textbf{0.78}&4.0/0.35&\textbf{0.13}&2.3/0.35&\textbf{0.29}\\
\hline {} &\multicolumn{2}{c||}{$(G,5,5)$/$\sigma=0.01$} &\multicolumn{2}{c||}{$(G,7,7)$/$\sigma=0.02$}&\multicolumn{2}{c|} {$(G,9,9)$/$\sigma=0.03$} \\
\cline{2-7} \raisebox{0.9ex}{Variables}&$\gamma/\lambda$&$SA(\%)$&$\gamma/\lambda$&$SA(\%)$&$\gamma/\lambda$& $SA(\%)$ \\
\hline TSMSM&13/0.02&0.55&8.9/0.02&0.92&4.5/0.02&1.30\\
\hline HTVUM&41/0.15&0.77&34/0.15&1.29&24/0.15&1.70\\
\hline HTVWM&13/0.6&\textbf{0.54}&7.0/0.6&\textbf{0.87}&3.5/0.6&\textbf{1.14}\\
\hline {}&\multicolumn{2}{c||}{$(M,5,10)$/$\sigma=0.01$} &\multicolumn{2}{c||}{$(M,10,20)$/$\gamma/\sigma=0.02$}&\multicolumn{2}{c|} {$(M,15,30)$/$\sigma=0.03$} \\
\cline{2-7} \raisebox{0.9ex}{Variables}&$\gamma/\lambda$&$SA(\%)$&$\gamma/\lambda$&$SA(\%)$&$\gamma/\lambda$& $SA(\%)$ \\
\hline TSMSM&11/0.02&\textbf{0.39}&5.0/0.02&1.0&2.8/0.02&1.24\\
\hline HTVUM&19/0.15&0.48&19/0.15&1.36&14/0.15&1.76\\
\hline HTVWM&12/0.6&\textbf{0.39}&5.0/0.6&\textbf{0.96}&3.0/0.6&\textbf{1.14}\\
\hline
\end{tabular}
\caption{The related data in Example \ref{ex42}.}
\end{table}
\end{example}
\subsection{Inhomogeneous image segmentation}
In this subsection, we extend our proposed method to segment two classes of inhomogeneous images shown in Figure \ref{fig45}. The left is a synthesized image by combining the arterial blood vessels of a human head.  The right is a MRI brain image based on an anatomical model of normal brain from the slice 91 of the normal brain database, which is available to the public at http://www.bic.mni.mcgill.ca/brainweb/. Here we set as ``modality=T1, Slice thickness=1mm, intensity non-uniformity = 20\%" for the original image \ref{fig45}(b). Different to the numerical comparisons of the piecewise constant image, we do not know the real segmentation due to the inhomogeneity. So we except to obtain a better restored image as the stopping condition in the first step of our proposed strategy.

\begin{figure}[h!]
\centering
\begin{minipage}[htbp]{0.35\linewidth}
\centering
\includegraphics[width=1.8in]{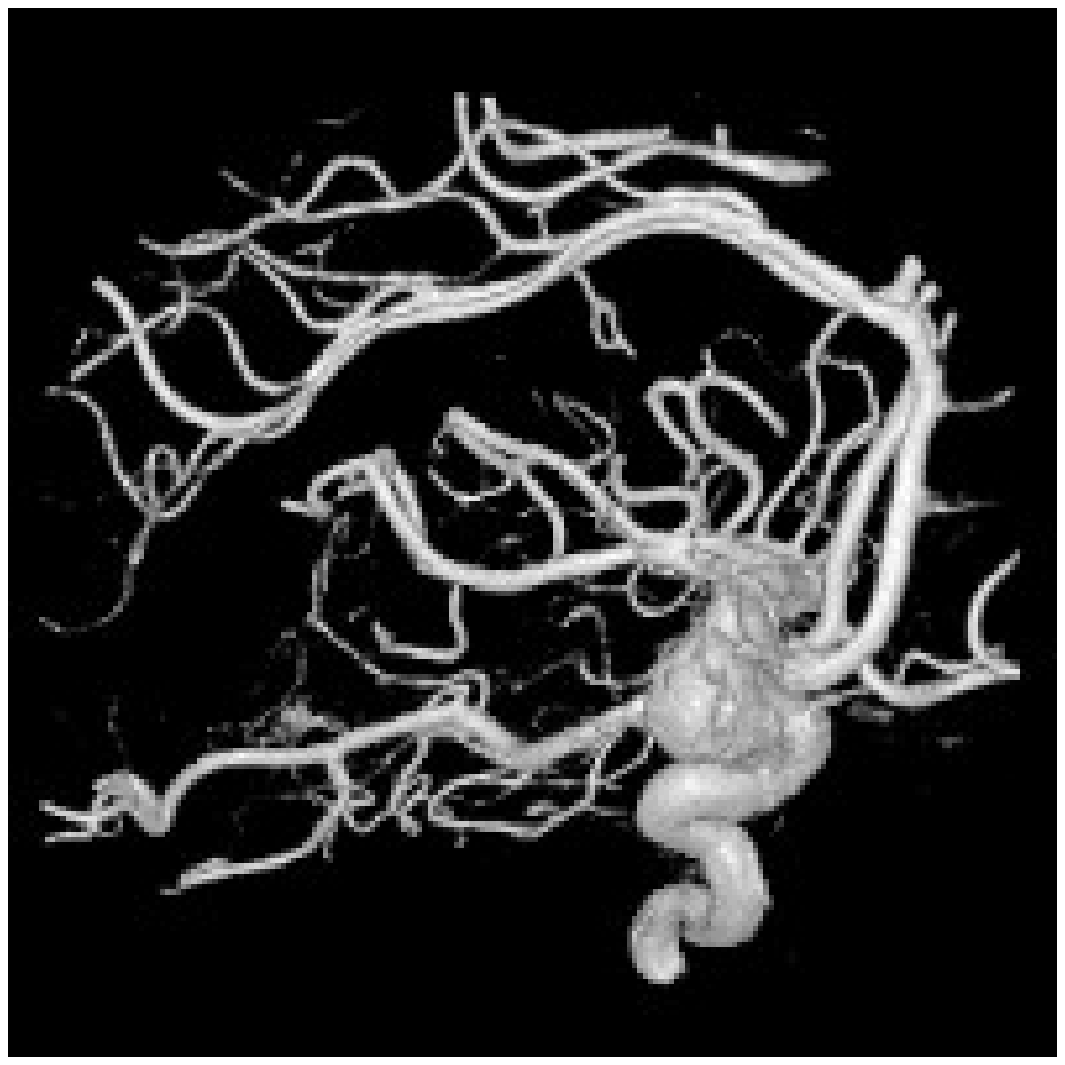}
{(a) Synthesis Image}
\end{minipage}
\begin{minipage}[htbp]{0.35\linewidth}
\centering
\includegraphics[width=1.5in]{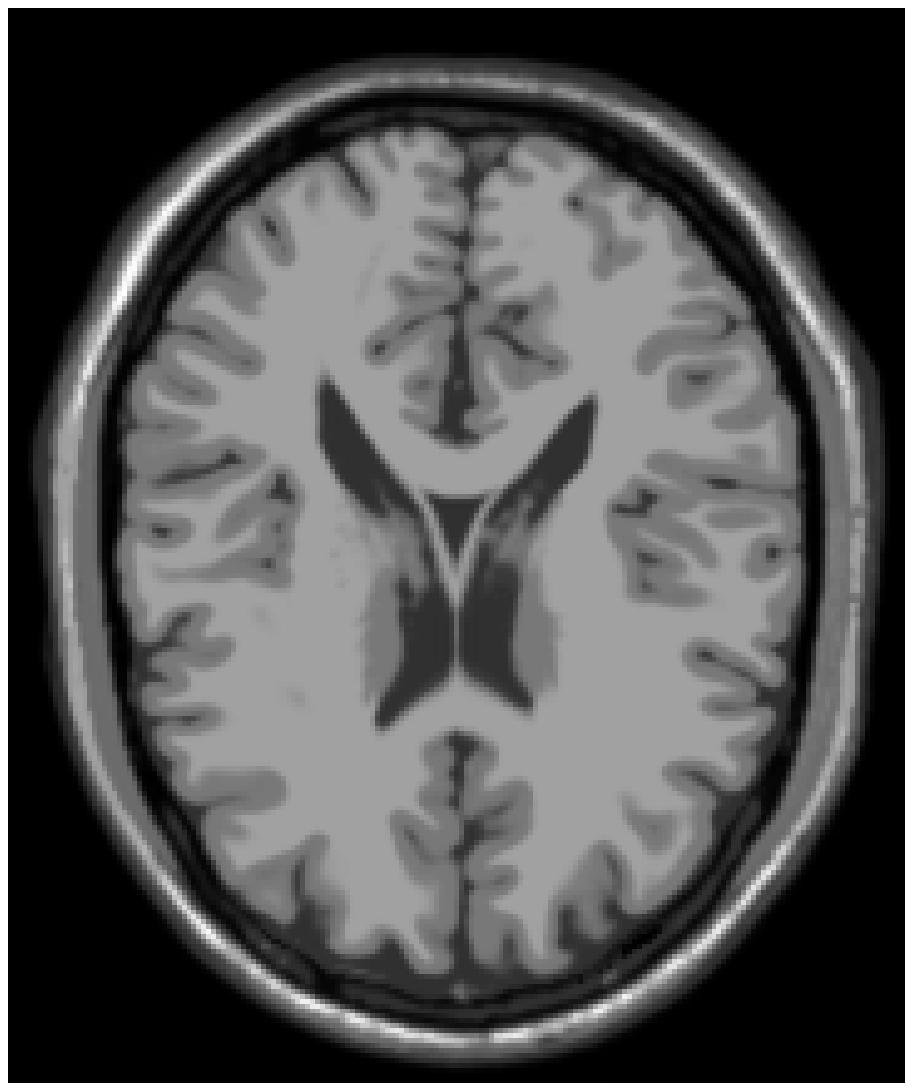}
{\hspace{10pt}(b) MRI Image}
\end{minipage}
\caption{\label{fig45}The original images in Example \ref{ex43} and \ref{ex44}.}
\end{figure}

\begin{example}\label{ex43}
Figure \ref{fig46} presents the applications of our proposed model to segment synthesis degraded images into two phases by using the different methods. It is clear that our proposed method can efficiently segment image, especially in the region of the endings of image. Actually, the proposed method (\ref{31}) can keep the information of edges due to the penalty of the weighted function $w(x)$.  The related parameters and data can be found in the caption of Figure \ref{fig46}.
\begin{figure}[h!]
\centering
\begin{minipage}[htbp]{0.225\linewidth}
\centering
\includegraphics[width=1.15in]{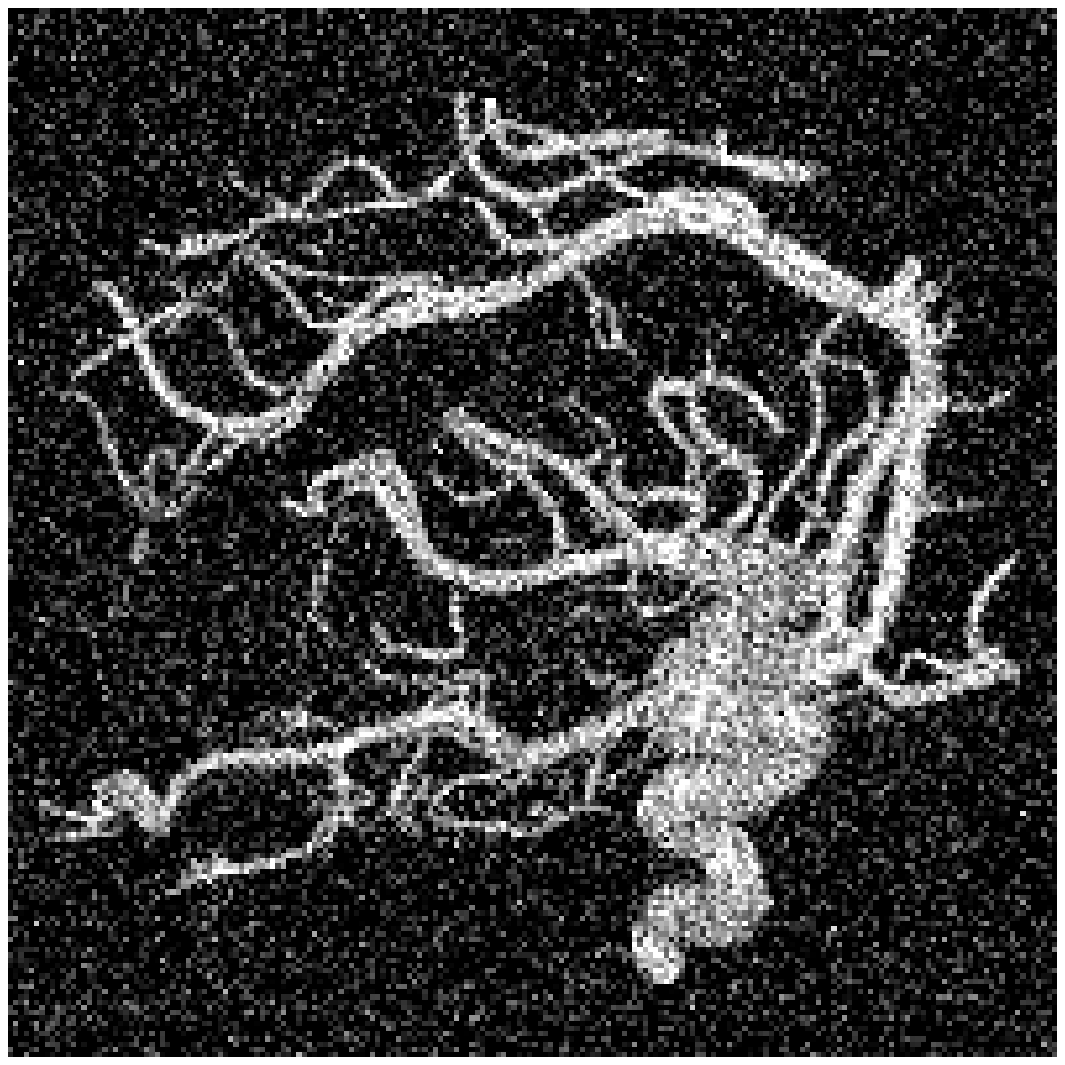}
{(a1)}
\end{minipage}
\begin{minipage}[htbp]{0.225\linewidth}
\centering
\includegraphics[width=1.15in]{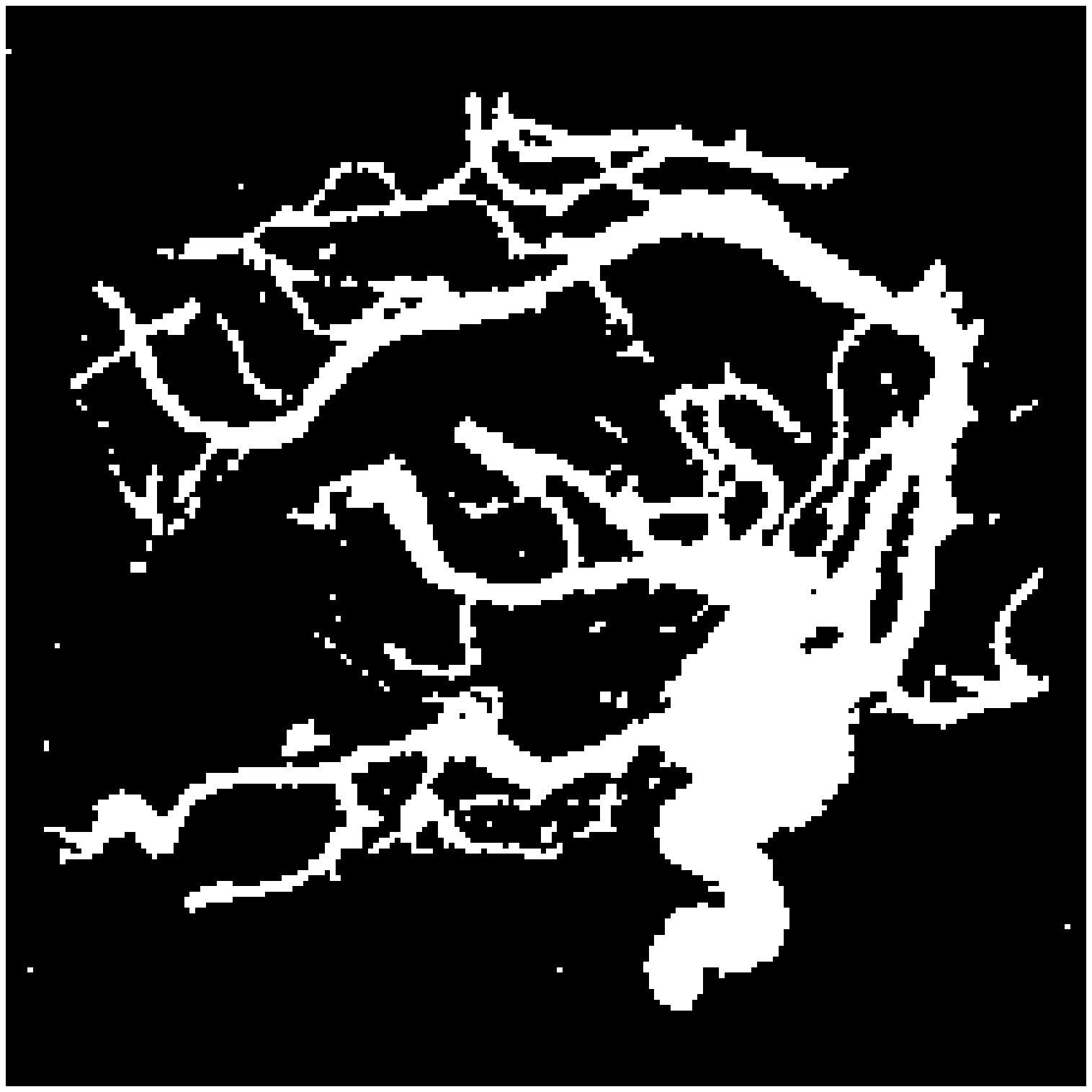}
{(b1) TSMSM}
\end{minipage}
\begin{minipage}[htbp]{0.225\linewidth}
\centering
\includegraphics[width=1.15in]{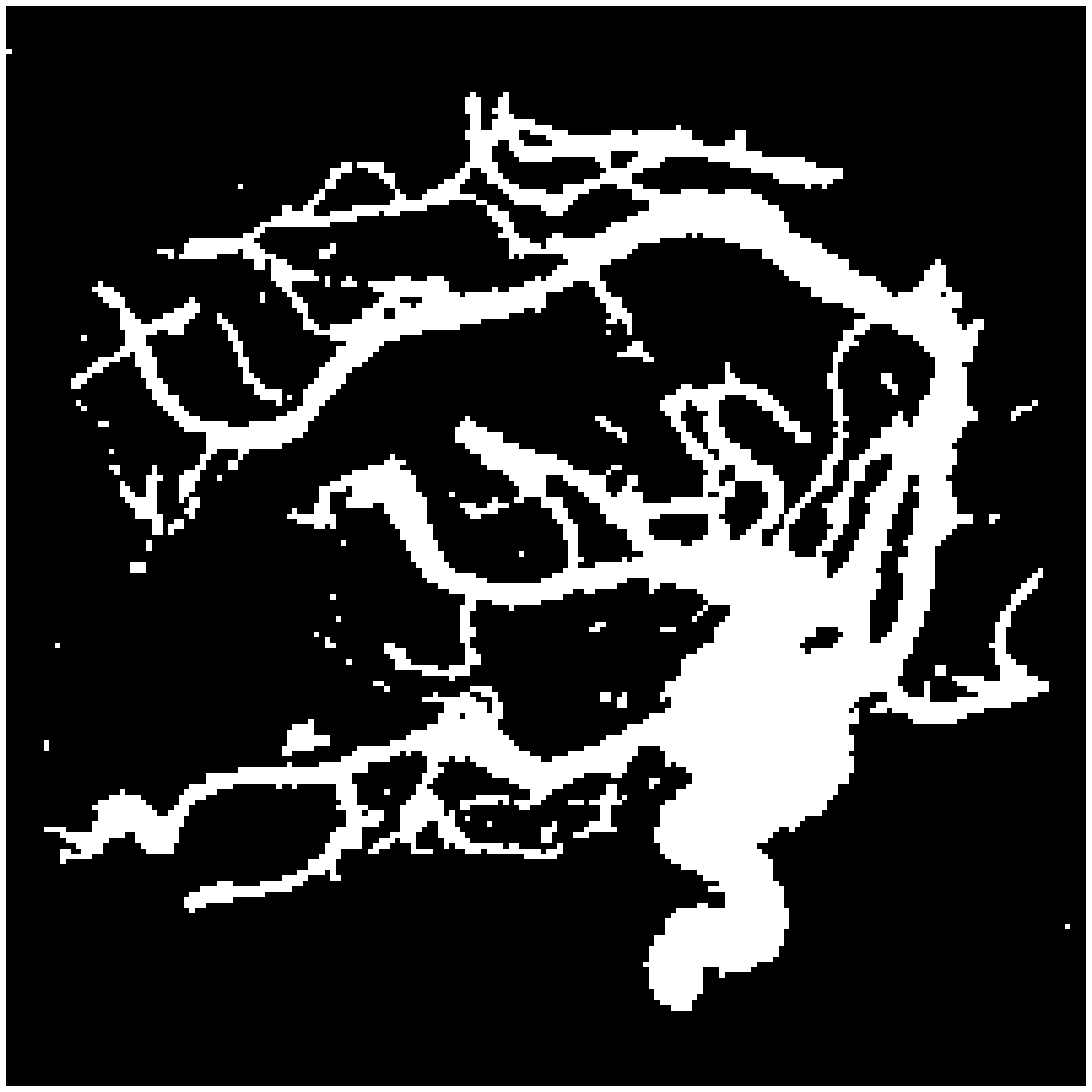}
{(c1) HEVUM}
\end{minipage}
\begin{minipage}[htbp]{0.225\linewidth}
\centering
\includegraphics[width=1.15in]{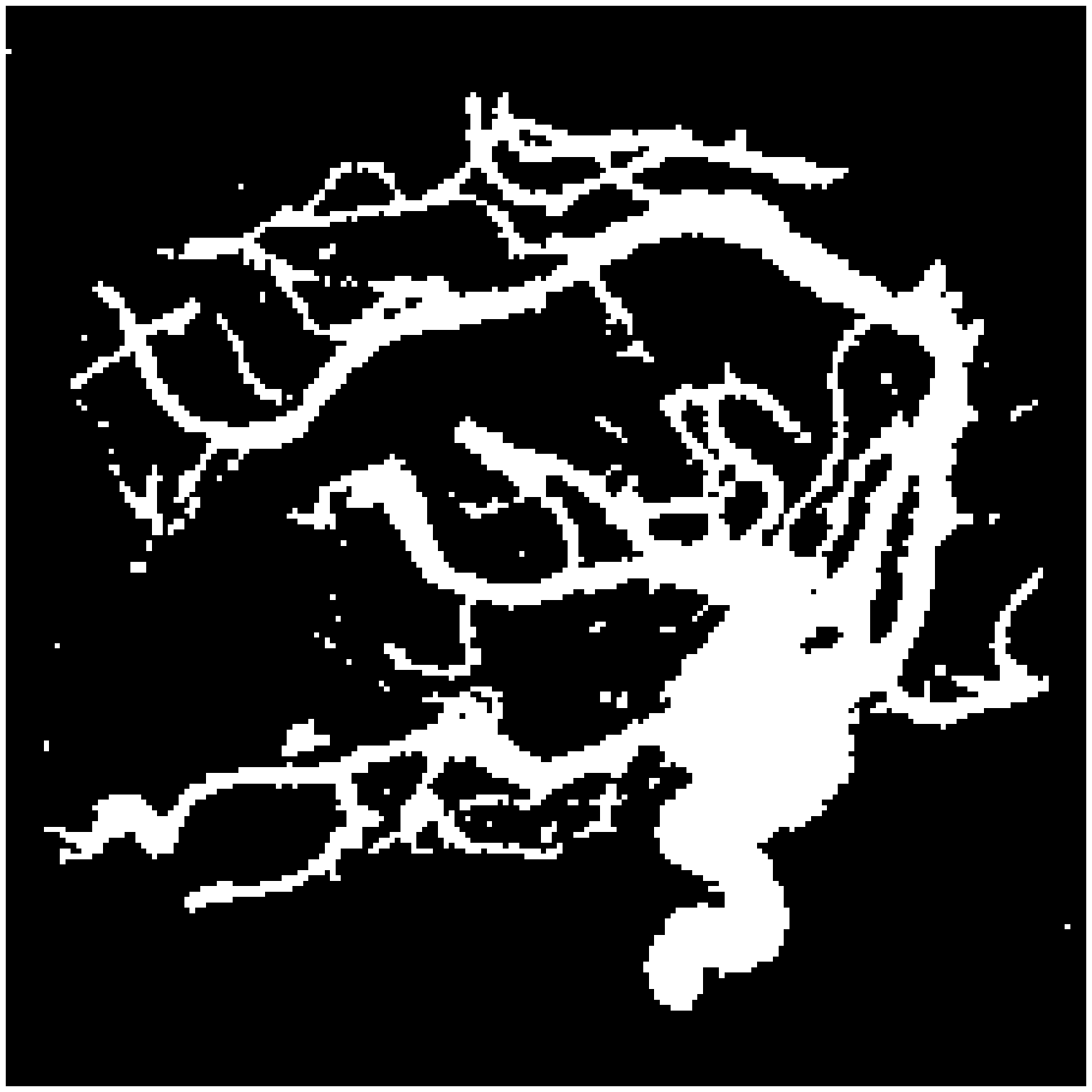}
{(d1) HTVWM}
\end{minipage}\\
\begin{minipage}[htbp]{0.225\linewidth}
\centering
\includegraphics[width=1.15in]{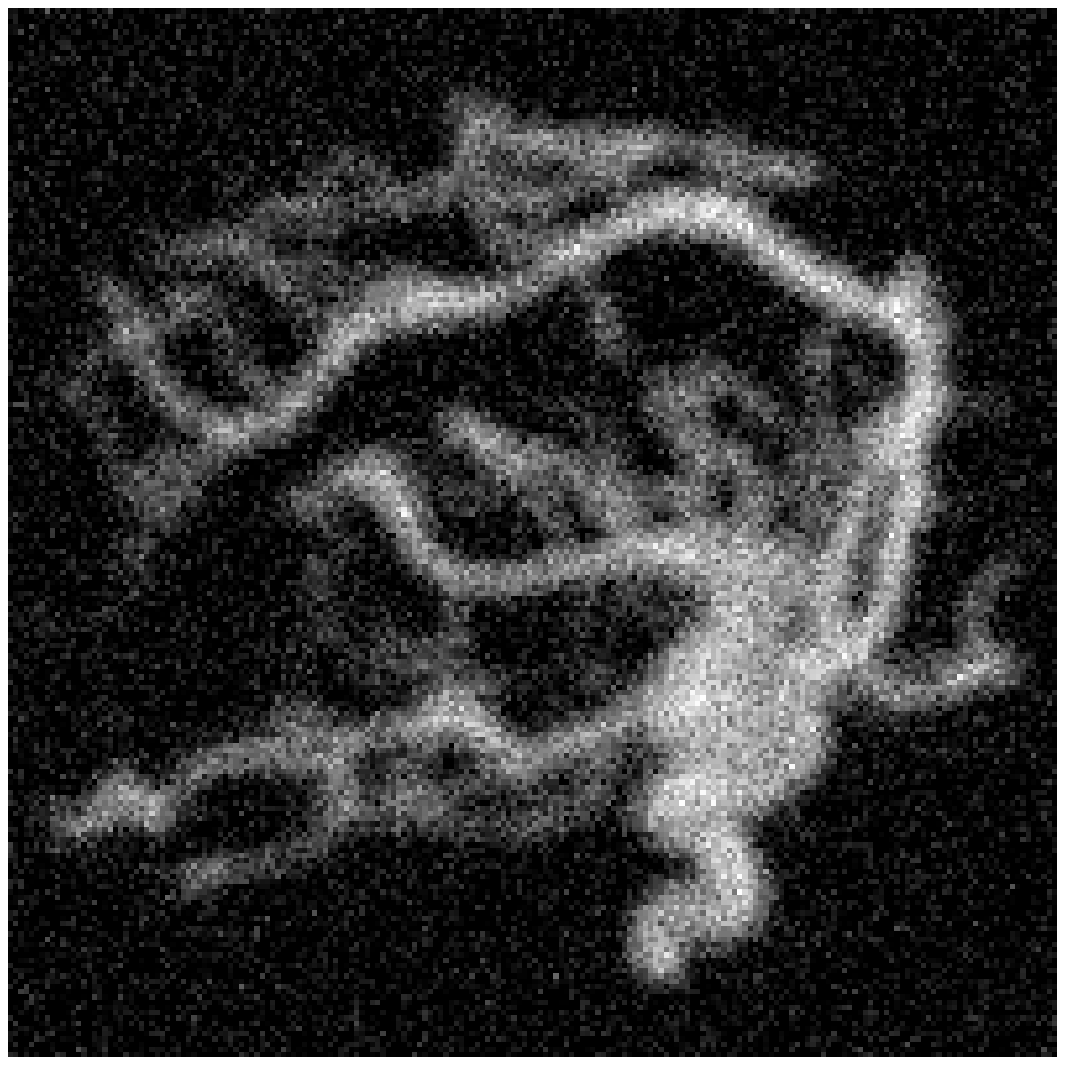}
{(a2) }
\end{minipage}
\begin{minipage}[htbp]{0.225\linewidth}
\centering
\includegraphics[width=1.15in]{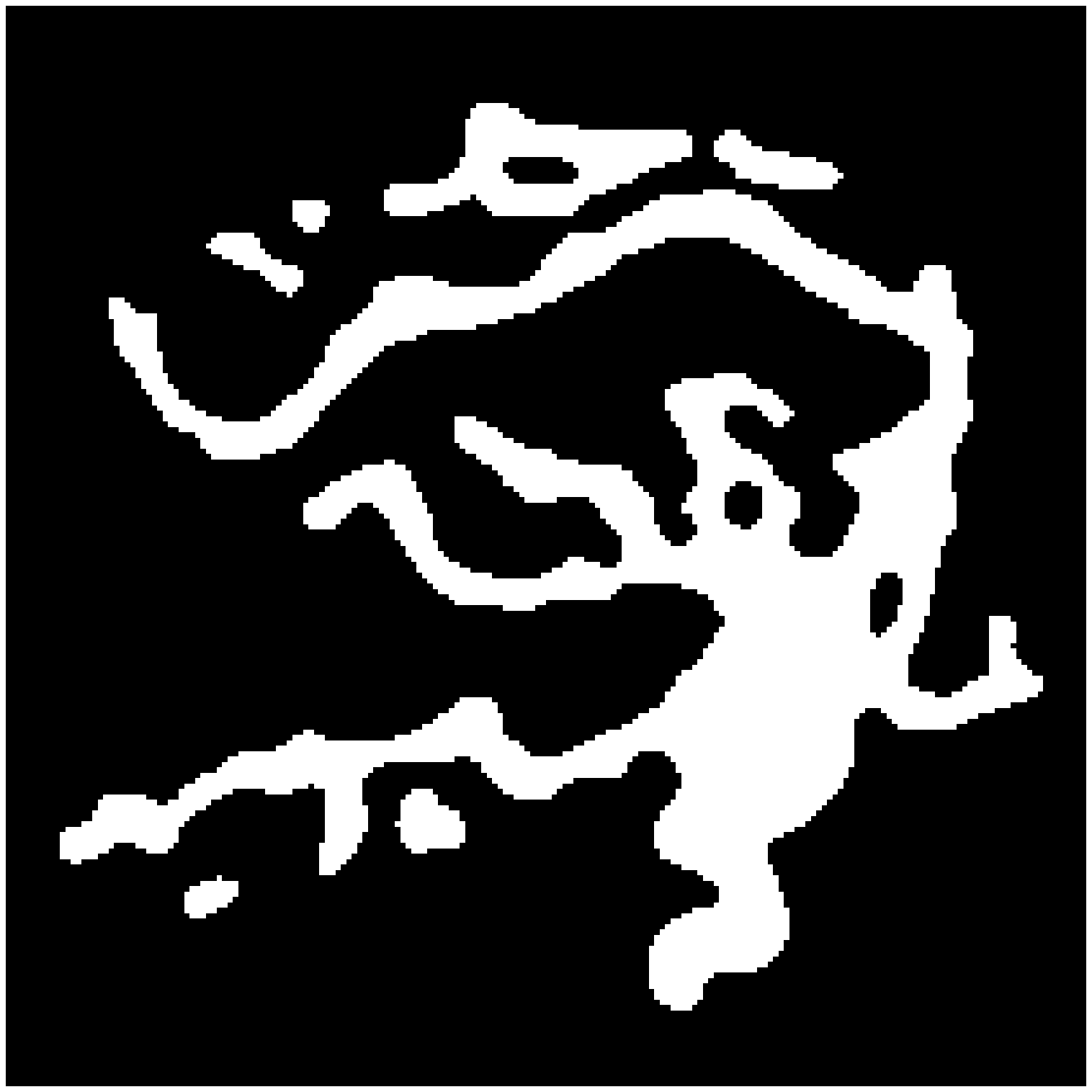}
{(b2) TSMSM}
\end{minipage}
\begin{minipage}[htbp]{0.225\linewidth}
\centering
\includegraphics[width=1.15in]{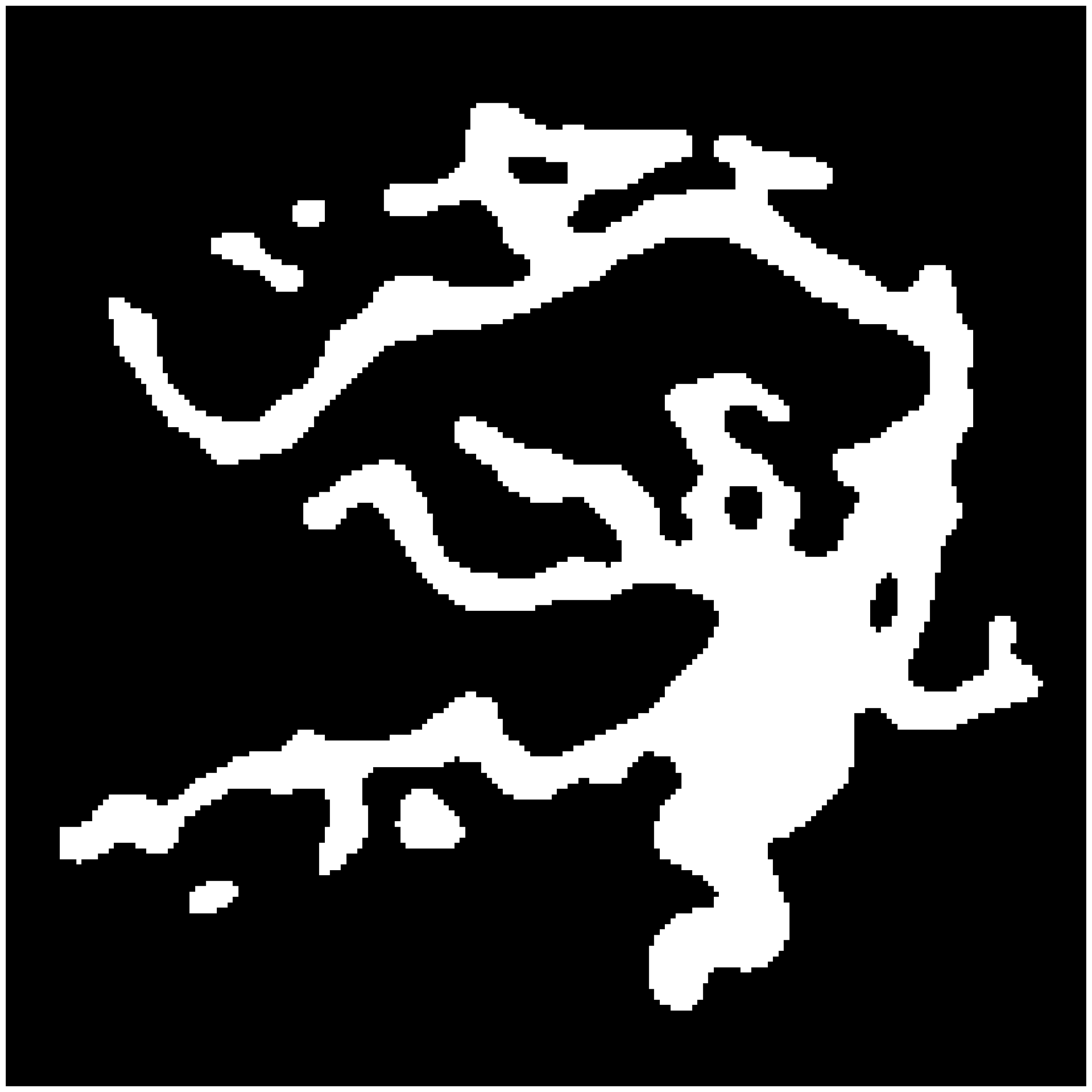}
{(c2) HTVUM}
\end{minipage}
\begin{minipage}[htbp]{0.225\linewidth}
\centering
\includegraphics[width=1.15in]{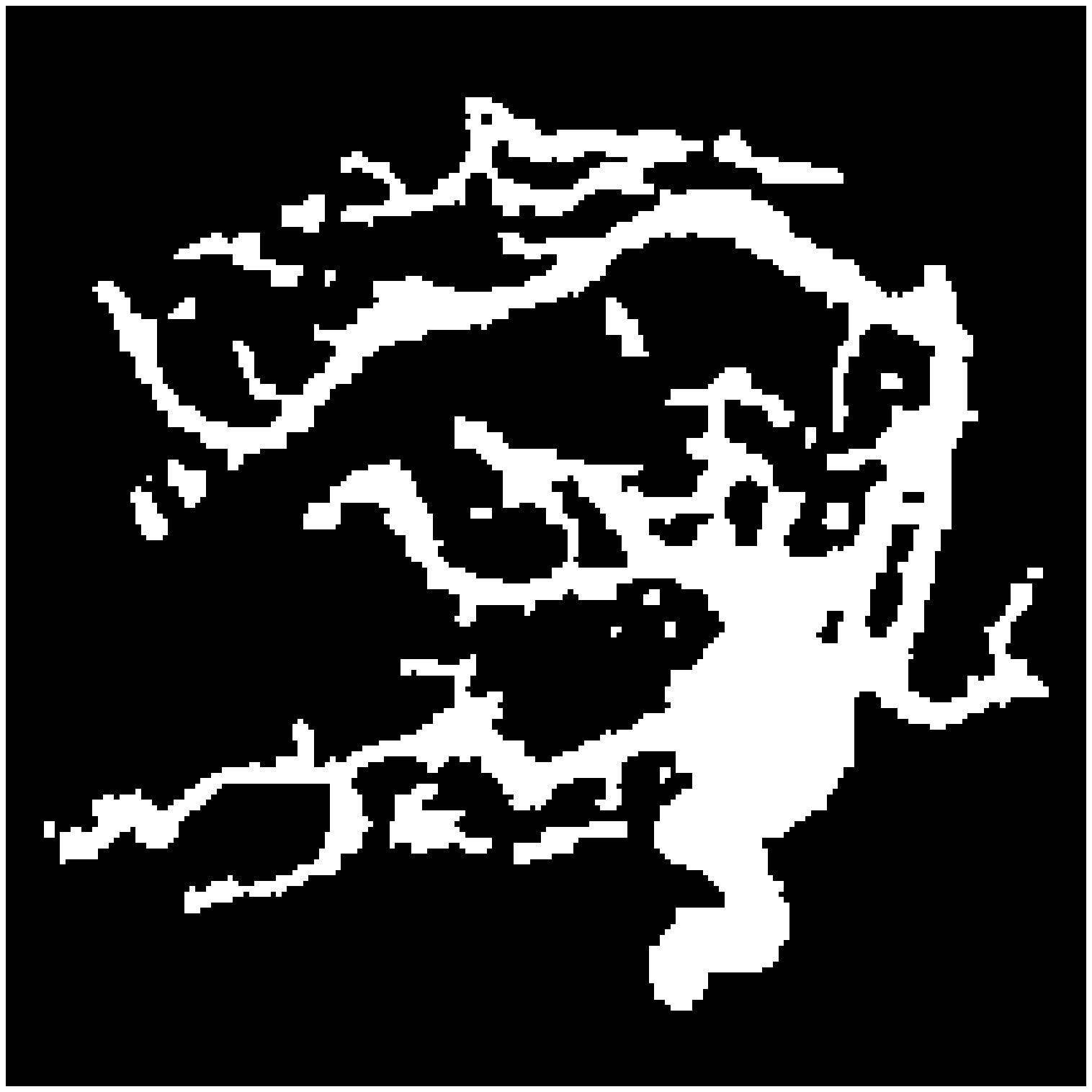}
{(d2) HTVWM}
\end{minipage}\\
\begin{minipage}[htbp]{0.225\linewidth}
\centering
\includegraphics[width=1.15in]{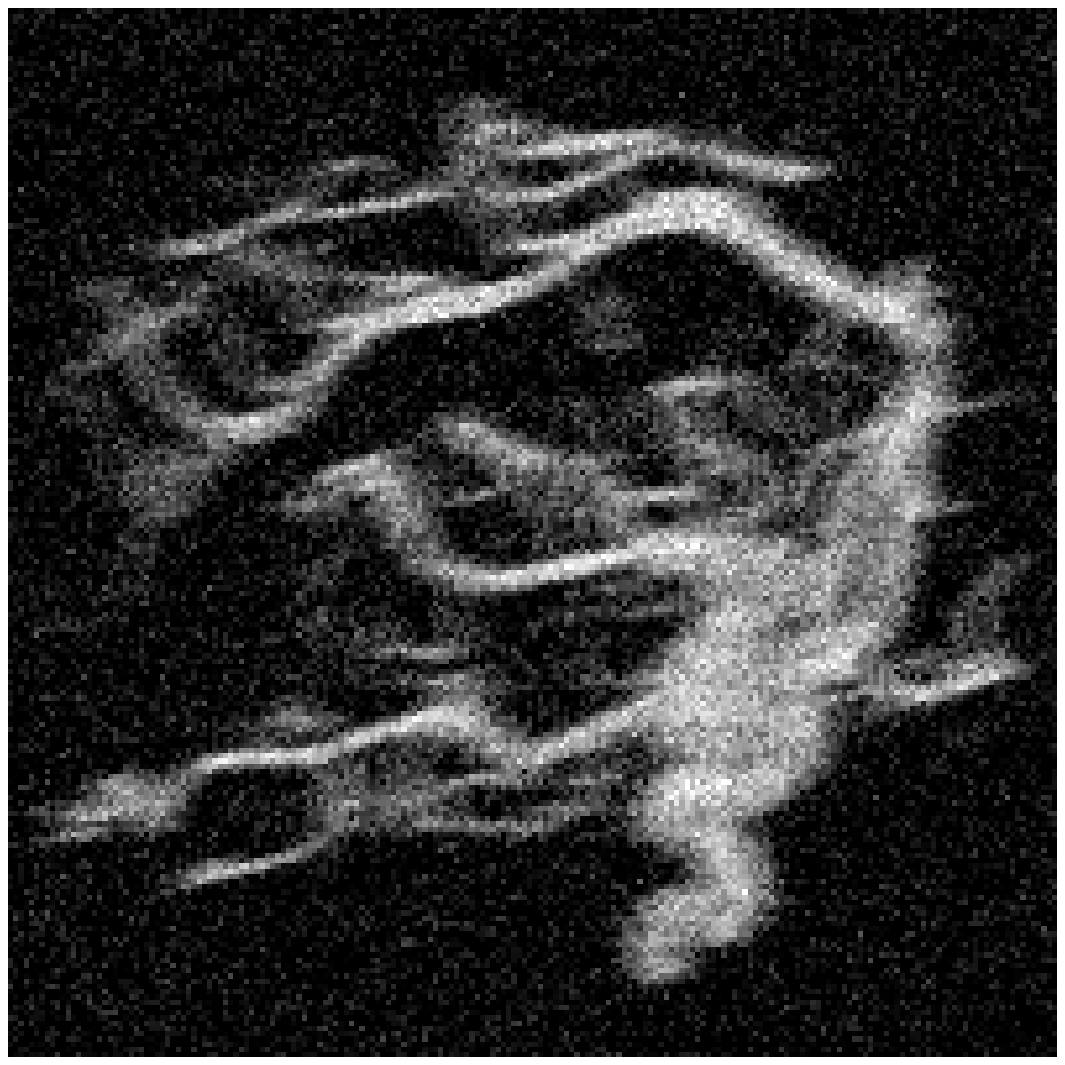}
{(a3) }
\end{minipage}
\begin{minipage}[htbp]{0.225\linewidth}
\centering
\includegraphics[width=1.15in]{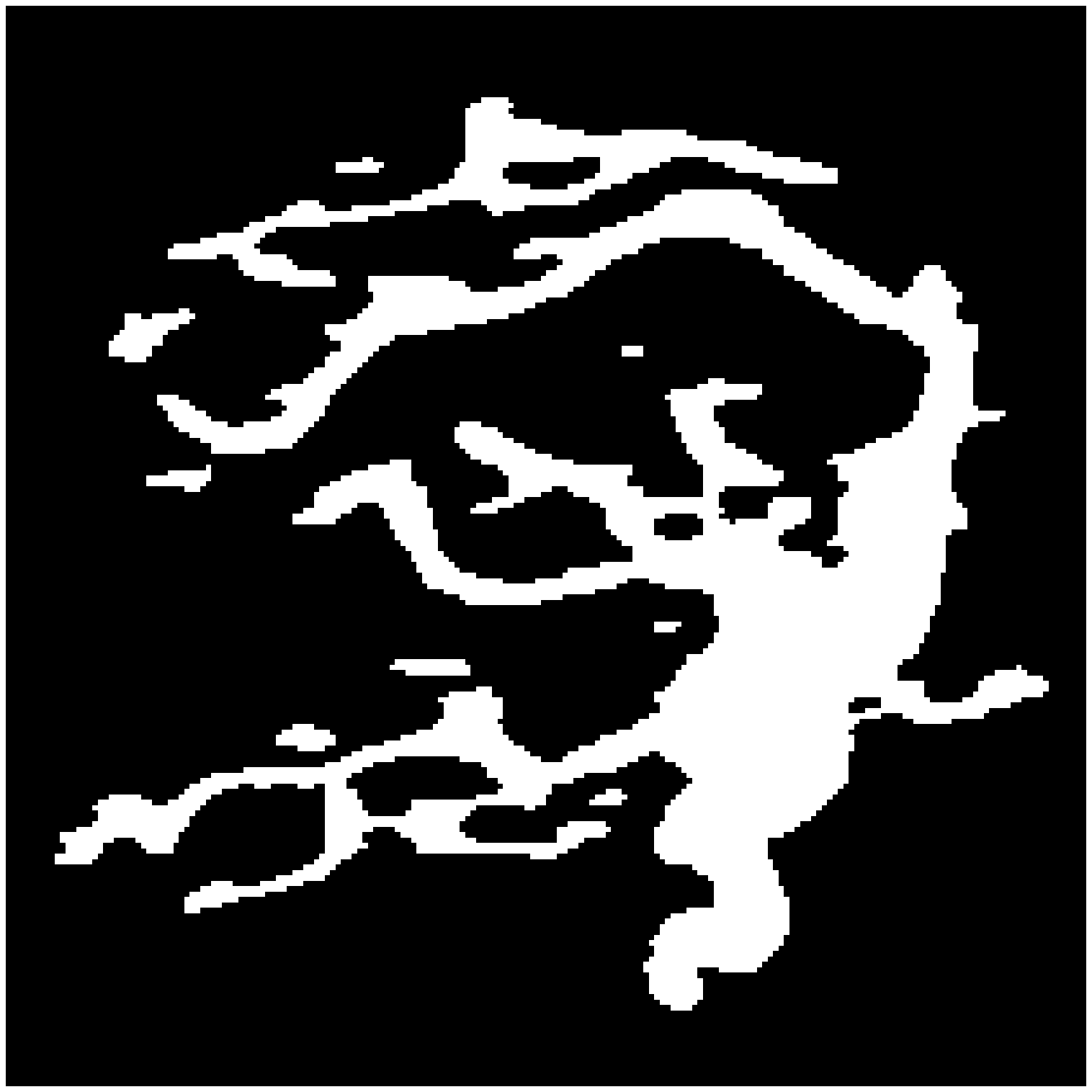}
{(b3) TSMSM}
\end{minipage}
\begin{minipage}[htbp]{0.225\linewidth}
\centering
\includegraphics[width=1.15in]{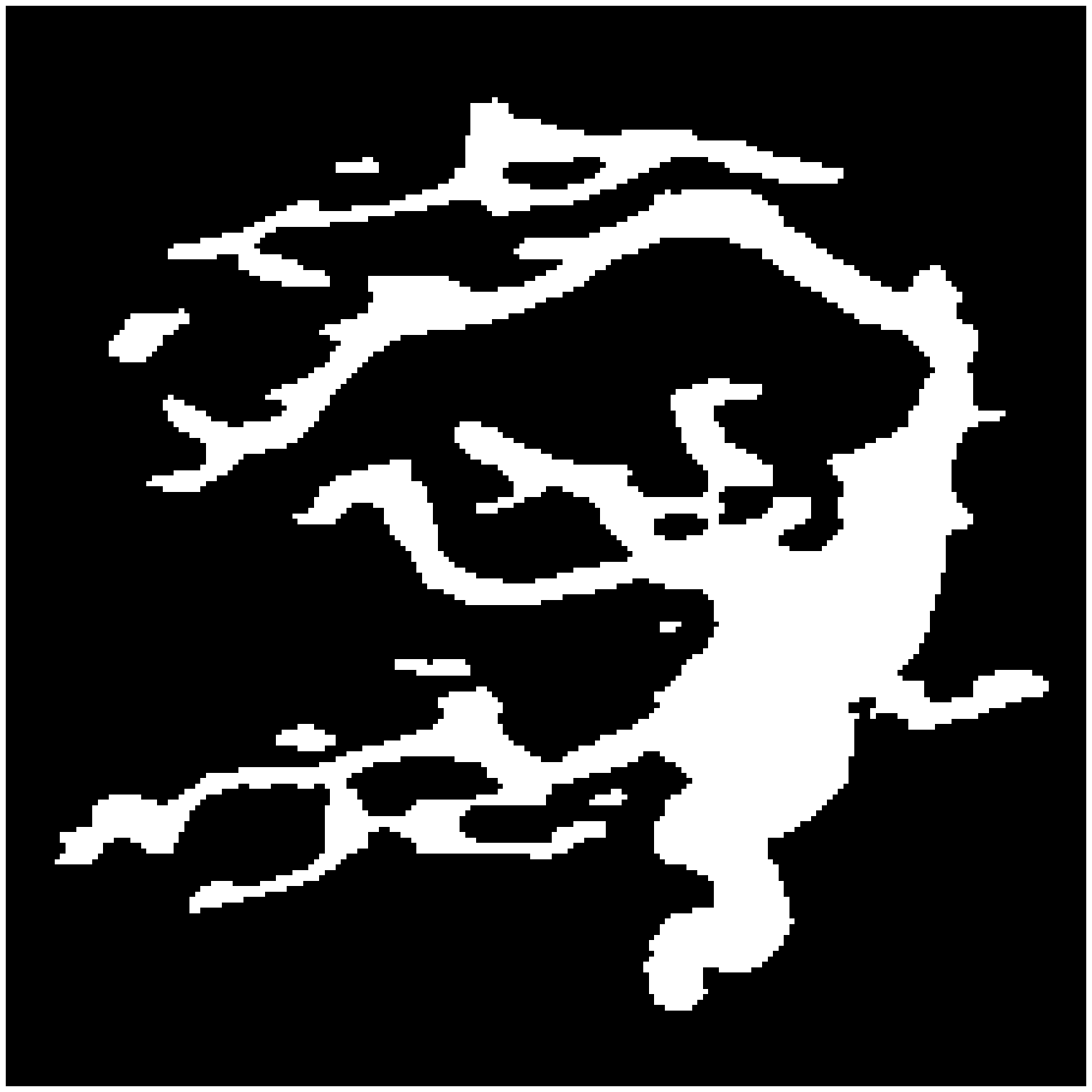}
{(c3) HTVUM}
\end{minipage}
\begin{minipage}[htbp]{0.225\linewidth}
\centering
\includegraphics[width=1.15in]{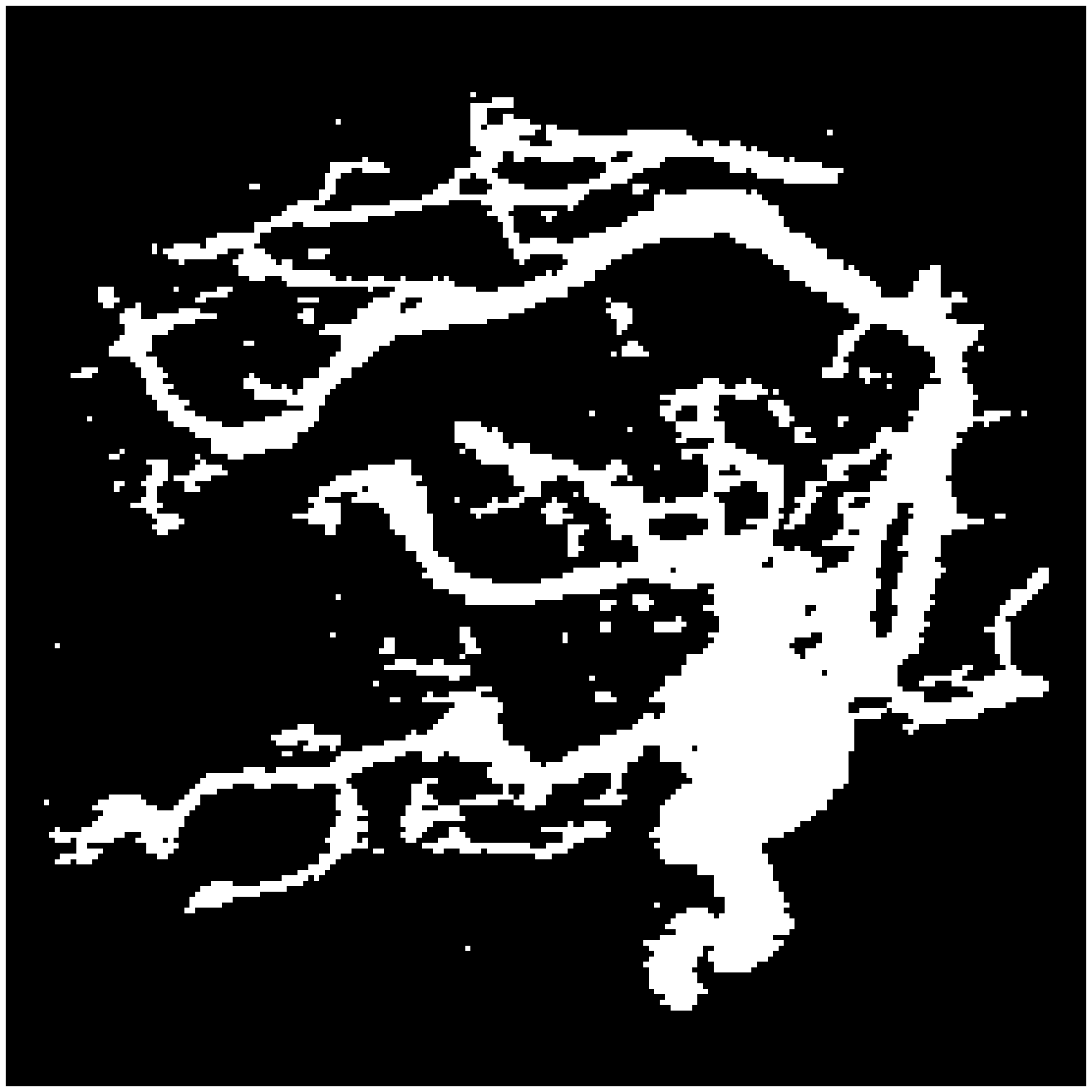}
{(d3) HTVWM}
\end{minipage}
\caption{\label{fig46}Comparison of segmentation results  by using different models for the noisy two-phase image. (a1)-(a3) Contaminated images: (a1) $\sigma=0.05$; (a2) $(G,7,7)$ and $\sigma=0.01$ ); (a3) $(M,10,10)$ and $\sigma=0.01$.  Related parameters in models. (b1) $\gamma=7.7$ and $\lambda=0.6$;  (c1) $\gamma=8$ and $\lambda=0.2$; (d1) $\gamma=7.6$ and $\lambda=0.2$;  (b2) $\gamma=20$ and $\lambda=20$;  (c2) $\gamma=17.3$ and $\lambda=0.003$; (d2) $\gamma=161$ and $\lambda=0.2$;  (b3) $\gamma=24$ and $\lambda=0.01$;  (c3) $\gamma=20$ and $\lambda=0.03$; (d3) $\gamma=147$ and $\lambda=0.0075$. }
\end{figure}
\end{example}

\begin{example}\label{ex44}
We consider our proposed method to segment a four-region  of the real brain MRI images shown in Figure \ref{fig45}, where the regions include  the grey matter, the white matter, the tumor-bone,  and the background region  respectively. The degraded images are shown in Figure \ref{fig470}.
\begin{figure}[h!]
\centering
\begin{minipage}[htbp]{0.30\linewidth}
\centering
\includegraphics[width=1.5in]{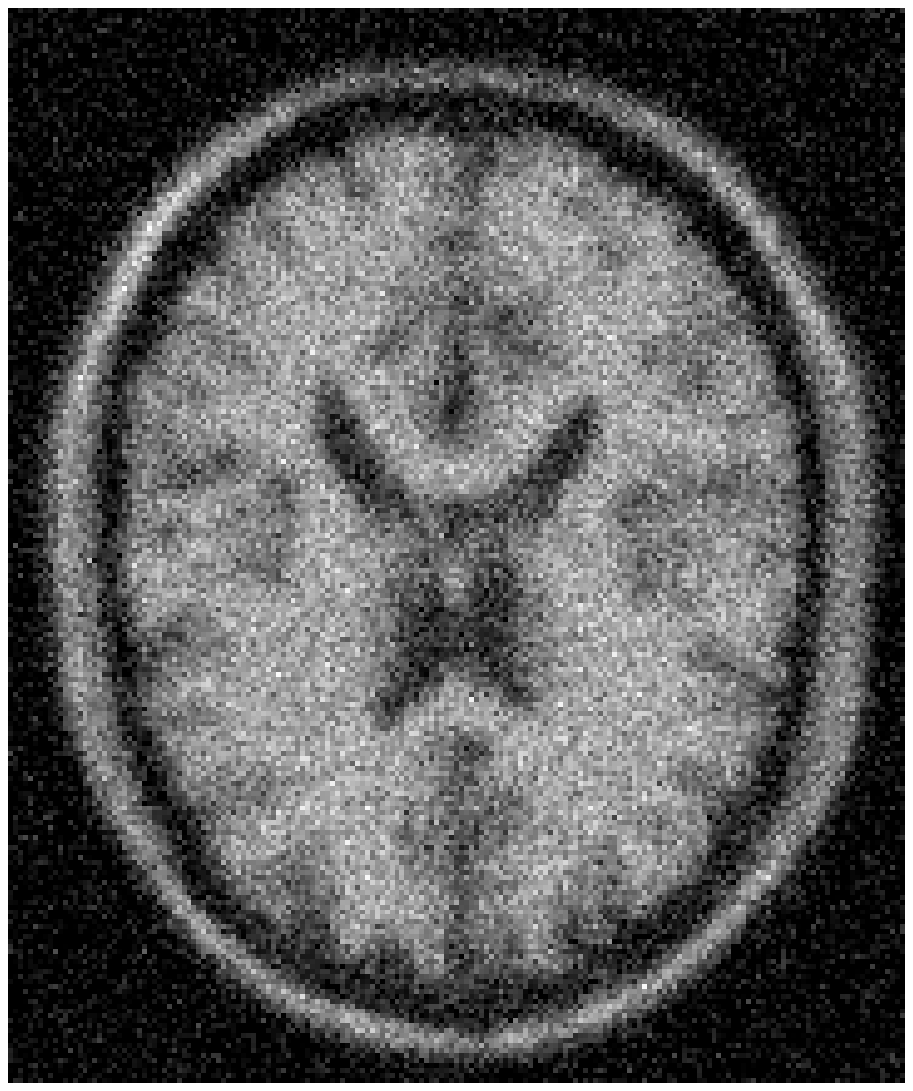}
{(a)Gaussion Blurring}
\end{minipage}
\begin{minipage}[htbp]{0.30\linewidth}
\centering
\includegraphics[width=1.5in]{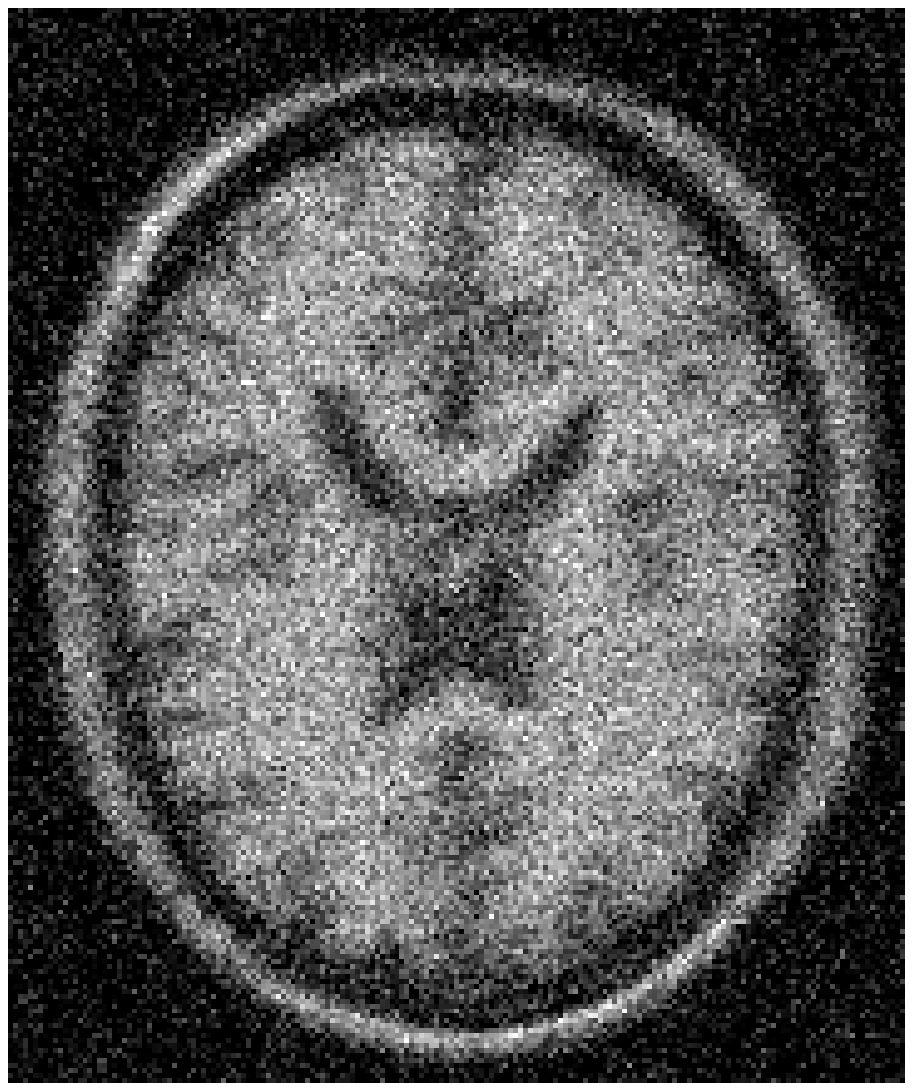}
{(b) Motion Blurring}
\end{minipage}
\caption{\label{fig470}The contaminated image in Example \ref{ex44}. (a) $(G,5,5)$ and $\sigma=0.01$; (b) $(M,7,15)$ and $\sigma=0.02$.}
\end{figure}
The first column of Figure \ref{fig47} and \ref{fig48} is the background showing as the white regions in the first column. Due to the little importance of the background, we essentially treat the problem as a 3-phase problem with the black background ordered in the second to forth column of all our results. The tumor region is ordered in the second column, the grey matter is ordered in the third column, and the white matter is ordered in the forth column. All the images  provide valuable information about the corresponding regions. Obviously,  the HTVWM can efficiently segment the tumor region and the white matter region quite well.

\begin{figure}[h!]
\centering
\begin{minipage}[htbp]{0.215\linewidth}
\centering
\includegraphics[width=1.05in]{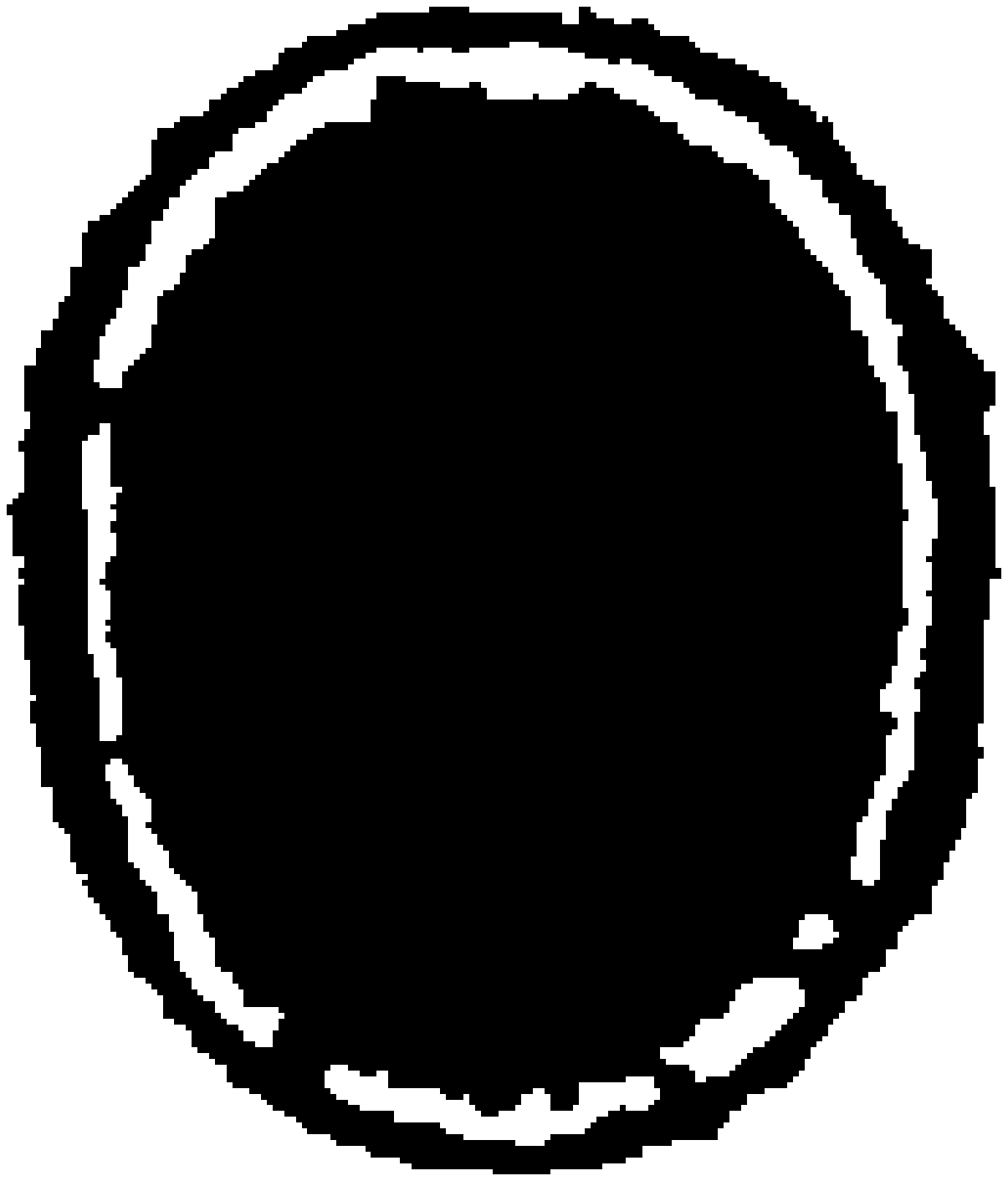}
{(a1) TSMSM}
\end{minipage}
\begin{minipage}[htbp]{0.215\linewidth}
\centering
\includegraphics[width=1.05in]{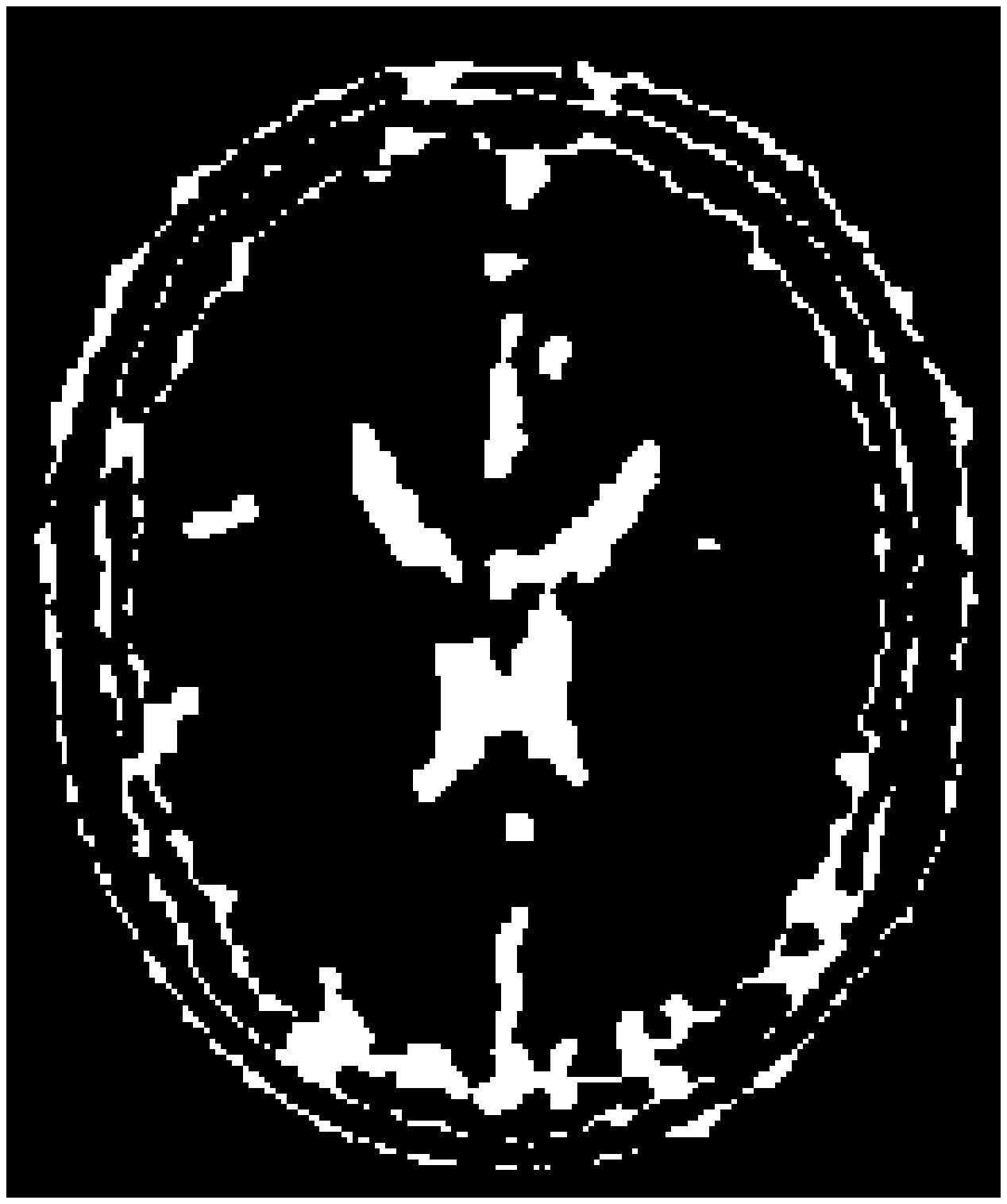}
{(a2) TSMSM}
\end{minipage}
\begin{minipage}[htbp]{0.215\linewidth}
\centering
\includegraphics[width=1.05in]{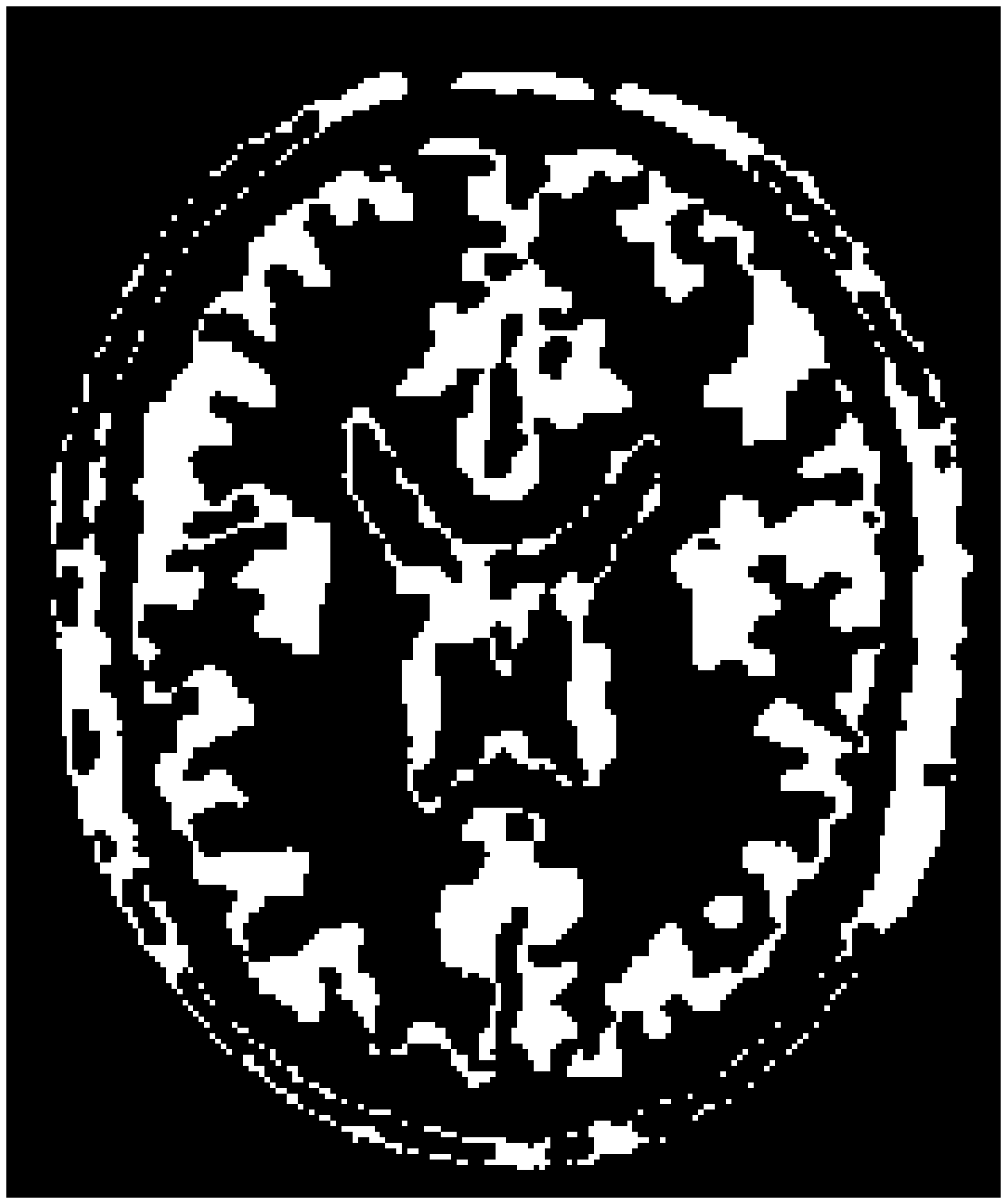}
{(a3)TSMSM}
\end{minipage}
\begin{minipage}[htbp]{0.215\linewidth}
\centering
\includegraphics[width=1.05in]{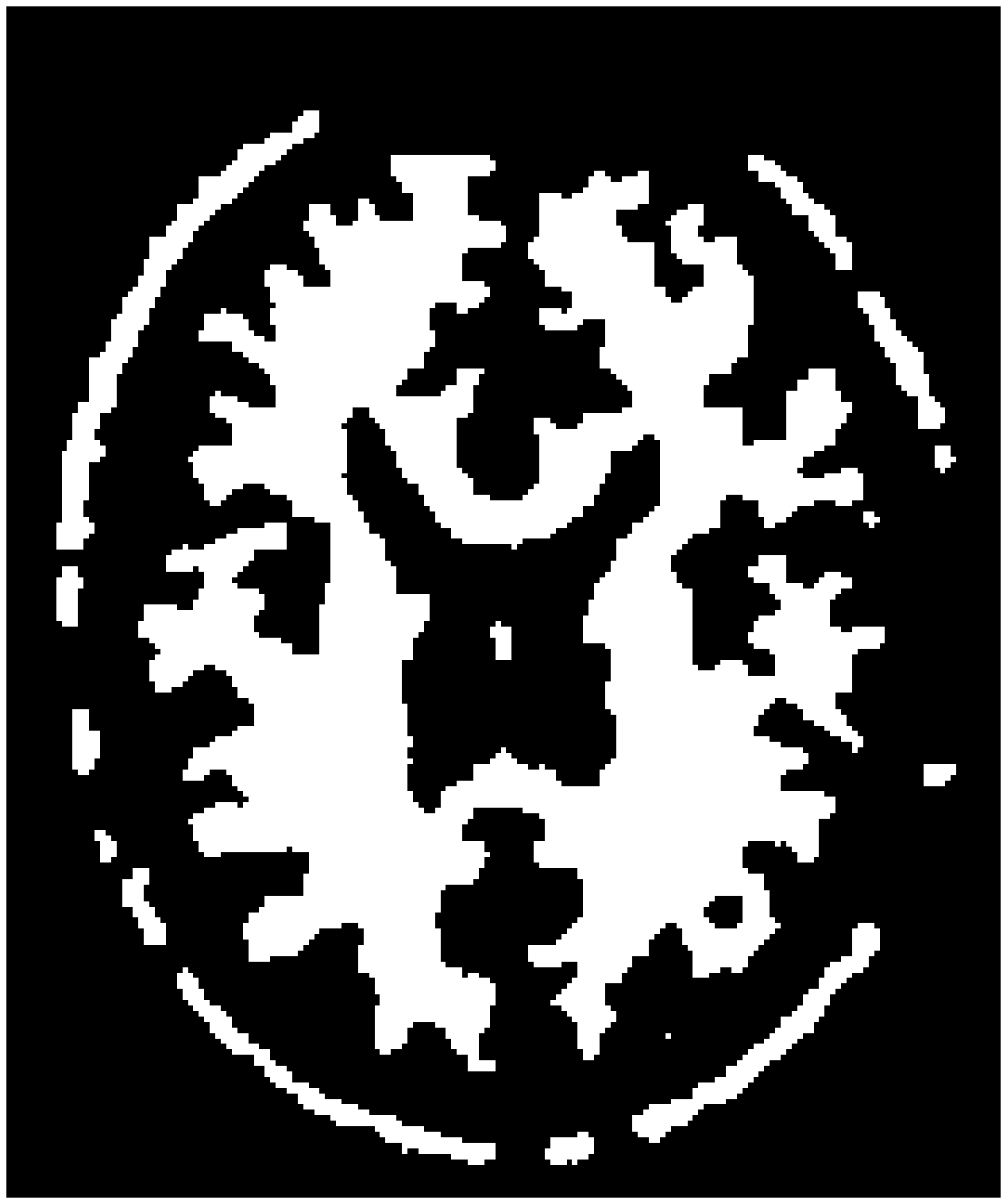}
{(a4) TSMSM}
\end{minipage}\\
\begin{minipage}[htbp]{0.215\linewidth}
\centering
\includegraphics[width=1.05in]{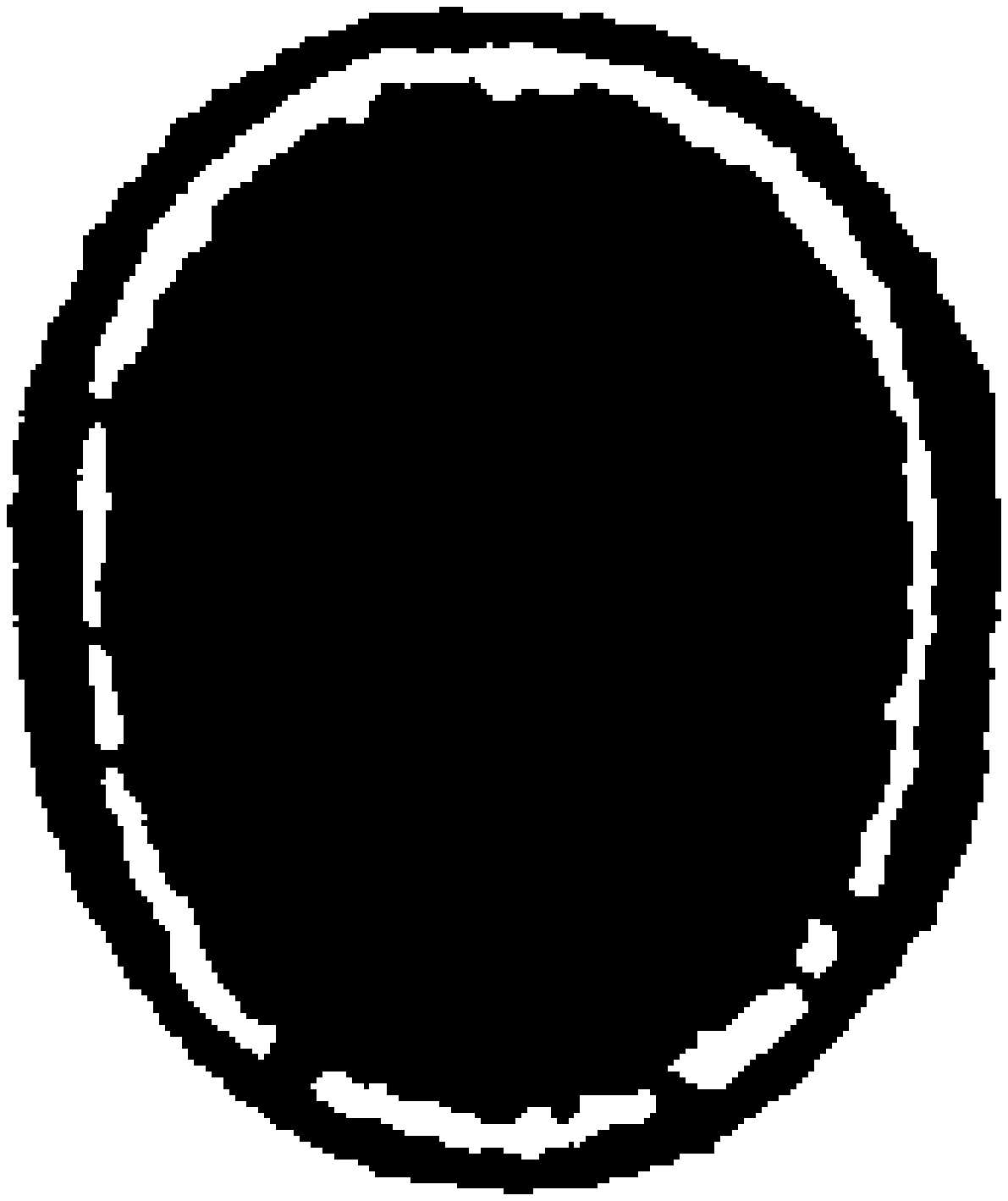}
{(b1) HTVUM}
\end{minipage}
\begin{minipage}[htbp]{0.215\linewidth}
\centering
\includegraphics[width=1.05in]{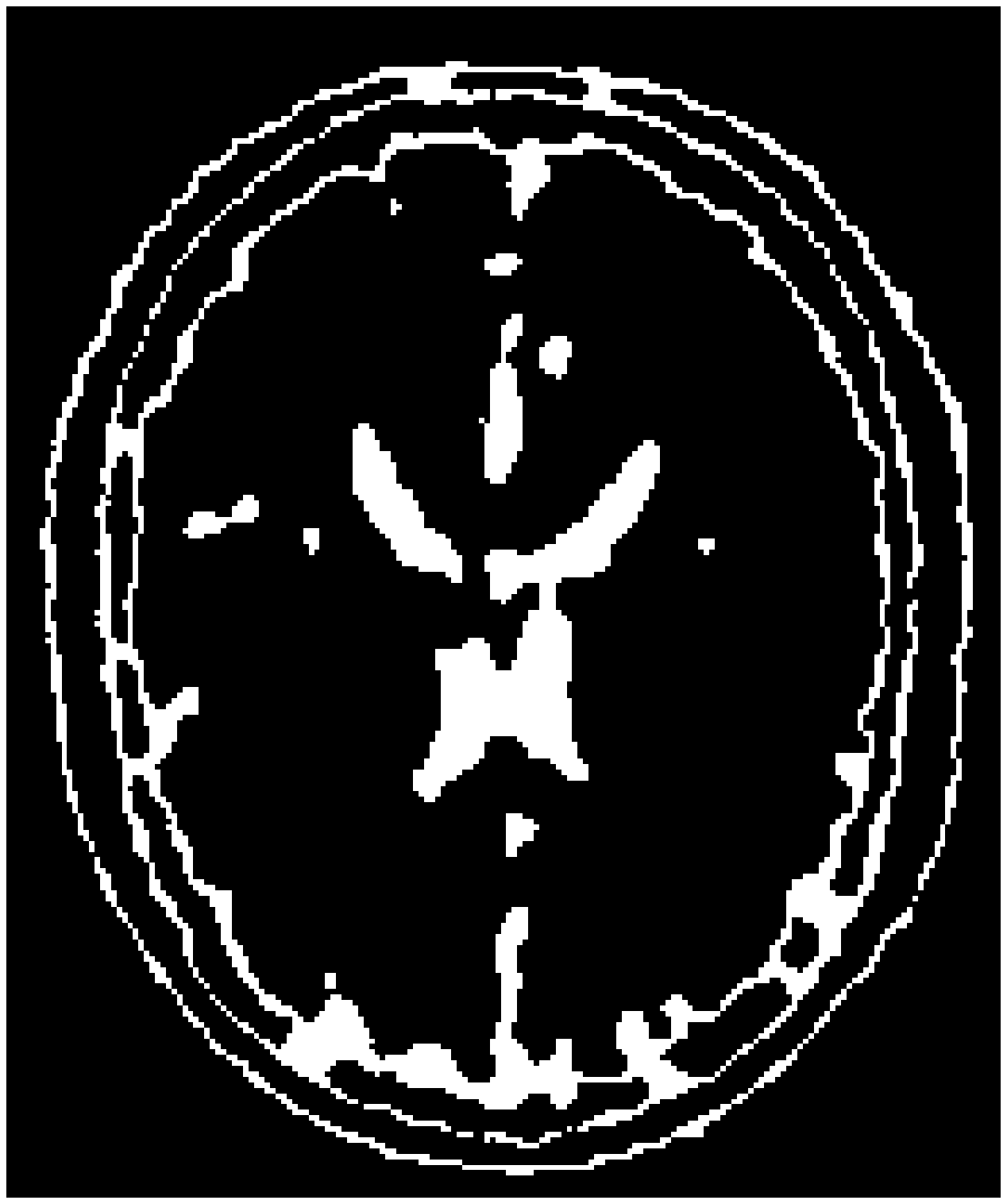}
{(b2) HTVUM}
\end{minipage}
\begin{minipage}[htbp]{0.215\linewidth}
\centering
\includegraphics[width=1.05in]{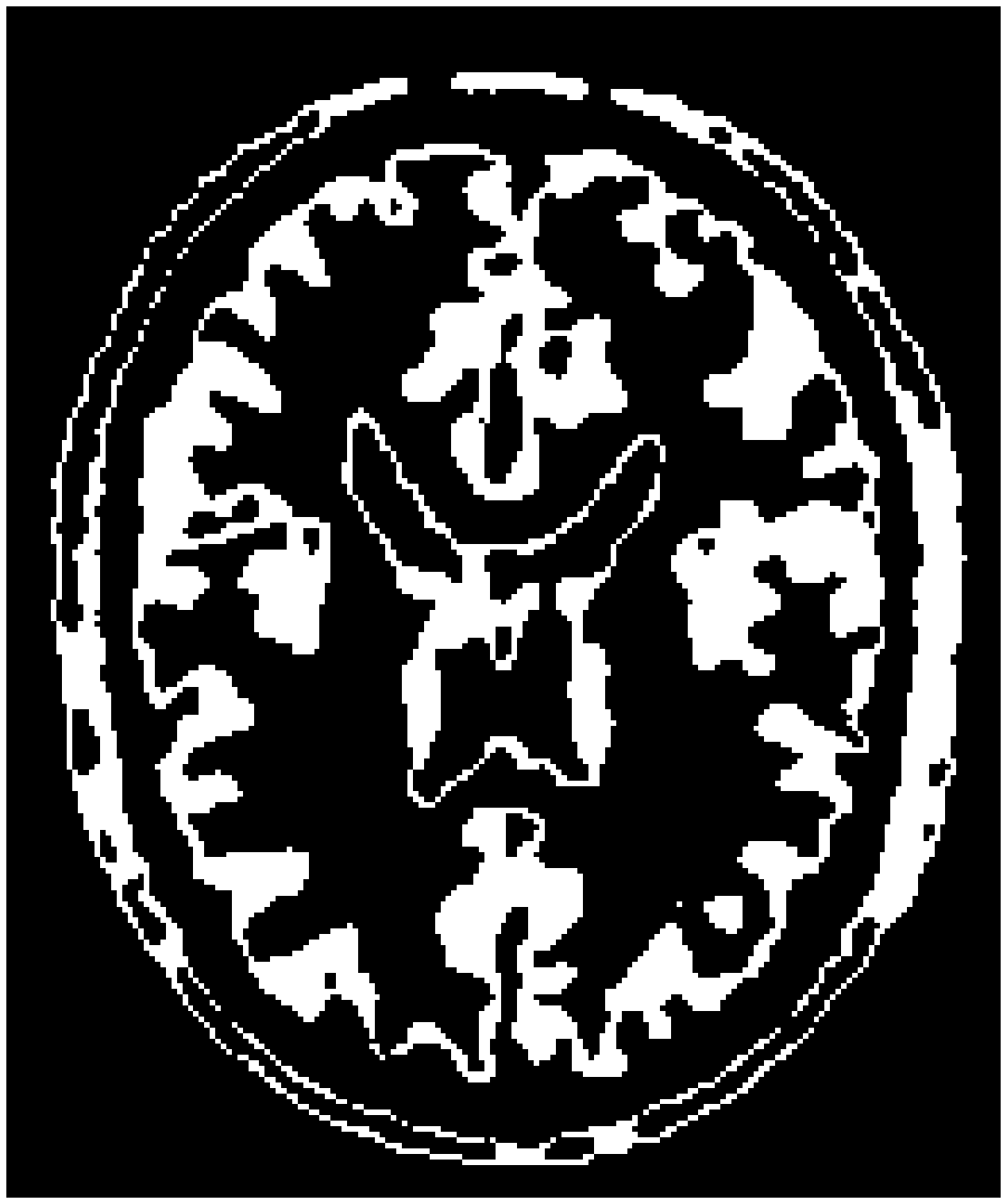}
{(cb3) HTVUM}
\end{minipage}
\begin{minipage}[htbp]{0.215\linewidth}
\centering
\includegraphics[width=1.05in]{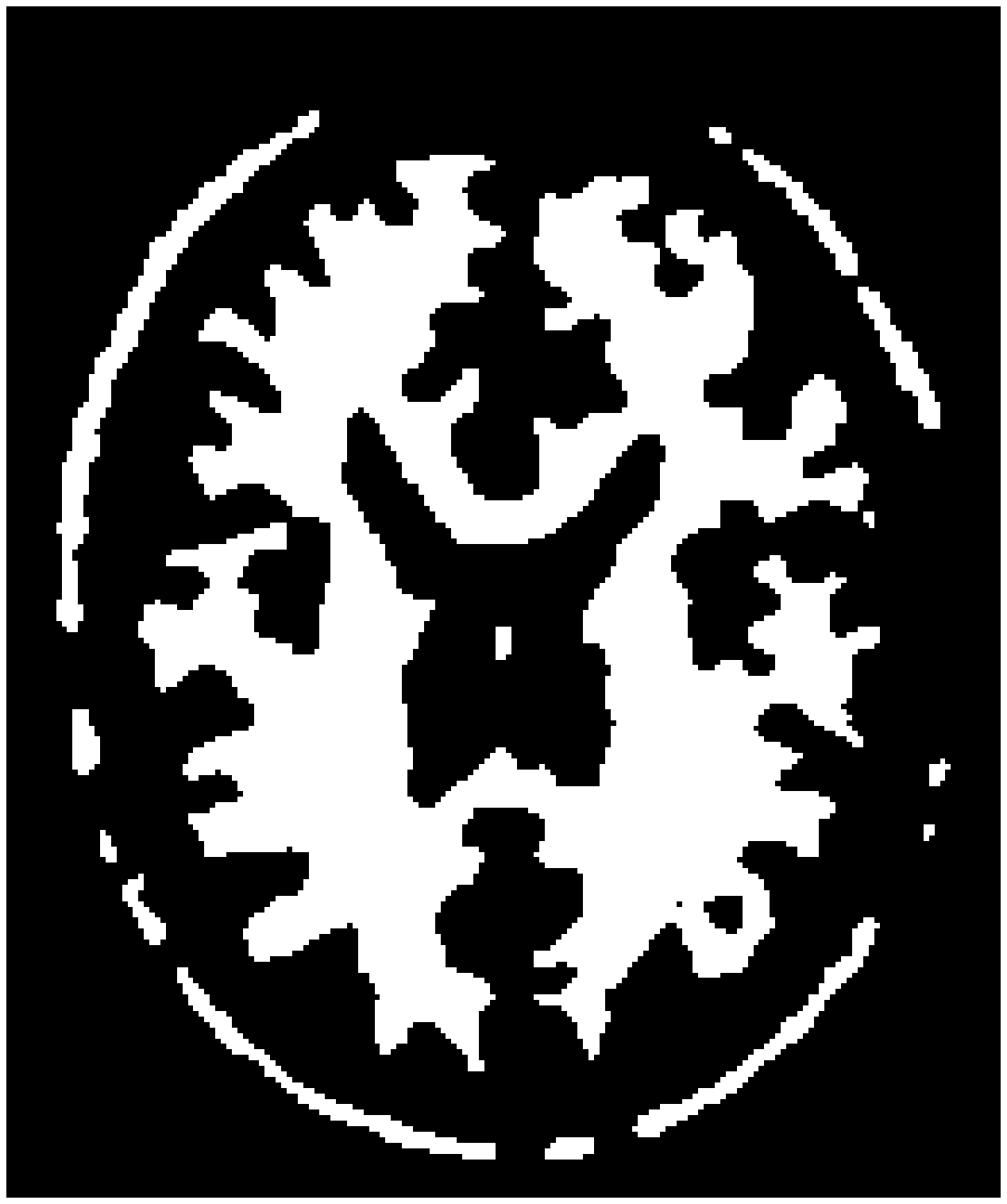}
{(b4) HTVWM}
\end{minipage}\\
\begin{minipage}[htbp]{0.215\linewidth}
\centering
\includegraphics[width=1.05in]{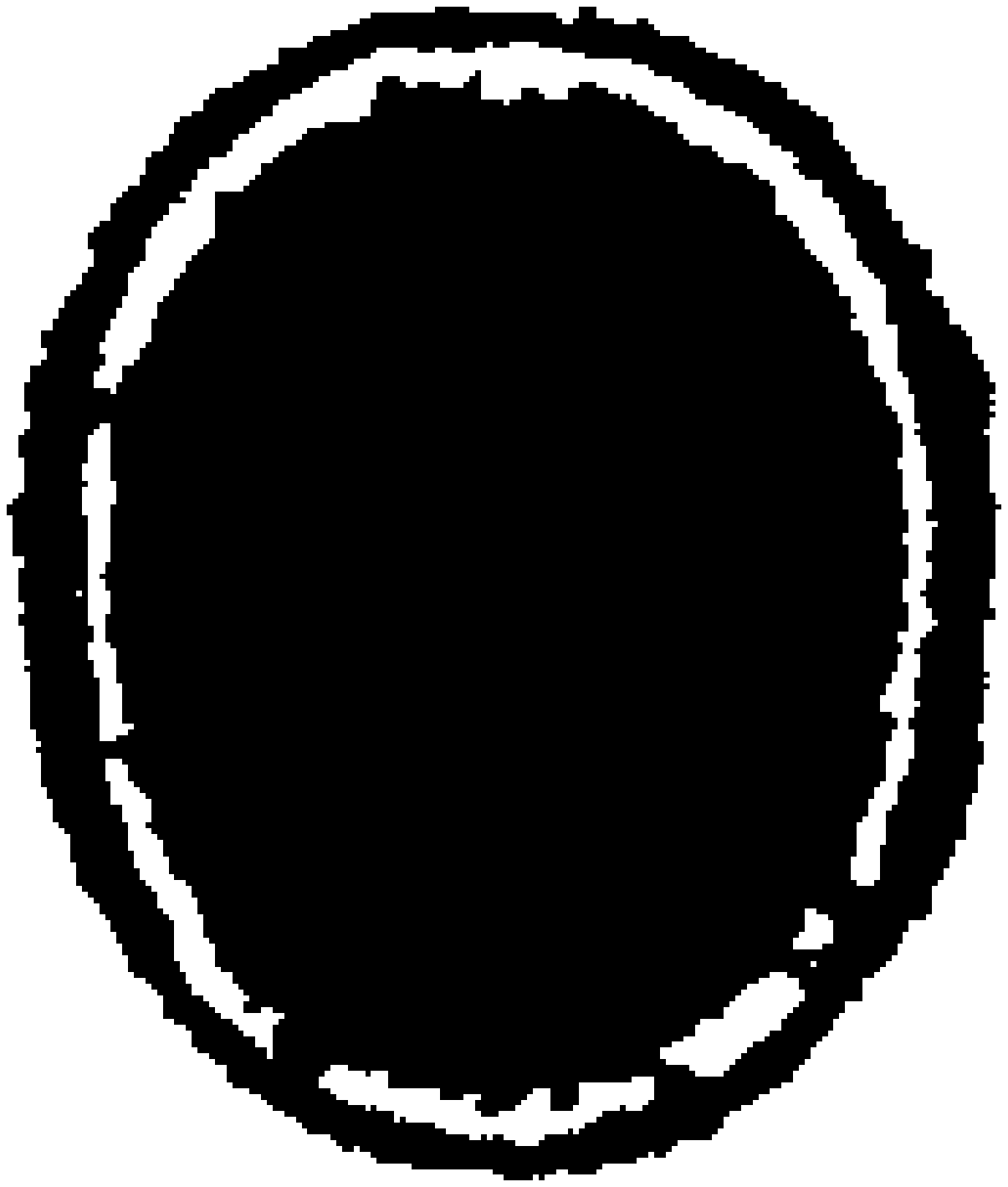}
{(c1) HTVWM}
\end{minipage}
\begin{minipage}[htbp]{0.215\linewidth}
\centering
\includegraphics[width=1.05in]{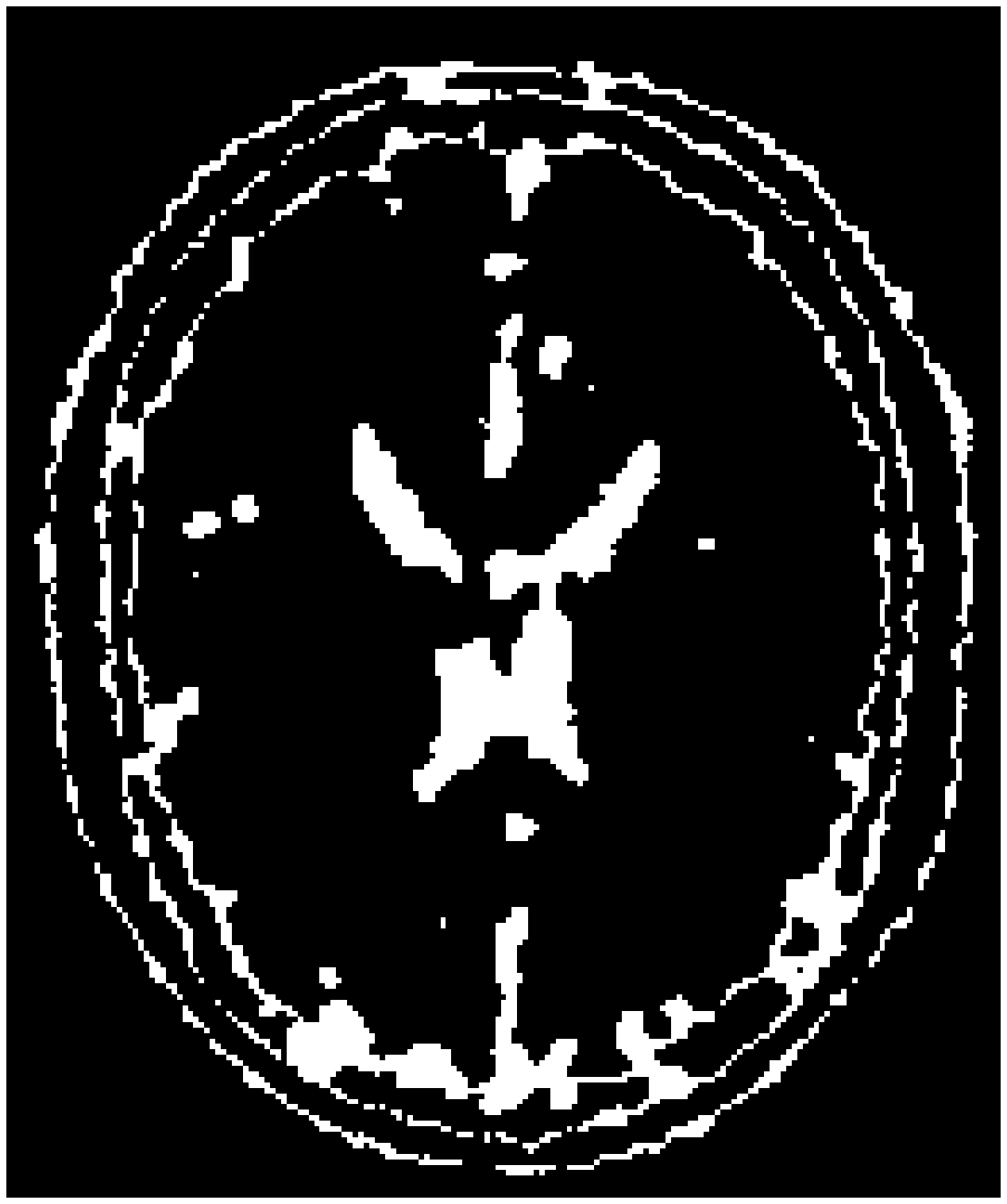}
{(c2) HTVWM}
\end{minipage}
\begin{minipage}[htbp]{0.215\linewidth}
\centering
\includegraphics[width=1.05in]{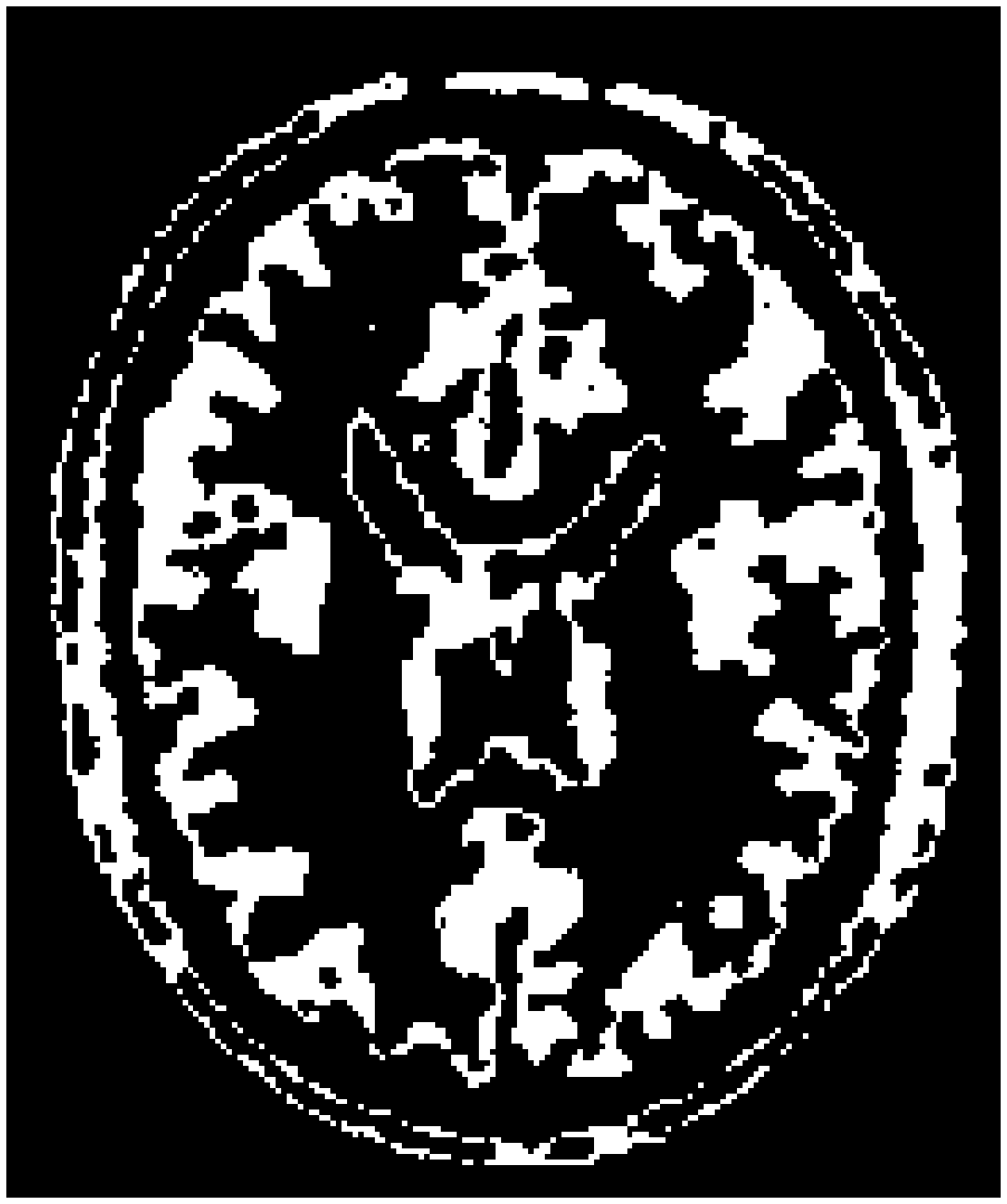}
{(c3) HTVWM}
\end{minipage}
\begin{minipage}[htbp]{0.215\linewidth}
\centering
\includegraphics[width=1.05in]{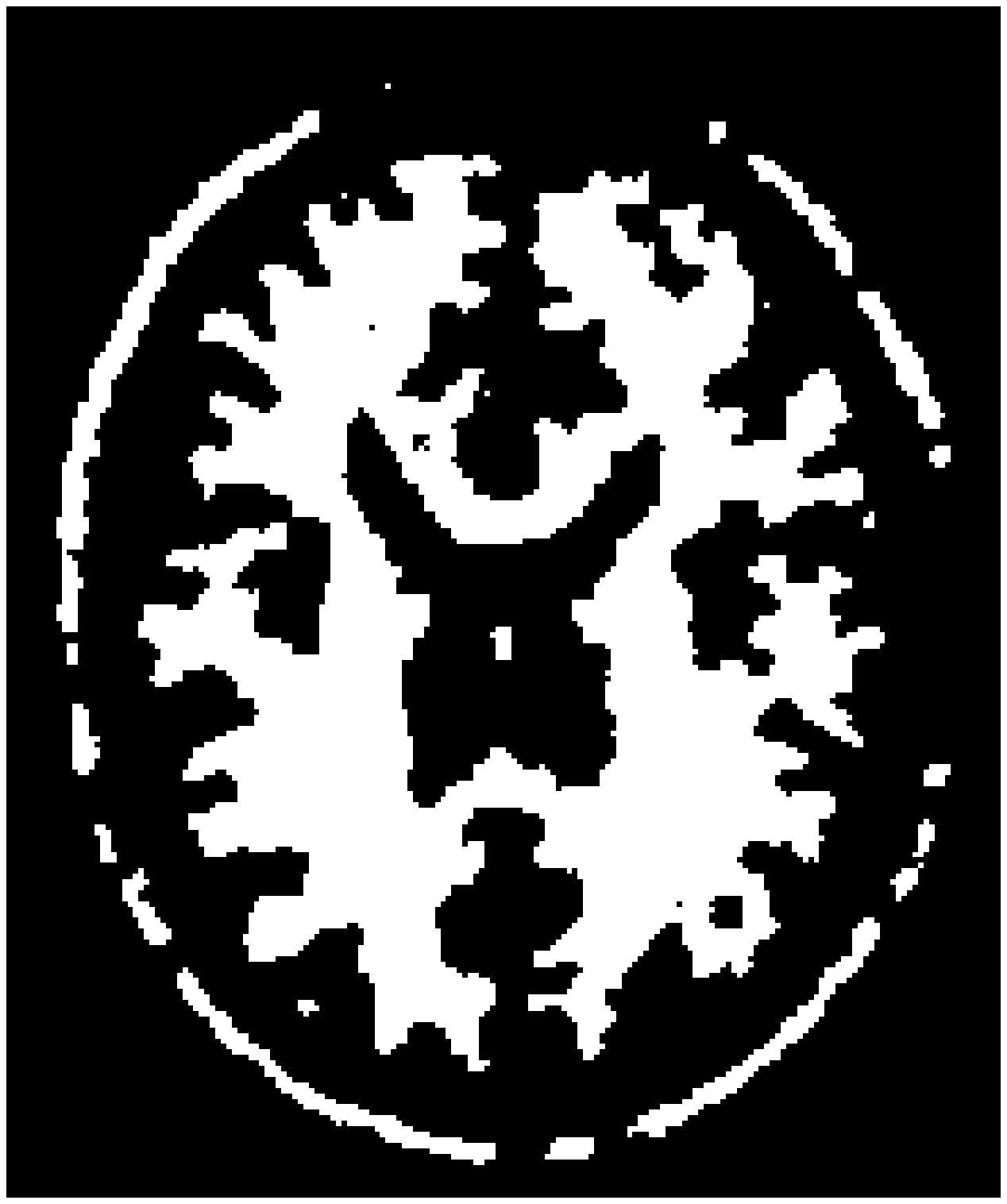}
{(c4) HTVWM}
\end{minipage}
\caption{\label{fig47}Comparisons of the four-phase segmentation results generated by the TSMSM, HTVUM, and HTVWM. Column: First: Background; Second: Tumor and bone; Third: Gray Matter; Forth: White matter. Parameters-TSMSM: $\lambda=51$ and $\gamma=0.03$; HTVUM: $\lambda=44$ and $\gamma=0.01$; HTVWM: $\lambda=54$ and $\gamma=0.008$.  }
\end{figure}

\begin{figure}[h!]
\centering
\begin{minipage}[htbp]{0.215\linewidth}
\centering
\includegraphics[width=1.05in]{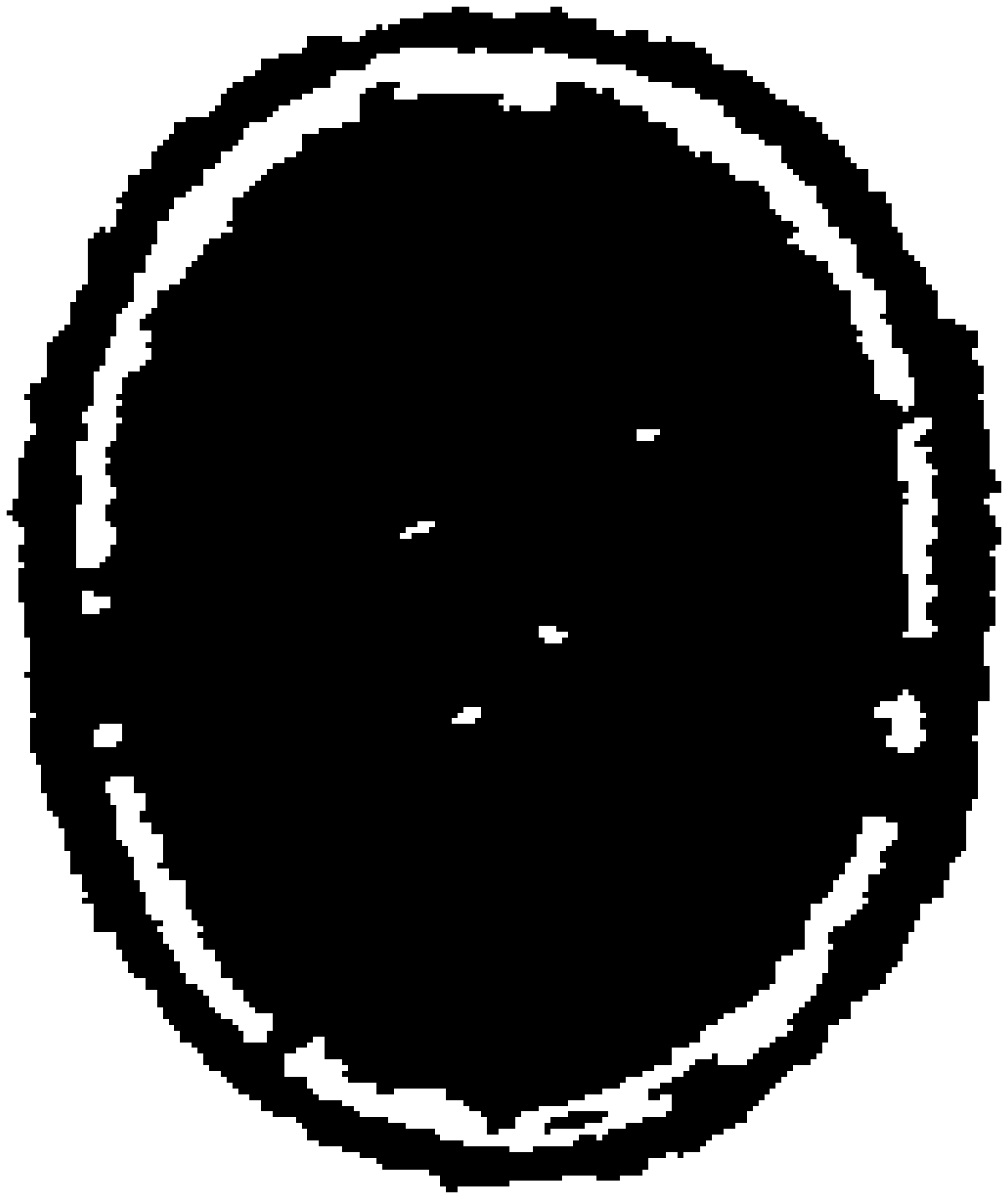}
{(a1) TSMSM}
\end{minipage}
\begin{minipage}[htbp]{0.215\linewidth}
\centering
\includegraphics[width=1.05in]{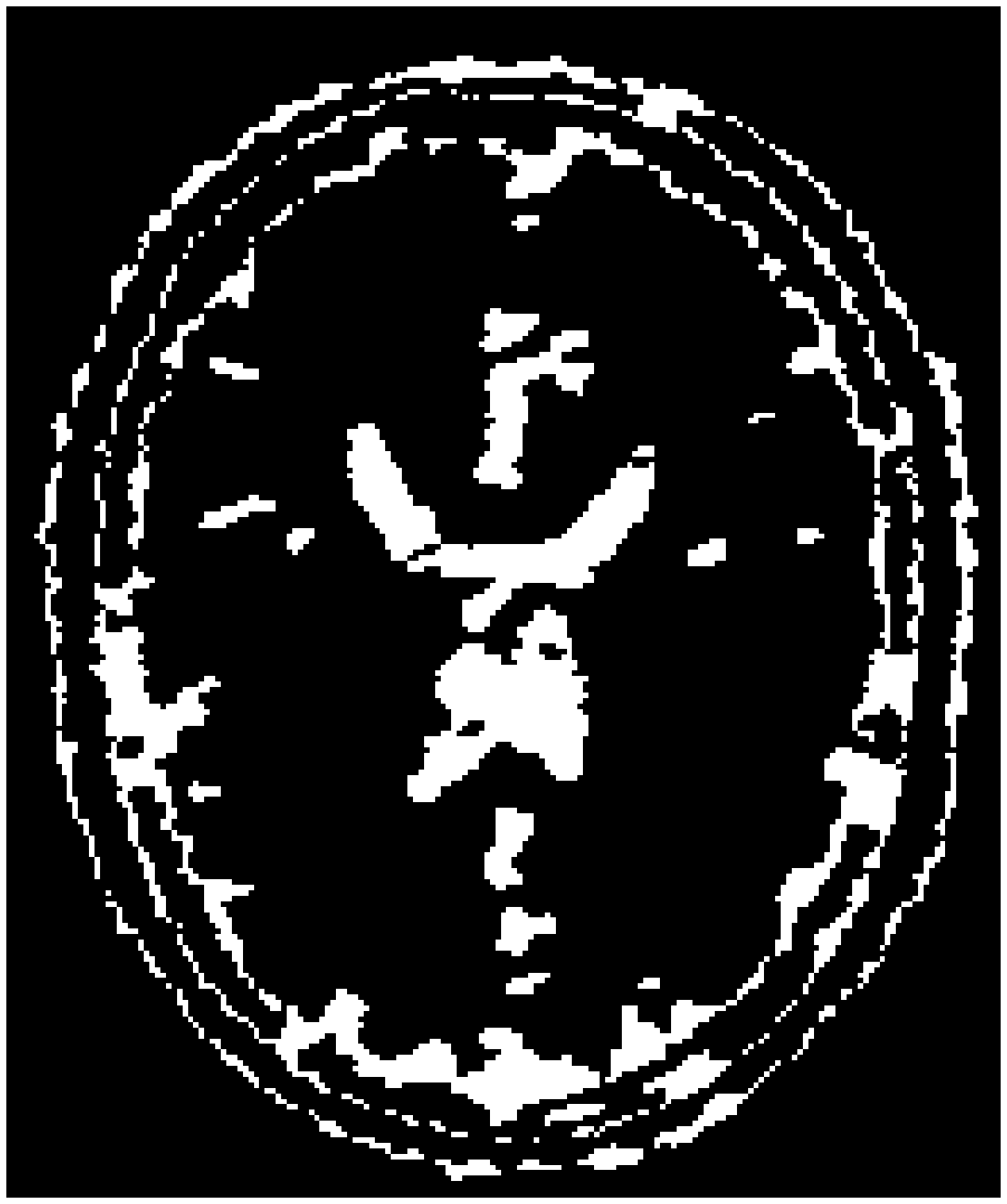}
{(a2) TSMSM}
\end{minipage}
\begin{minipage}[htbp]{0.215\linewidth}
\centering
\includegraphics[width=1.05in]{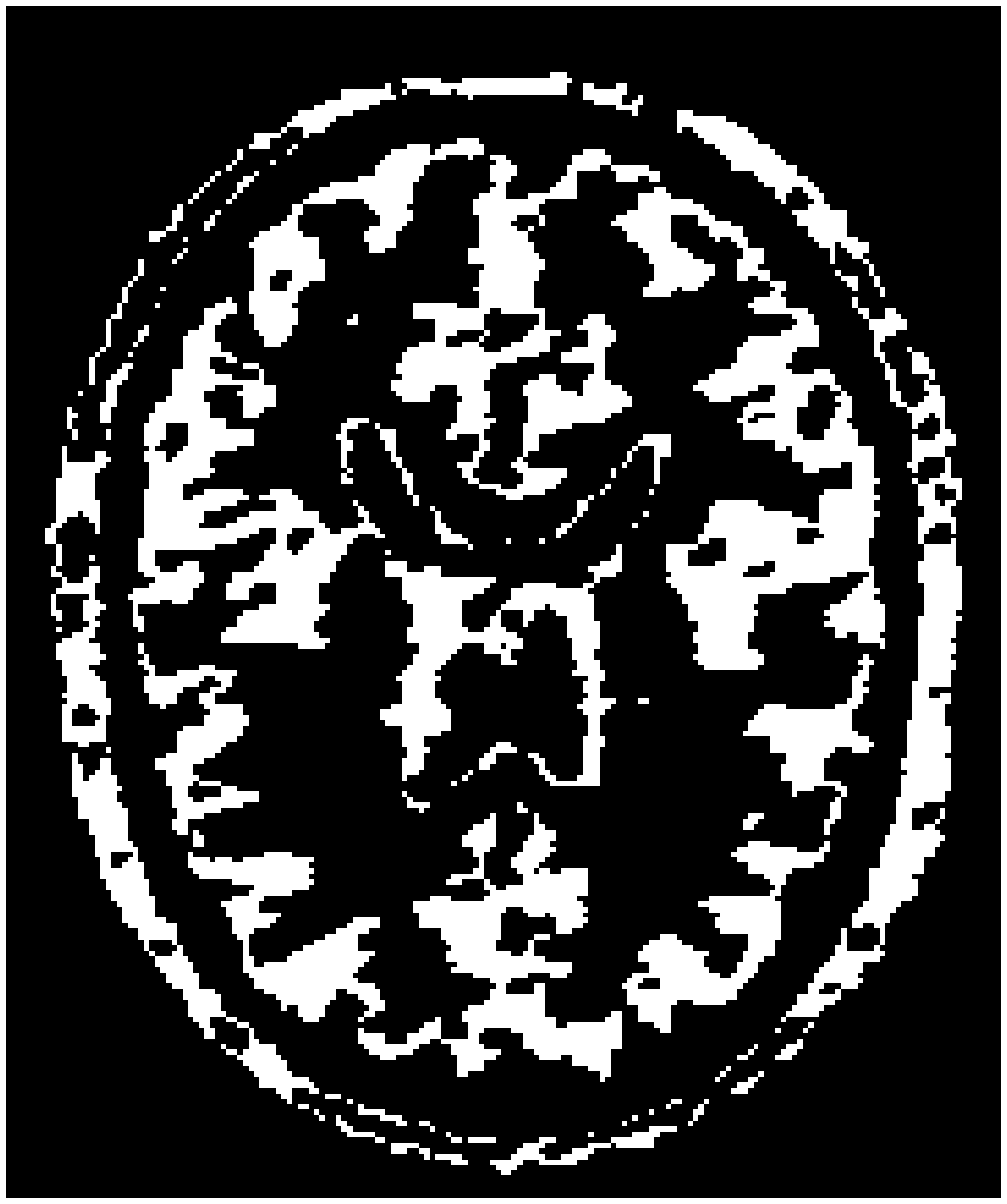}
{(a3) TSMSM}
\end{minipage}
\begin{minipage}[htbp]{0.215\linewidth}
\centering
\includegraphics[width=1.05in]{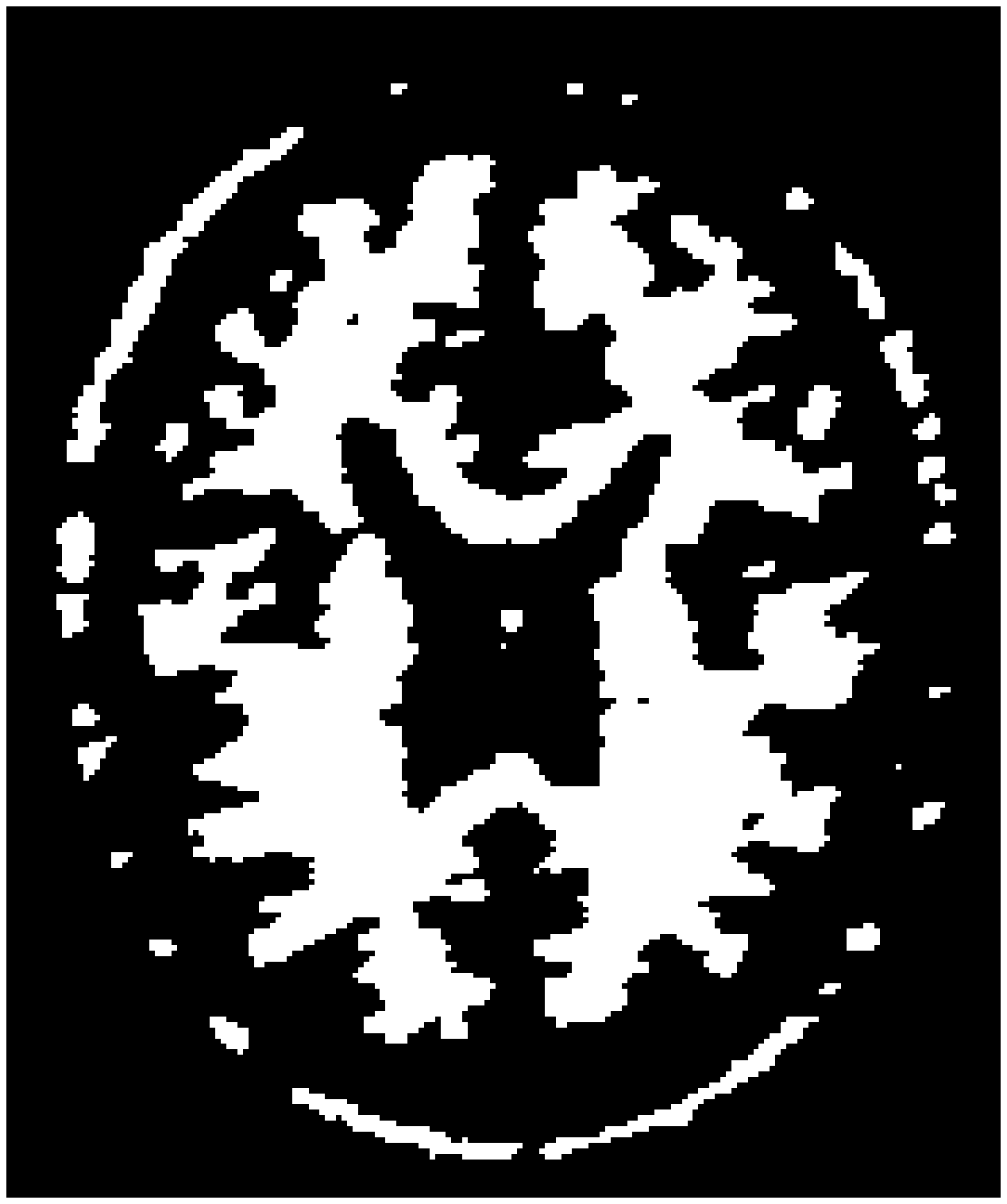}
{(a4) TSMSM}
\end{minipage}\\
\begin{minipage}[htbp]{0.215\linewidth}
\centering
\includegraphics[width=1.05in]{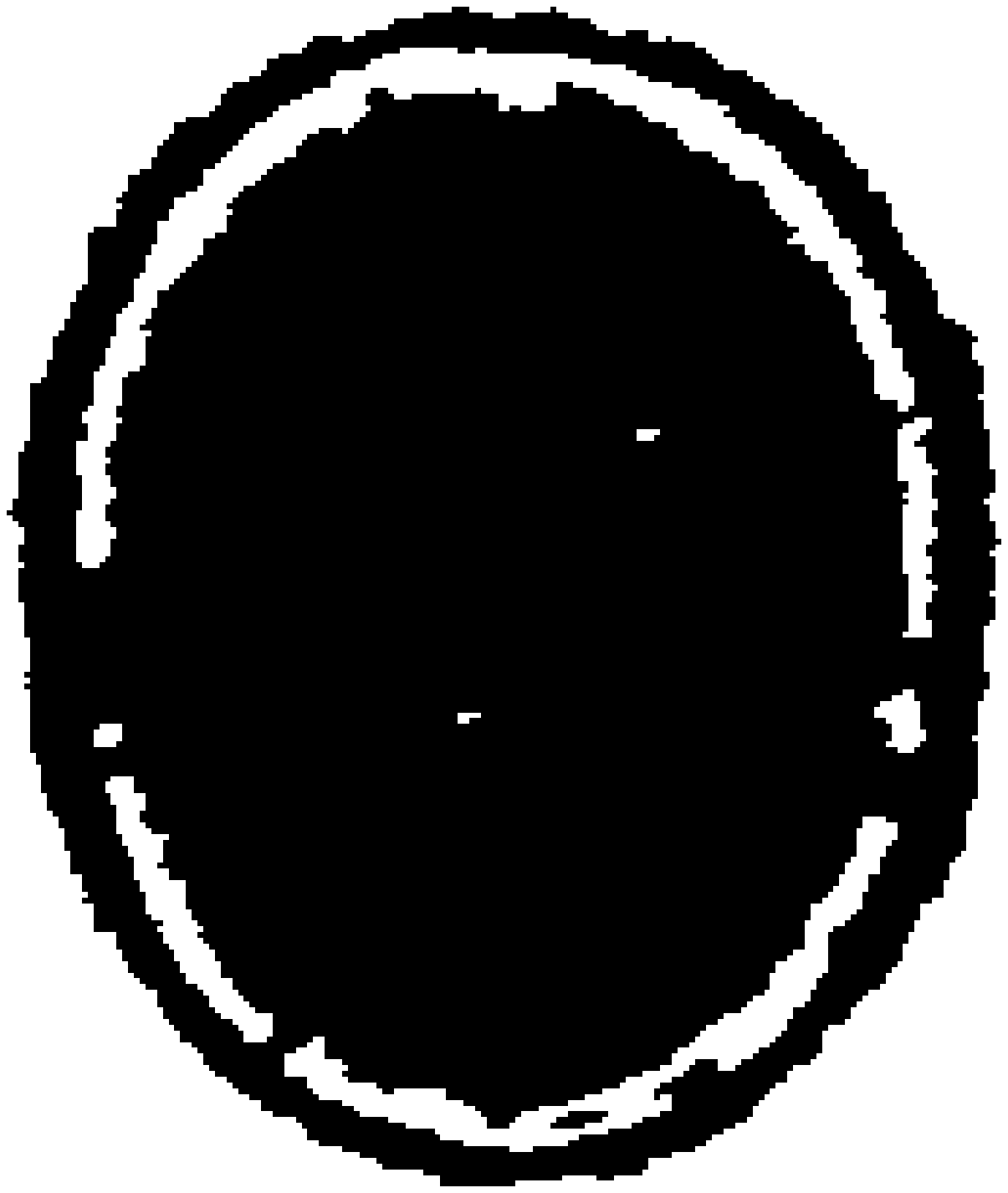}
{(b1) HTVUM}
\end{minipage}
\begin{minipage}[htbp]{0.215\linewidth}
\centering
\includegraphics[width=1.05in]{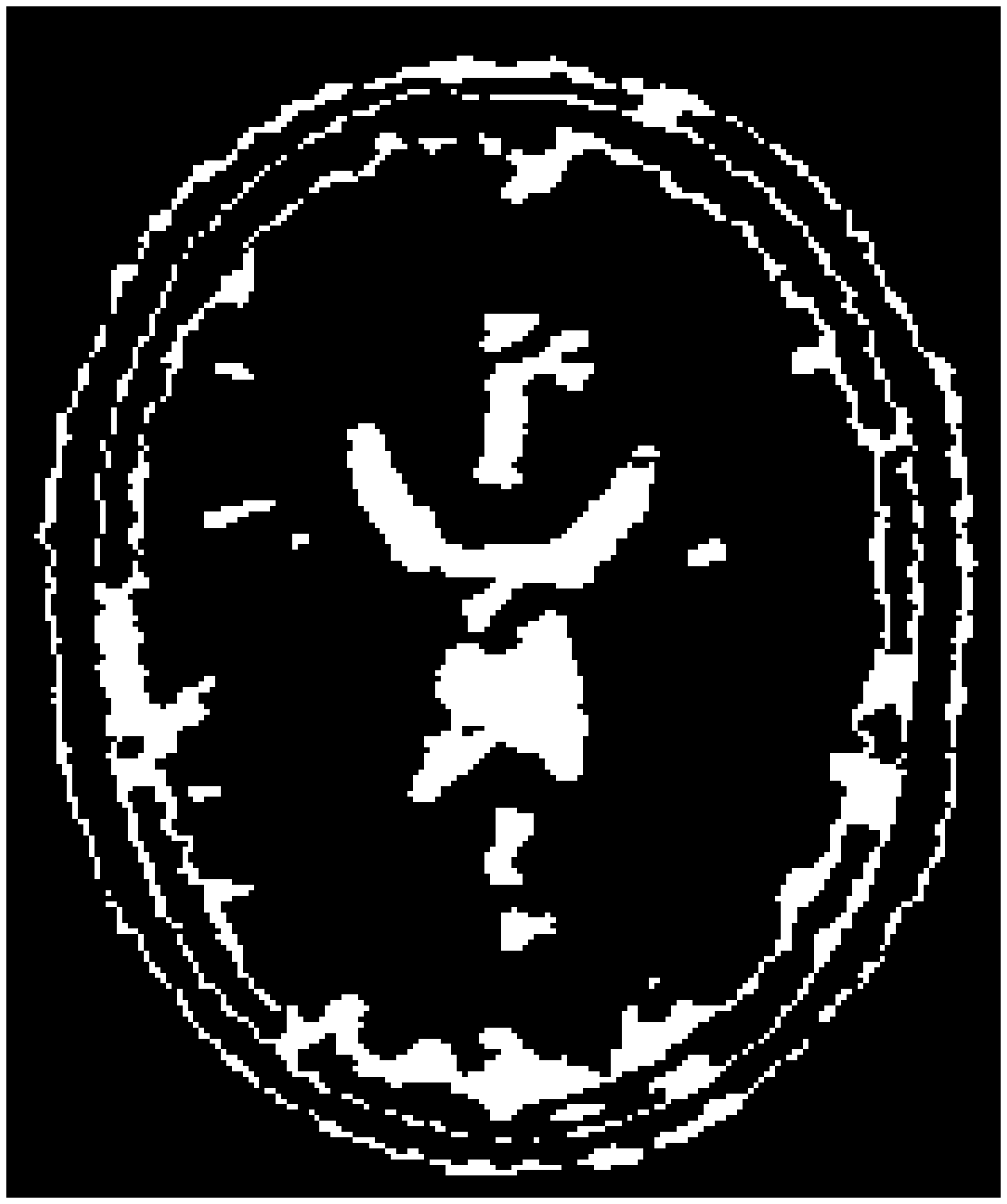}
{(b2) HTVUM}
\end{minipage}
\begin{minipage}[htbp]{0.215\linewidth}
\centering
\includegraphics[width=1.05in]{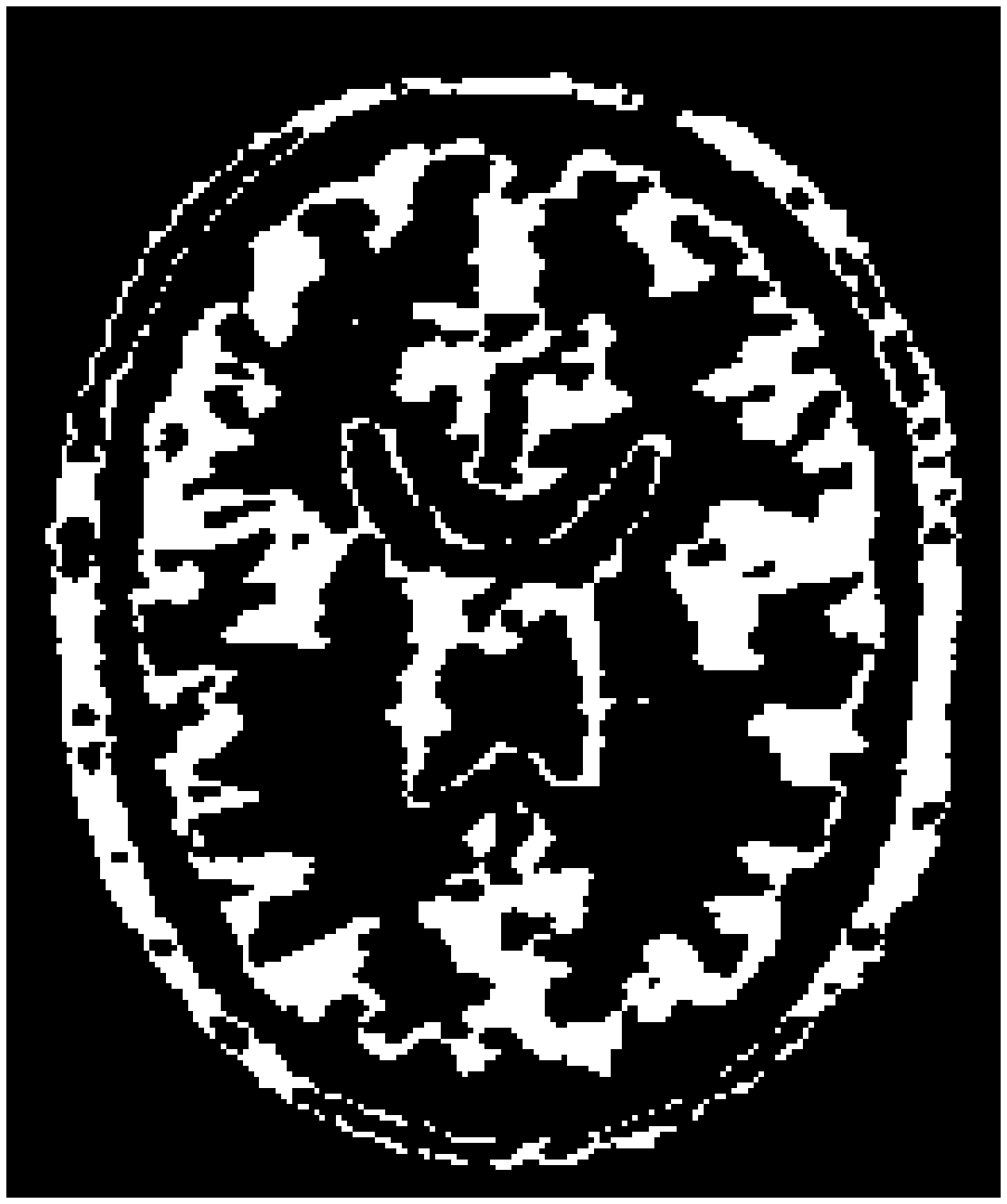}
{(b3) HTVUM}
\end{minipage}
\begin{minipage}[htbp]{0.215\linewidth}
\centering
\includegraphics[width=1.05in]{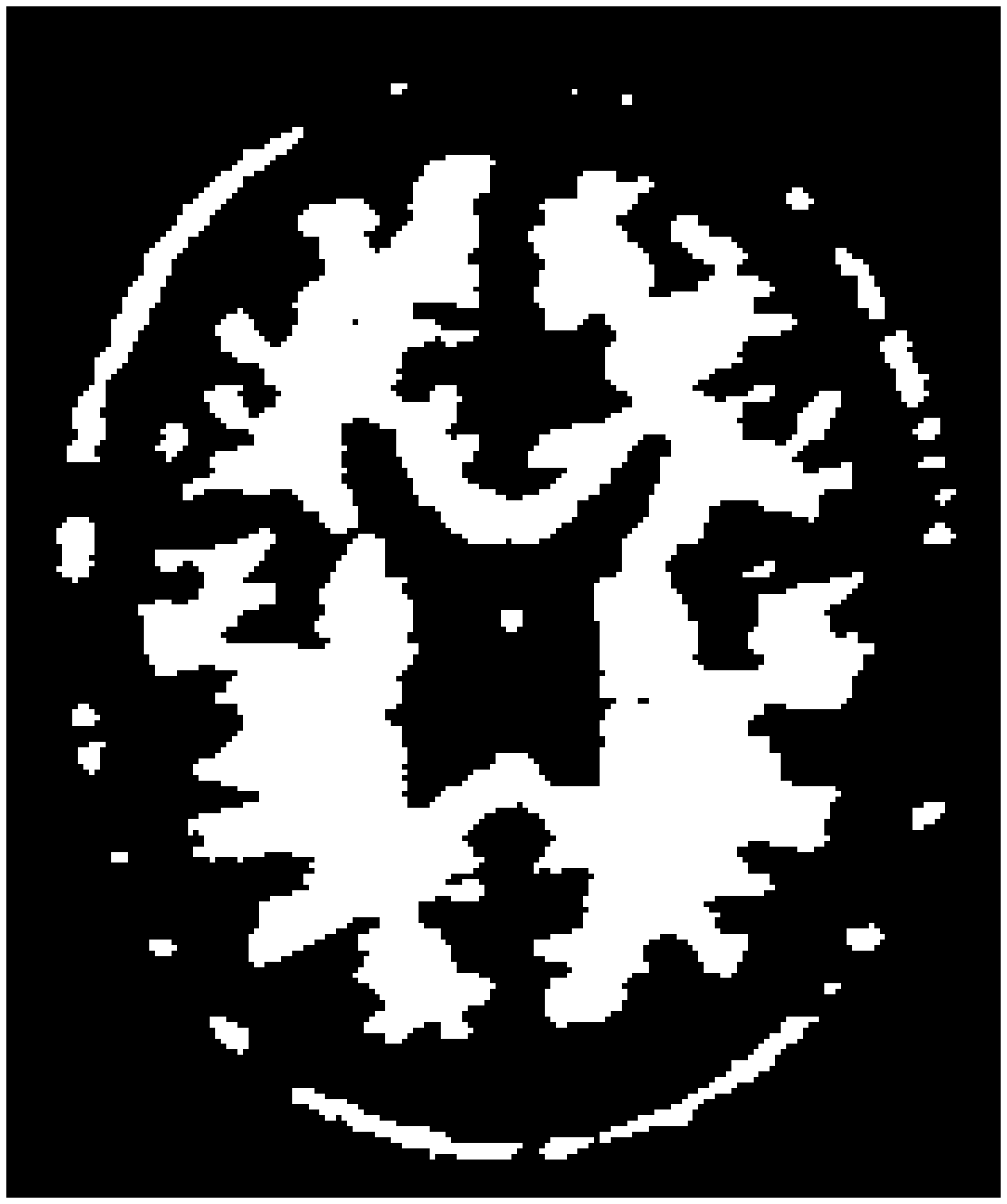}
{(b4) HTVUM}
\end{minipage}\\
\begin{minipage}[htbp]{0.215\linewidth}
\centering
\includegraphics[width=1.05in]{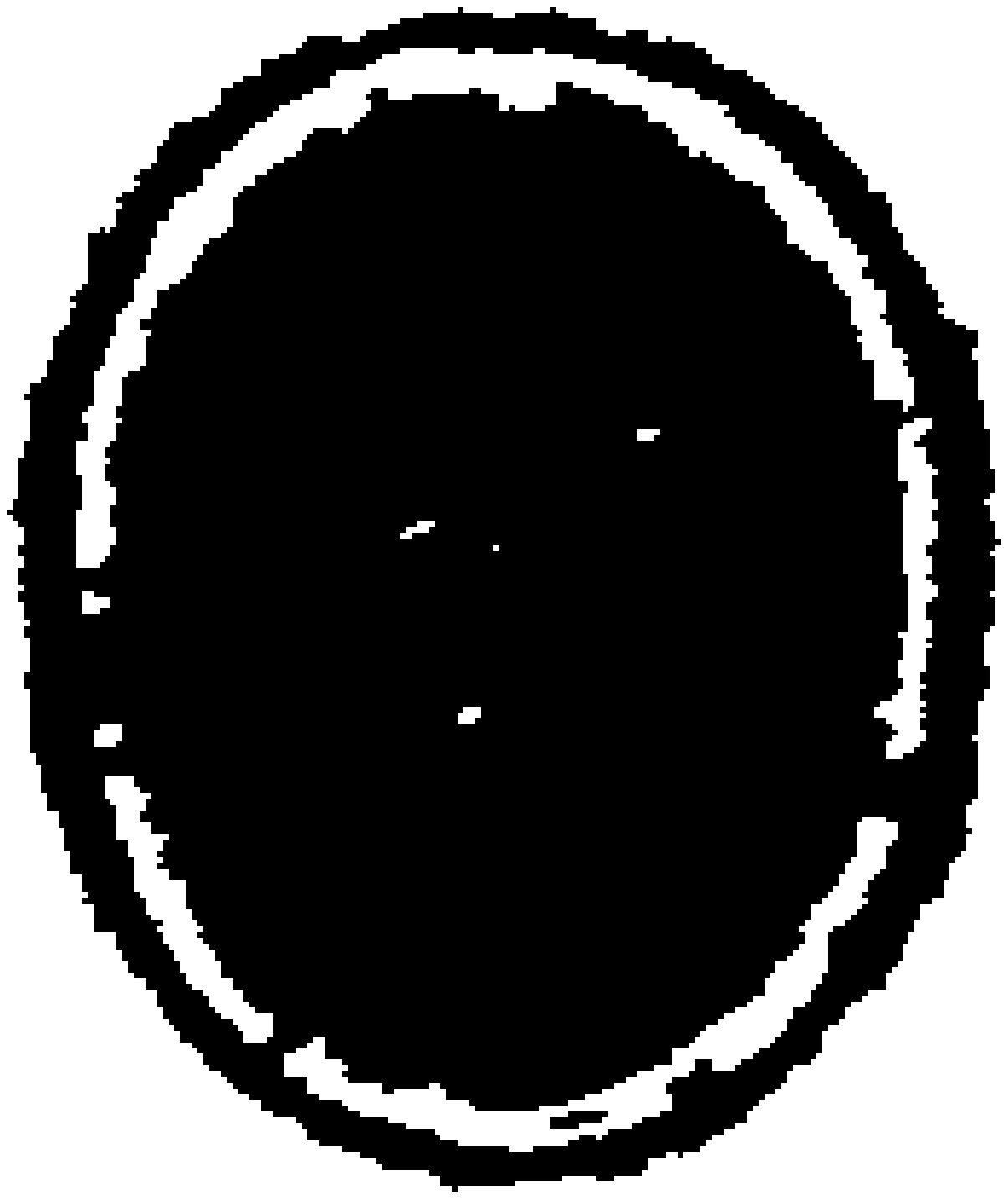}
{(c1) HTVWM}
\end{minipage}
\begin{minipage}[htbp]{0.215\linewidth}
\centering
\includegraphics[width=1.05in]{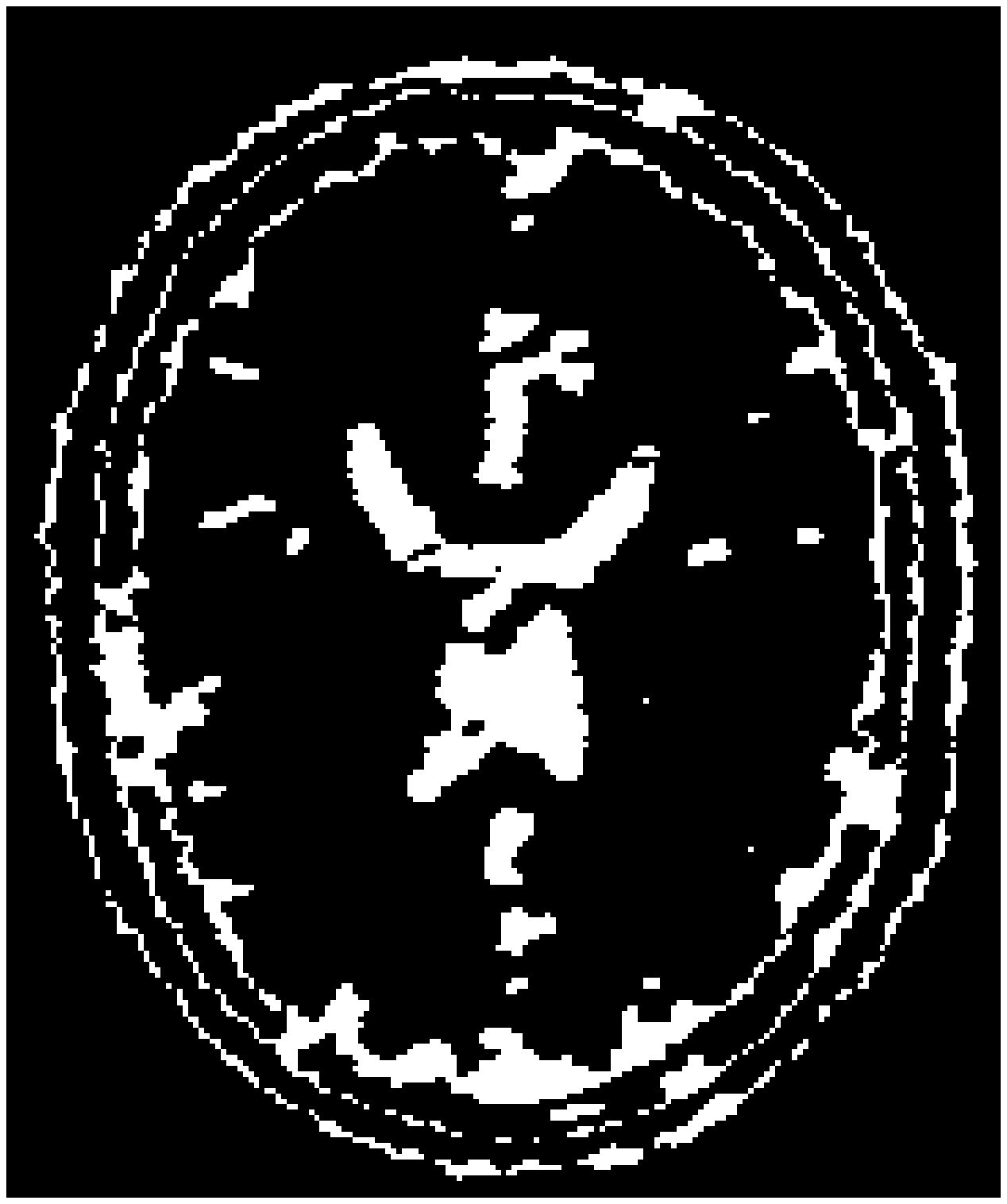}
{(c2) HTVWM}
\end{minipage}
\begin{minipage}[htbp]{0.215\linewidth}
\centering
\includegraphics[width=1.05in]{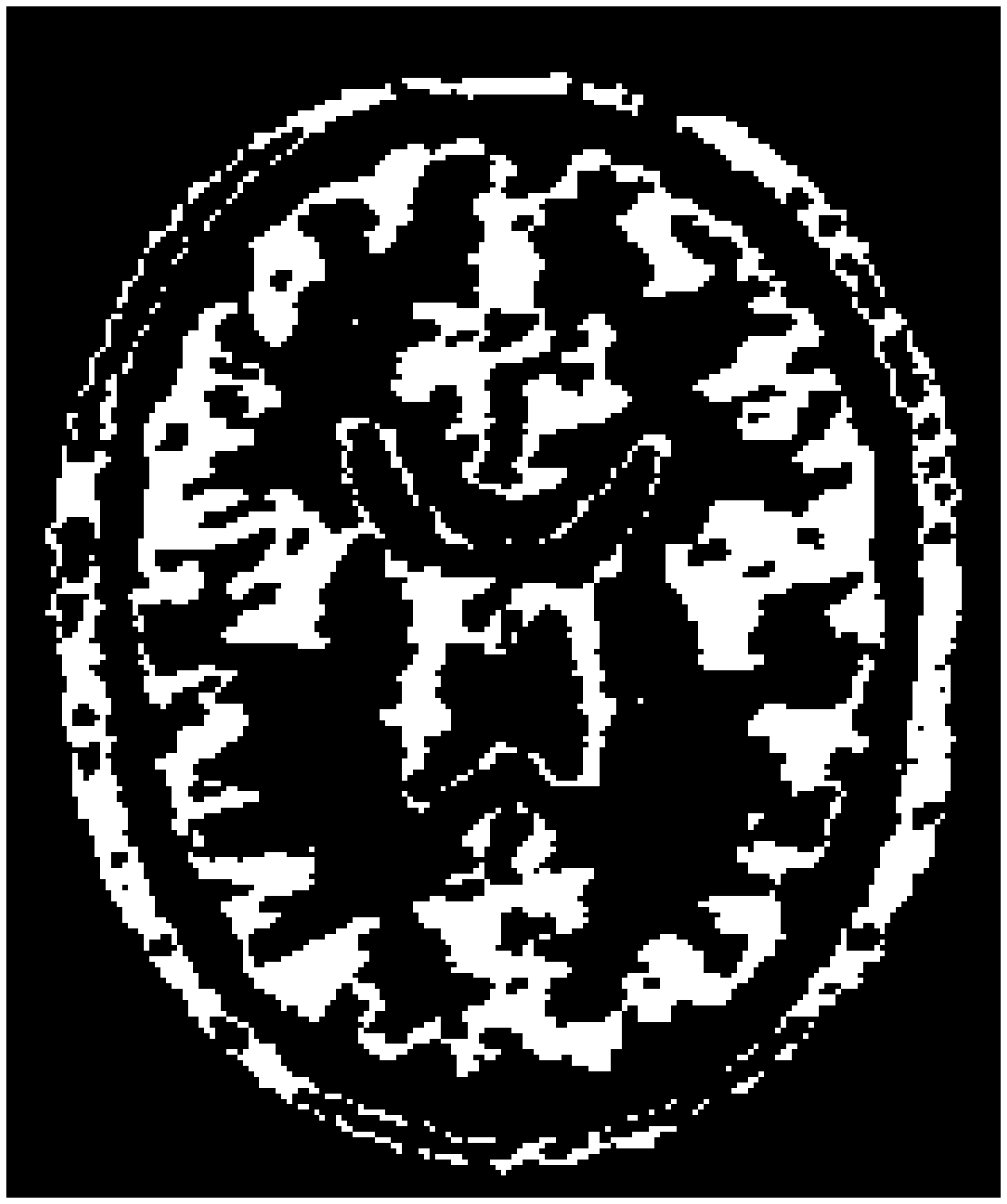}
{(c3) HTVWM}
\end{minipage}
\begin{minipage}[htbp]{0.215\linewidth}
\centering
\includegraphics[width=1.05in]{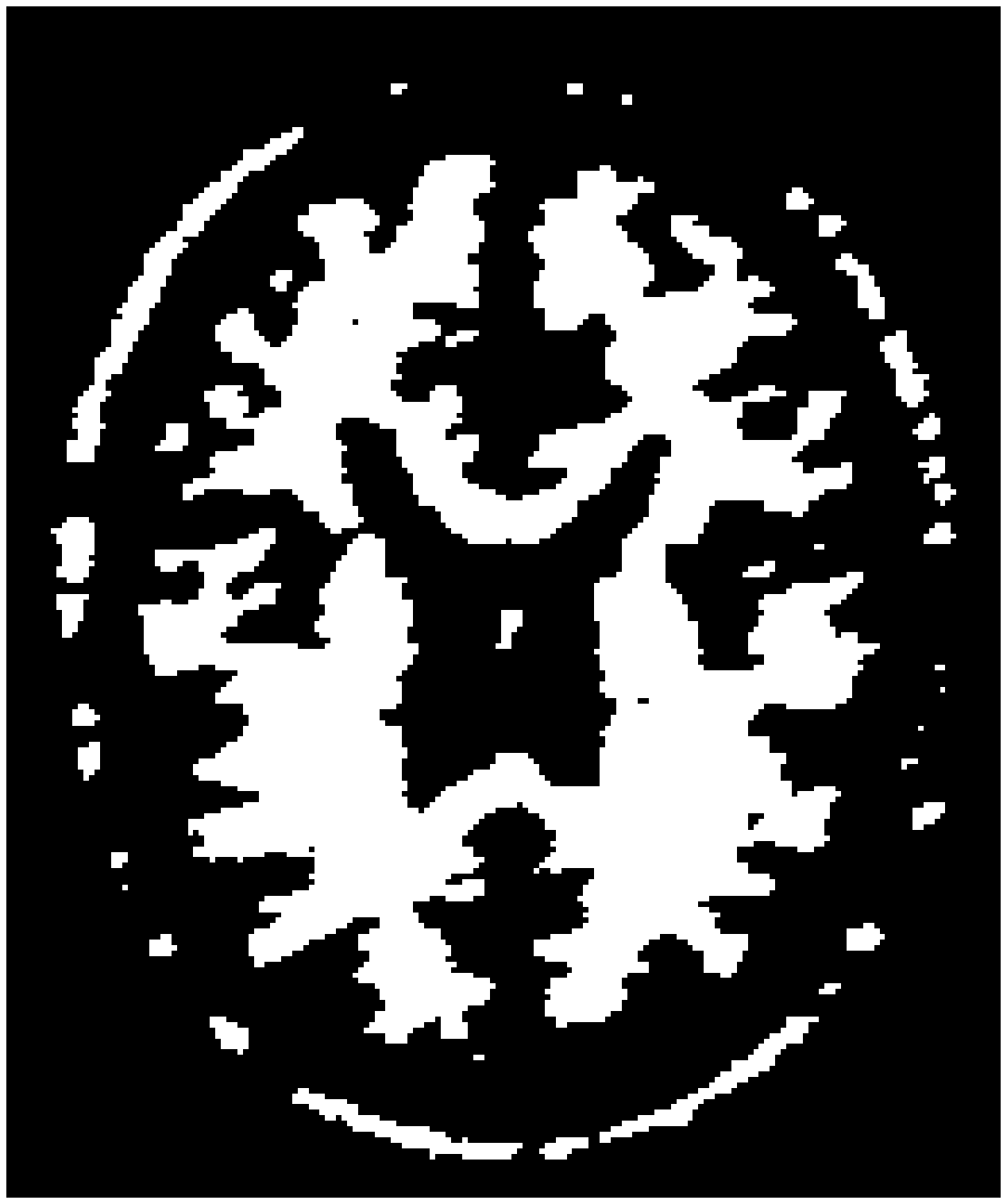}
{(c4) HTVWM}
\end{minipage}
\caption{\label{fig48}Comparisons of the four-phase segmentation results generated by the TSMSM, HTVUM, and HTVWM. Column: First: Background;  Second: Tumor and bone; Third: Gray Matter; Forth: White matter. Parameters-TSMSM: $\lambda=29$ and $\gamma=0.4$; HTVUM: $\lambda=25$ and $\gamma=0.04$;
HTVWM: $\lambda=26$ and $\gamma=0.02$. }
\end{figure}

\end{example}

\section{Conclusions}
In this paper, we proposed a two step strategy to segment the contaminated image by combining the hybrid total variation model (\ref{31}) and K-means clustering method. The  model (\ref{31}) convexly combined the total variation functional and the high order total variation functional with a weighted balance to obtain efficient restored image. In order to solve this nonsmoothing model,  we used the alternating  split Bregman method and also analyzed its convergence. Once the restored image is obtained, we clustered it into the expected phase by using the K-means clustering method. Numerical experiments illustrated the effectiveness of our proposed method compared with the methods in \cite{5,18,26}.  In the experiments we found that a better restoration image did not imply a better segmentation result. So choosing a suitable restoration measure is very important and is also  a part of the future work.

\section*{Acknowledgments} We would like to thank Dr. Huibin Chang for his suggestions on the numerical experiments and anonymous referees for their helpful comments and suggestions for improving this paper.


\medskip
Received May 2014; revised July 2015, accepted April 2016 by Inverse problem and Imaging.
\medskip

\end{document}